\documentclass[letterpaper, 10pt,reqno]{amsart}

\usepackage[usenames,dvipsnames]{xcolor}

\usepackage[portrait,margin=3cm]{geometry}
\usepackage{mathrsfs}

\usepackage{amssymb,amsthm,amsfonts,amsbsy,amstext,latexsym}
\usepackage[reqno]{amsmath}

\usepackage{color}
\usepackage{graphicx}
\usepackage{bbm}
\usepackage{comment}
\usepackage{a4wide,graphicx, listings,lscape, epstopdf, units, textcomp, fancyhdr, url, soul, color}
\usepackage[textsize=small]{todonotes}
\usepackage{enumitem}

\usepackage{mathtools}
\mathtoolsset{showonlyrefs}

\definecolor{MyDarkblue}{rgb}{0,0.08,0.50}
\definecolor{Brickred}{rgb}{0.65,0.08,0}

\usepackage{hyperref,upref}
\hypersetup{
	colorlinks=true,       
	linkcolor=MyDarkblue,          
	citecolor=Brickred,        
	filecolor=red,      
	urlcolor=cyan           
}

\newtheorem*{theorem*}{Theorem}
\newtheorem{theorem}{Theorem}[section]
\newtheorem{lemma}[theorem]{Lemma}

\newtheorem{proposition}[theorem]{Proposition}

\newtheorem{conjecture}[theorem]{Conjecture}

\theoremstyle{definition}
\newtheorem{definition}[theorem]{Definition}
\newtheorem{assumption}[theorem]{Assumption}
\newtheorem{remark}[theorem]{Remark}

\renewcommand{\P}{\mathbb{P}}

\newcommand{\Pv}{\mathbb{P}}

\newcommand{\eps}{\varepsilon}

\newcommand{\Var}{{\rm Var}}

\newcommand{\e}{{\mathrm e}}

\newcommand{\R}{\mathbb{R}}
\newcommand{\N}{\mathbb{N}}
\newcommand{\Z}{\mathbb{Z}}

\newcommand*{\wt}{\widetilde}

\newcommand*{\be}{\begin{equation}}
\newcommand*{\ee}{\end{equation}}
\newcommand*{\ba}{\begin{aligned}}
	\newcommand*{\ea}{\end{aligned}}
\newcommand*{\barr}{\begin{array}{c}}
	\newcommand*{\earr}{\end{array}}
\def \toinp    {\buildrel {\Pv}\over{\longrightarrow}}
\def \toindis  {\buildrel {d}\over{\longrightarrow}}
\def \toas     {\buildrel {a.s.}\over{\longrightarrow}}

\newcommand*{\ind}{\mathbbm{1}}
\makeatletter
\def\namedlabel#1#2{\begingroup
	#2%
	\def\@currentlabel{#2}%
	\phantomsection\label{#1}\endgroup
}

\makeatother

\newcommand{\bes}{\begin{equation*}}
\newcommand{\ees}{\end{equation*}}

\renewcommand{\P}[1]{\mathbb{P}\!\left(#1\right)}
\newcommand{\E}[1]{\mathbb{E}\left[#1\right]}

\newcommand{\F}{W}
\newcommand{\G}{\mathcal{G}}
\newcommand{\Zm}{\mathcal{Z}}
\renewcommand{\N}{\mathbb{N}}

\newcommand{\dzn}[1]{\Delta \Zm_n(#1)}
\newcommand{\Ef}[2]{\mathbb{E}_\F#1[#2#1]}
\newcommand{\Pf}[1]{\mathbb{P}_\F\!\left(#1\right)}
\newcommand{\I}{\mathbb{I}}

\renewcommand{\d}{\mathrm{d}}

\newcommand{\zni}{\Zm_n(i)}

\newcommand{\inn}{i\in[n]}

\numberwithin{equation}{section}

\newcommand{\cb}{\color{black}}

\renewcommand{\e}{\mathrm{e}}

\makeatletter
\newcommand{\leqnomode}{\tagsleft@true\let\veqno\@@leqno}
\newcommand{\reqnomode}{\tagsleft@false\let\veqno\@@eqno}
\makeatother
\newlength{\tagmarginsep} 
\begin{document}

	\title{The maximal degree in random recursive graphs with random weights}

	\date{\today}
	
	\author[Lodewijks]{Bas Lodewijks}
	\address{Department of Mathematics, University of Augsburg, Universit\"atsstra\textup{\ss}e 2, 86159, Germany}
	\email{b.lodewijks@uni-a.de}
	
	\author[Ortgiese]{Marcel Ortgiese}
	\address{Department of Mathematical Sciences,
		University of Bath,
		Claverton Down,
		Bath,
		BA2 7AY,
		United Kingdom.}
	\email{m.ortgiese@bath.ac.uk}
	
	\begin{abstract}
	We study a generalisation of the random recursive tree (RRT) model and its multigraph counterpart, the uniform directed acyclic graph (DAG). Here, vertices are equipped with a random  vertex-weight representing initial inhomogeneities in the network, so that a new vertex connects to one of the old vertices with a probability that is proportional to their vertex-weight. We first identify the asymptotic degree distribution of a uniformly chosen vertex for a general vertex-weight distribution. For the maximal degree, we distinguish several classes that lead to different behaviour: For bounded vertex-weights we obtain results for the maximal degree that are similar to those observed for RRTs and DAGs. If the vertex-weights have unbounded support, then the maximal degree has to satisfy the right balance between having a high vertex-weight and being born early. 

For vertex-weights in the Fr\'echet maximum domain of attraction the first order behaviour of the maximal degree is random, while for those in the Gumbel maximum domain of attraction the leading order is deterministic. Surprisingly, in the latter case, the second order is random when considering vertices in a compact window in the optimal region, while it becomes deterministic when considering all vertices.
\end{abstract}

\maketitle 

\section{Introduction}

A random recursive graph  with random weights (WRG) is a family of growing random graphs
that combine features of an inhomogeneous random graph, such as e.g.\ the Chung-Lu model,  with random weights with those of a dynamical model that generalizes the classical model of random recursive trees.
In the WRG model, we assign to every vertex a random, independent, non-negative vertex-weight and at each step of the evolution a new vertex is introduced to the graph
and connected to a predecessor with probability  proportional to the predecessor's vertex-weight.

Similar to an inhomogeneous random graph, the random vertex-weight can be thought
of as modelling the naturally inhomogeneous attractiveness of each vertex in the network. 
Unlike for the inhomogeneous random graph, which is a static model, we are also interested
in describing the growth of the network.

If the weights are taken to be deterministic and the same for every vertex, then 
the model reduces to the classical 	
random recursive tree (RTT) or its 
multigraph counterpart known as the uniform directed acyclic graph (DAG or uniform DAG). The latter was introduced by Devroye and Lu in~\cite{DevLu95} and allows for an incoming vertex to connect to $k$ predecessors. The original RRT was first introduced by Na and Rapoport in $1970$~\cite{NaRap70} and has since attracted 
a wealth of interest, uncovering the behaviour of many of its properties, including, among others: the number of leaves, profile of the tree, height of the tree, vertex degrees and the size of sub-trees. \cite{SmyMah95} and the more recent~\cite{Drm09} provide good surveys on the topic. 

Another important reference point for us are preferential attachment models, which are  dynamic models where vertices are also not chosen uniformly, but instead are chosen 
with probability proportional to their degree.
For a review  of the recent literature, see~\cite{Hof16}.

In this paper, the main question is to understand 
the influence of the random weights, which can be thought of as a random environment, 
on the well-known dynamics of  the RRT.
As a first step, we analyse the 
asymptotic degree distribution which we obtain as the limit distribution of a uniformly chosen vertex.
In the random graph literature it is folklore that the tail behaviour of the degree distribution
determines
many of the features of the model, such as the asymptotics of the largest degree or typical distances.

For the RRT is known, see e.g.\ the recent work of Addario-Berry and Eslava~\cite{AddEsl18}, that the degree distribution decays like~$2^{-(k+1)}$. 
In contrast, preferential attachment models are famous for producing power law 
tails. In our case, the random weights produce a variety of tail behaviour that interpolates between these extremes. 
In the case of bounded weights, the tails correspond up to first order to the RRT. 
However, in the unbounded case 
when the weights have stretched exponential tails, which we call the Gumbel case, 
the degree distribution also has stretched exponential tails (although with a different exponent as the vertex-weight distribution).
If the weights have power law tails, which we call the Fr\'echet case, then the degree distribution has the same tails as the weights. 

The second statistics we consider is the asymptotic behaviour of the largest degree.
Understanding the  largest degree can tell us, for example, about the existence of hubs, 
i.e.\ vertices of very large degree that are known to play an important role in the 
connectivity structure of random graphs, see e.g.\ the book of Van der Hofstad for a great overview on random graphs and their connections to real-world networks~\cite{Hof16}.
By comparing to the i.i.d.\ case, we could expect that the tail of the degree distribution 
predicts the asymptotic behaviour of the largest degree.
We confirm this conjecture on the level of the first order growth and show the following 
results for the three regimes discussed above:
for bounded weights the system behaves similarly to a RRT and grows logarithmically, whereas for unbounded weights that are in the domain of attraction of a Gumbel distribution the maximal degree grows faster.
Finally, in the case when the weights are in the domain of attraction of a Fr\'echet distribution, 
the leading asymptotics of the maximal degree is random and we identify the limit as a functional of a Poisson point process. 

For Gumbel weights, in contrast to the Fr\'echet case, the first-order growth of the degrees is deterministic 
and by comparison to the case of the maximum of i.i.d.\ weights, it would be natural to conjecture that the second order 
is random. We confirm this observation for two special sub-cases of Gumbel weights. However, this result is only true when we consider a compact window in the 
region of indices that should correspond to the  index that realises the maximum degree. Finally, we 
identify the true second order and, somewhat surprisingly, it is also deterministic. This behaviour comes from the fact that 
we have to consider a much larger optimal window than initially suspected.

Finally, we address the question regarding the location of the currently fittest vertex (i.e.\ that attains the maximum degree). For classical preferential attachment models it is known, 
see e.g.~\cite{Hof16}, that there is persistence: 
in other words there is a vertex that after some time becomes the most powerful (largest) vertex
and then remains so forever.
However, in real-world networks it seems natural to allow newly incoming vertices
to become so powerful that they can compete with old vertices. 
It is known that for the WRG there is never persistence  and we establish the asymptotics of the location of the maximum in 
the Gumbel and Fr\'echet case (and see the recent work of the first author~\cite{Lod21} for the case of bounded weights).

Our results  recover and extend results on the degree distribution of RRTs and WRTs as well as the maximum degree in RRTs. Degree distributions in RRTs have been studied in ~\cite{GasWir84,MahLue92,MeirMoon88} and~\cite{NaRap70}; a very general class of weighted growing trees has been studied by Iyer in~\cite{Iyer20}. However, the results discussed so far
only consider trees, so that unlike in our work DAGs (in both the weighted and the non-weighted case) are not included.

Another type of weighted recursive trees was originally introduced by Borovkov and Vatutin in~\cite{BorVat06,BorVat206}, where the vertex-weights have a specific product-form, and in a general form by Hiesmayr and I\c{s}lak in~\cite{HieIsl17} (general in the sense that the weights need not have the product form as in~\cite{BorVat06} and~\cite{BorVat206}).
Recent work on weighted recursive trees includes the work of Mailler and Uribe Bravo~\cite{MaiBra19}, S\'enizergues~\cite{Sen19} and Pain and S\'enizergues~\cite{PainSen21} where the profile and height of the tree are analysed as well as vertex degrees, together with~\cite{Iyer20}, in which degree distributions of many weighted growing tree models are studied and the weighted recursive tree is a particular example.

Szyma\'nski~\cite{Szy90} was the first to obtain results on the growth rate of the maximum degree in RRTs, which were later extended by Devroye and Lu~\cite{DevLu95}, after which finer properties of high degrees were analysed by Goh and Schmutz~\cite{GohSch02} and Addario-Berry and Eslava~\cite{AddEsl18}. Recently, Banerjee and Bhamidi~\cite{BanBha20} studied the occurrence of persistence in growing random networks 
and they include results describing the growth rate of the location of the maximum degree in RRTs. In WRTs, the behaviour of degrees and the maximum degree has received attention from S\'enizergues in~\cite{Sen19}, where the vertex-weights satisfy a more general product-form compared to~\cite{BorVat06,BorVat206}.


A related dynamical model that also allows to produce degree distributions with a variety of tail behaviour is
the preferential attachment model with a general attachment function analysed 
by Dereich and M\"orters~\cite{DereichMoerters2009}.
In their case, assuming that the probability to connect to a vertex of degree $n$ is proportional to $ \approx n^\alpha$, the model exhibits persistence if $\alpha \in (1/2,1]$ and no persistence for $\alpha \in (0,1/2)$. In the latter case, 
the authors also give precise asymptotics for the location of the maximum. 
This contrasts with our results, where even for the heavier tailed case we do not have persistence, even though if for  $\alpha \in (0,1)$ their degree distributions exhibit stretched exponential decay.

Our methods are  related to the analysis of the preferential attachment with additive fitness carried out by the authors in~\cite{LodOrt20}.
For these preferential attachment models, the attachment probabilities are proportional to the degree plus a random weight (fitness). In these models, 
we distinguish three different regimes: first of all  a \emph{weak disorder regime}, where
the preferential attachment mechanism dominates (and there is persistence). 
This is closely related to the work of S\'enizergues~\cite{Sen19}, which in turn  corresponds to a WRT where the partial sums of the weights is at most of order $n^\gamma$ for $\gamma \in (0,1)$.
Moreover, in~\cite{LodOrt20} we identify a \emph{strong and extreme disorder regime} where the influence of the random weights takes over, which appears when the distribution of the weights is sufficiently heavy-tailed.
For WRGs there is no preferential attachment component for the vertex-weights to compete with so that the influence of the random weights is more immediate and already appears for less heavy-tailed weights.

Our results for the degree distribution follow by adjusting the proofs in~\cite{LodOrt20}, as at least on the level of 
degree distributions the WRG model is essentially a simpler model compared to the PAF models, but the asymptotic analysis is new. The results for the maximum degree in the case of bounded weights follow  with similar  techniques as developed by Devroye and Lu~\cite{DevLu95}, which can be extended to WRGs. 
The main contribution of this paper is to understand the first and second order asymptotics of the maximal degree
in the case when the weights are unbounded and satisfy suitable regularity assumptions.
For unbounded weights, the system is driven by the competition between the benefit of being an old vertex and so having time to accumulate a high degree and the benefit of being a young vertex with a large weight. 
To understand the first-order growth rate (as well as the second order when considering a compact window), 
we show concentration of the degrees around the conditional means (when conditioning on the weights) and then in a second 
step analyse the conditional mean degree using 
extreme value theory in the case of random fluctuations (similarly as in~\cite{LodOrt20}). However, as the weights have a more immediate 
impact, the results become more dependent on the exact distribution of weights chosen and thus require more 
intricate calculations.
Finally, to obtain the true second-order asymptotics in the Weibull case, we can no longer rely on 
the elegant tools of convergence to Poisson processes from 
extreme value theory and instead have to carefully keep track of errors made in the corresponding approximations.

\textbf{Notation.} Throughout the paper we use the following notation: we let $\N:=\{1,2,\ldots\}$ be the natural numbers, set $\N_0:=\{0,1,\ldots\}$ to include zero and let $[t]:=\{i\in\N: i\leq t\}$ for any $t\geq 1$. For $x,y\in\R$, we let $\lceil x\rceil:=\inf\{n\in\Z: n\geq x\}$ and $\lfloor x\rfloor:=\sup\{n\in\Z: n\leq x\}$ { \cb and let $x\vee y:=\max\{x,y\}$ and $x\wedge y:=\min\{x,y\}$}. Moreover, for sequences $(a_n)_{n\in\N}$ and $(b_n)_{n\in\N}$ we say that $a_n=o(b_n), a_n\sim b_n, a_n=\mathcal{O}(b_n)$ if $\lim_{n\to\infty}a_n/b_n=0, \lim_{n\to\infty} a_n/b_n=1$ and if there exist constants $C>0,n_0\in\N$ such that $a_n\leq Cb_n$ for all $n\geq n_0$, respectively. For random variables $X,(X_n)_{n\in\N}$ we denote $X_n\toindis X, X_n\toinp X$ and $X_n\toas X$ for convergence in distribution, probability and almost sure convergence of $X_n$ to $X$, respectively. Also, we write $X_n=o_\mathbb{P}(1)$ if $X_n \toinp 0$. Throughout, we denote by $(W_i)_{i\in\N}$  i.d.d.\ random variables and use the conditional probability measure $\Pf{\cdot}:=\mathbb{P}(\,\cdot\, |(\F_i)_{i\in\N})$ and conditional expectation $\Ef{}{\cdot}:=\E{\,\cdot\,|(\F_i)_{i\in\N}}$.

\section{Definitions and main results}\label{sec:results}

The Weighted Recursive Graph (WRG) model is a growing random graph model that is a generalisation of the random recursive tree (RRT) and the uniform directed acyclic graph (DAG) models in which vertices are assigned (random) weights and new vertices connect with existing vertices with a probability proportional to the vertex-weights.

We define the WRG model as follows:

\begin{definition}[Weighted Recursive Graph]\label{def:wrg}
	Let $(\F_i)_{i\geq 1}$ be a sequence of i.i.d.\ copies of a non-negative random variable $\F$ such that $\P{\F>0}=1$, let $m\in\N$ and set
	\be 
	S_n:=\sum_{i=1}^n\F_i.
	\ee
	We construct the \emph{Weighted Recursive Graph} as follows:
	\begin{enumerate}
		\item[1)] Initialise the graph with a single vertex $1$, denoted as the root, and assign to the root a vertex-weight $\F_1$. Denote this graph by $\G_1$.
		\item[2)] For $n\geq 1$, introduce a new vertex $n+1$ and assign to it the vertex-weight $\F_{n+1}$ and $m$ half-edges. Conditionally on $\G_n$, independently connect each half-edge to some $\inn$ with probability $\F_i/S_n$. Denote the resulting graph by $\G_{n+1}$.
	\end{enumerate}
	We will treat $\G_n$ as a directed graph, where edges are directed from new vertices towards old vertices.
\end{definition}

\begin{remark}\label{remark:def} $(i)$ Note that the edge connection probabilities remain the same if 
	we multiply each weight by the same constant. In particular, if convenient, we may without loss of generality assume for vertex-weight distributions with bounded support, i.e.\ $x_0:=\sup\{x\in\R\,|\, \P{\F\leq x}<1\}<\infty$, that $x_0=1$. 
	Alternatively, and we will do this in particular for distributions with unbounded support and finite mean, i.e.\ $x_0=\infty$ and $\E{\F}<\infty$, we can assume that  $\E{\F}=1$.
	
	$(ii)$ It is possible to extend the definition of the WRG such that the \emph{out-degree is random} and the results presented in this paper still hold. Namely, we can allow that vertex $n+1$ connects to $\emph{every}$ vertex $\inn$ independently with probability $\F_i/S_n$. At the start of sections dedicated to proving the results we present below, we discuss why the results hold for the random out-degree model as well.
\end{remark}

{\cb To formulate our results we need to assume that the {\cb tail of the }distribution of the weights is sufficiently regular, allowing us to control their extreme value behaviour. To this end, we draw on classical results from extreme value theory. Consider a sequence $(W_i)_{i\in\N}$ of i.i.d.\ random variables with distribution function $F$. Suppose there exist a distribution function $G$ and sequences $(a_n)_{n\in\N}$ and $(b_n)_{n\in\N},a_n>0,b_n\in\R,$ such that for any continuity point $x\in\R$ of $G$, 
	\be \label{eq:evt}
	\lim_{n\to\infty}F(a_nx+b_n)^n=G(x).
	\ee   
	Observe that~\eqref{eq:evt} is equivalent to
	\be \label{eq:maxconv}
	\frac{\max_{\inn}W_i-b_n}{a_n}\toindis X,
	\ee 
	where $X$ is a random variable with distribution function $G$.	The Fischer-Tipett-Gnedenko theorem, see e.g.~\cite[Proposition $0.3$]{Res13}, says that then $G$ belongs, {\cb up to translation and scaling,} to one of three families of distribution functions, parametrised by a parameter ${\cb \alpha > 0}$:
	\begin{table}[h]
		\begin{tabular}{rll}
			$(\mathrm{i})$ & Weibull: & $G(x)=\exp(-(-x)^\alpha)$ if $x<0$ and $G(x)=1$ otherwise. \\
			$(\mathrm{ii})$ & Gumbel: & $G(x)=\exp(-\exp(x)$ for $x\in \R$.\\
			$(\mathrm{iii})$ & Fr\'echet: & $G(x)=\exp(-x^{-\alpha})$ if $x>0$ and $G(x)=0$ otherwise.
		\end{tabular}
	\end{table}\\
	
	Here, $\alpha$ is known as the exponent of the distribution. If a sequence of i.i.d.\ random variables satisfies~\eqref{eq:evt}, we say that these random variables belong to (or their distribution function belongs to) the maximum domain of attraction (MDA) of one of the three cases. 
	
	{\cb 	Typical examples that belong to each of the three MDAs are: }
	\begin{enumerate}[label=(\roman*)]
		\item Weibull MDA: random variables with bounded support and polynomially decaying tails, e.g.\ beta random variables. 
		\item Gumbel MDA: random variables with either bounded or unbounded support and exponentially decaying tails, e.g. exponential or log-normal random variables.
		\item Fr\'echet MDA: random variables with unbounded support and polynomially decaying tails, e.g.\ Cauchy random variables.
	\end{enumerate}
	
	We refer to~\cite{Res13} for a more formal characterisation of the domains of attraction and more in-depth information about extreme value theory. 
	{\cb These results motivate the following assumptions:}	
}

\begin{assumption}[Vertex-weight distributions]\label{ass:weights}
	The vertex-weights $W,(W_i)_{i\in\N}$ satisfy one of the following conditions:
	\begin{enumerate}[labelindent = 1cm, leftmargin = 2.2cm]
		\item[\namedlabel{ass:weightbdd}{$($\textbf{Bounded}$)$}]The vertex-weights are almost surely bounded, i.e.\ \[ x_0:=\sup\{x\in\R\,|\, \P{\F\leq x}<1\}<\infty.\] Without loss of generality, we can assume that $x_0=1$.\\ Within this class, we can further identify vertex-weight distributions that belong to the Weibull {\cb MDA}
		and Gumbel MDA.
		\item[\namedlabel{ass:weightgumbel}{$($\textbf{Gumbel}$)$}] The vertex-weights follow a distribution that belongs to the Gumbel {\cb MDA}
		such that $x_0=\infty$. Without loss of generality, $\E{\F}=1$. {\cb  This implies that there exist sequences $(a_n)_{n\in\N}$ and $(b_n)_{n\in\N}$ such that~\eqref{eq:evt} (or, equivalently,~\eqref{eq:maxconv}) is satisfied with $G$ as in case $(ii)$.}
		
		\noindent Within this class, we further distinguish the following three sub-classes:
		\begin{enumerate}[labelindent = 1cm, leftmargin = 1.3cm]
			\item[\namedlabel{ass:weighttauinf}{$($\textbf{SV}$)$}] $b_n\sim \ell(\log n)$ where $\ell$ is an increasing function that is slowly-varying at infinity, i.e. $\lim_{x\to\infty}\ell(cx)/\ell(x)=1$ for all $c>0$.
			\item[\namedlabel{ass:weighttaufin}{$($\textbf{RV}$)$}]
			There exist $a,c_1,\tau>0$, and $b\in\R$ such that
			\[ \P{ \F \geq x} \sim a x^b 	\e^{-( x/c_1)^\tau} \quad \mbox{as } x \rightarrow \infty .\] 
			\item[\namedlabel{ass:weighttau0}{$($\textbf{RaV}$)$}] 
			There exist $a,c_1>0, b\in\R,$ and $\tau>1$ such that 
			\[ \P{ \F \geq x} \sim a (\log x)^b 	\e^{- (\log (x)/c_1)^\tau} \quad \mbox{as } x \rightarrow \infty .\] 
		\end{enumerate}
		\item[\namedlabel{ass:weightfrechet}{$($\textbf{Fr\'echet}$)$}] The vertex-weights follow a distribution that belongs to the Fr\'echet MDA. Without loss of generality, $\E{\F}=1$ (given that $\E{\F}<\infty$ is satisfied). This implies that there exists a non-negative function $\ell$ that is slowly-varying at infinity and some $\alpha>1$, such that
		\be 
		\P{\F\geq x}=\ell(x)x^{-(\alpha-1)}.
		\ee
		Moreover, if we let $u_n:=\sup\{t\in \R:\P{\F\geq t}\geq 1/n\}$, then~\eqref{eq:evt} (or, equivalently,~\eqref{eq:maxconv}) is satisfied with $b_n=0, a_n=u_n$.
	\end{enumerate} 
\end{assumption}

\begin{remark}\label{remark:weights}
	Note that~\cite{Tak87} shows (with a slight error in the paper in that the $\log a$ term below is a $\log \tau$ in~\cite{Tak87}) that if the weight distribution satisfies the assumption~\ref{ass:weighttaufin}, then we can choose 
	\be\label{eq:bnrv} a_n= c_2(\log n)^{1/\tau-1}, b_n= c_1 (\log n)^{1/\tau}+a_n((b/\tau)\log\log n+b\log c_1+\log a),
	\ee 
	for the same constants as above and $c_2:=c_1/\tau$. Moreover, in the~\ref{ass:weighttau0} sub-case, we can choose 
	\be\ba \label{eq:bnrav}
	a_n&=c_2(\log n)^{1/\tau-1}\exp(c_1(\log n)^{1/\tau}+c_2(\log n)^{1/\tau-1}((b/\tau)\log\log n+b\log c_1+\log a)),\\
	b_n&= \exp(c_1 (\log n)^{1/\tau}+c_2(\log n)^{1/\tau-1}((b/\tau)\log\log n+b\log c_1+\log a)).
	\ea\ee 
	
	In particular, the  three sub-cases in the~\ref{ass:weightgumbel} case,~\ref{ass:weighttauinf},~\ref{ass:weighttaufin}, and~\ref{ass:weighttau0}, can be distinguished as $b_n=g(\log n)$, with $g$ a slowly-varying, regularly-varying and rapidly-varying function at infinity, respectively. Note that in all cases, $b_n$ itself is slowly varying at infinity. In the~\ref{ass:weighttaufin} sub-case, we  very often use the asymptotic equivalence for $b_n$, that is, $b_n\sim c_1(\log n)^{1/\tau}$. Moreover, in the~\ref{ass:weighttau0} sub-case, we can think of $b_n$ as $\exp((\log n)^{1/\tau}\ell(\log n))$ and $a_n$ as $c_2(\log n)^{1/\tau-1}b_n$.
	
	Furthermore, the~\ref{ass:weighttaufin} sub-case contains (stretched and compressed) exponential distributions, e.g.\ the gamma distribution. The~\ref{ass:weighttau0} sub-case contains log-compressed exponential distributions, e.g.\ the log-normal distribution. 
	
	{\cb Finally, despite the fact that the sequences $a_n$ and $b_n$ are indexed by integer values, we sometimes will, for ease of writing, abuse notation and index them by non-integer values such as $n^\gamma$ with $\gamma\in(0,1)$. As the examples in~\eqref{eq:bnrv} and~\eqref{eq:bnrav} are continuous {\cb  if we replace $n$ by 	$t \in (0,\infty)$, we can extend their definitions correspondingly.}}
	\end{remark}
	
	We now present the results for the degree distribution and the maximum degree in the WRG model. In comparison to the preferential attachment with additive fitness (PAF) models as studied in~\cite{LodOrt20}, vertex-weights with a distribution with a `thin' tail, i.e.\ distributions with exponentially decaying tails or bounded support, now can also exert their influence on the behaviour of the system.
	
	Throughout, we write 
	\[  	\zni := \mbox{ in-degree of vertex } i \mbox{ in } \G_n . \]
	We prefer to work with the in-degree as it then is easier to (in principle) generalize our methods to graphs with random out-degree. Obviously, if the out-degree is fixed, we can recover the results for the degree from our results for $\zni$.
	
	\subsection{Degree distribution} The first result deals with the degree distribution of the WRG model. Let us first introduce the following measures and quantities:
	\be \label{eq:gammas}
	\Gamma_n:=\frac{1}{n}\sum_{i=1}^n \zni \delta_{\F_i},\quad \Gamma_n^{(k)}:=\frac{1}{n}\sum_{i=1}^n \ind_{\{\zni=k\}}\delta_{\F_i},\quad p_n(k):=\Gamma_n^{(k)}([0,\infty)),
	\ee 
	where $\delta$ is a Dirac measure, and which correspond to the empirical weight distribution of a vertex sampled weighted by its in-degree, then the joint empirical vertex-weight and in-degree distribution and finally the empirical
	degree distribution. We can then formulate the following theorem:
	
	\begin{theorem}[Degree distribution in WRGs]\label{Thrm:wrrtdegree}
Consider the WRG model in Definition~\ref{def:wrg} and suppose that the vertex-weights have finite mean and denote their distribution by $\mu$. Then, almost surely, for any $k\in\N_0$, as $n\to\infty$, 
\be \label{eq:gammaconv}
\Gamma_n\to\Gamma,\qquad \Gamma_n^{(k)}\to\Gamma^{(k)},\quad \text{and}\quad p_n(k)\to p(k),
\ee 
where the first two statements hold with respect to the weak$^\star$
topology and the limits are given as
\be \label{eq:gammalimit}
\Gamma(\d x):=\frac{xm}{\E W}\mu(\d x),\qquad \Gamma^{(k)}(\d x)=\frac{\E W/m}{\E W/m+x}\bigg(\frac{x}{\E W/m+x}\bigg)^{k}\mu(\d x),
\ee
and
\be \label{eq:pk}
p(k)=\int_0^\infty \frac{\E W/m}{\E W/m+x}\bigg(\frac{x}{\E W/m+x}\bigg)^{k}\mu(\d x).
\ee 
Finally, let the vertex-weight distribution be a power law as in the~\ref{ass:weightfrechet} case of Assumption~\ref{ass:weights} with $\alpha\in(1,2)$, such that there exists an $x_l>0$ with $\mu(x_l,\infty)=1$, i.e.\ the vertex-weights are bounded away from zero almost surely. Let $U_n$ be a uniformly at random selected vertex in $\G_n$, let $\eps>0$ and let $E_n:=\{\Zm_n(U_n)=0\}$. Then, for all $n$ sufficiently large, 
\be \label{eq:gwrrtEn}
\P{E_n}\geq 1-Cn^{-((2-\alpha)\wedge (\alpha-1))/\alpha+\eps},
\ee 
for some constant $C>0$.
\end{theorem}

An important question regarding the degree distribution $p(k)$,  as in~\eqref{eq:pk}, is its asymptotic behaviour as $k \rightarrow \infty$. As it turns out, the answer depends on the particular choice of the distribution of the random variable $W$. Before we present a theorem dedicated to the asymptotic behaviour of the limiting degree distribution $p(k)$, we first introduce the following general lemma, which will allows us 
to distinguish several different cases for bounded weights.

\begin{lemma}\label{lemma:gumbelgumbel}
Let $W$ be a non-negative random variable such that $x_0:=\sup\{x>0:\P{W\leq x}<1\}<\infty$. Then, the distribution of $W$ belongs to the Weibull (resp.\ Gumbel) MDA if and only if $(x_0-W)^{-1}$ is a non-negative random variable with a distribution with unbounded support that belongs to the Fr\'echet (resp.\ Gumbel) MDA.
\end{lemma}

We are now ready to present the results on the asymptotics of $p(k)$.

\begin{theorem}[Asymptotic behaviour of $p(k)$]\label{thrm:pkasymp}
Consider the WRG with vertex-weights $(W_i)_{i\in\N}$, i.i.d.\ copies of a non-negative random variable $W$ such that $\E{W}<\infty$. We consider the different cases with respect to the vertex-weights as in Assumption~\ref{ass:weights}.
\begin{enumerate}[labelindent = 1cm, leftmargin = 2.2cm]
	\item[\namedlabel{thrm:pkbdd}{$(\mathrm{\mathbf{Bounded}})$}] Let $\theta_m:=1+\E{\F}/m$ and recall that $x_0=\sup\{x>0:\P{W\leq x}<1\}<\infty$. When $W$ belongs to the Weibull MDA  with parameter $\alpha>1$, for all $k>m/\E{W}$,
	\be\label{eq:pkbddweibull} 
	\underline L(k)k^{-(\alpha-1)}\theta_m^{-k} \leq p(k)  \leq \overline L(k)k^{-(\alpha-1)} \theta_m^{-k},
	\ee
	where $\underline L,\overline L$ are slowly varying at infinity.\\
	When $W$ belongs to the Gumbel MDA,
	\begin{enumerate}
		\item[$(i)$] If $(1-W)^{-1}$ satisfies the~\ref{ass:weighttaufin} sub-case with parameter $\tau>0$, then for $\gamma:=1/(\tau+1)$, 
		\be\label{eq:pkbddgumbelrv}
		p(k)= \exp\Big(-(1+o(1))\frac{\tau^\gamma}{1-\gamma}\Big(\frac{(1-\theta_m^{-1})k}{c_1}\Big)^{1-\gamma}\Big)\theta_m^{-k}.
		\ee			
		\item[$(ii)$] If $(1-W)^{-1}$ satisfies the~\ref{ass:weighttau0} sub-case with parameter $\tau>1$,
		\be\label{eq:pkbddgumbelrav}
		p(k)  = \exp\Big(-\Big(\frac{\log k}{c_1}\Big)^\tau\Big(1-\tau(\tau-1)\frac{\log\log k}{\log k} + \frac{K_{\tau,c_1,\theta_m}}{\log k}(1+o(1))\Big)\Big)\theta_m^{-k},		
		\ee 
		where $K_{\tau,c_1,\theta_m}:=\tau\log ( \e c_1^{\tau}(1-\theta_m^{-1})/ \tau)$.
	\end{enumerate}
	Finally, when $W$ has an atom at $x_0$, i.e.\ $q_0=\P{W=x_0}>0$, then
	\be \label{eq:pkatom_orig} p(k)=q_0(1-\theta_m^{-1})\theta_m^{-k} (1+o(1)).\ee 
	\item[\namedlabel{thrm:pkgumb}{$(\mathrm{\mathbf{Gumbel}})$}] If $W$ satisfies the~\ref{ass:weighttaufin} sub-case with parameter $\tau$, then for $\gamma:=1/(\tau+1)$,
	\be\label{eq:pkgumbrv}
	p(k)=\exp\Big(-\frac{\tau^\gamma}{1-\gamma}\Big(\frac{k}{c_1m}\Big)^{1-\gamma}(1+o(1))\Big).			
	\ee
	If $W$ satisfies the~\ref{ass:weighttau0} sub-case with parameter $\tau>1$,
	\be\label{eq:pkgumbrav}
	p(k)= k^{-1}\exp\Big(-\Big(\frac{\log (k/m)}{c_1}\Big)^\tau\Big(1-\tau(\tau-1)\frac{\log \log (k/m)}{\log (k/m)}+\frac{K_{\tau,c_1}}{\log (k/m)}(1+o(1))\Big)\Big),
	\ee 
	where $K_{\tau,c_1}:=\tau\log ( \e c_1^{\tau}/ \tau)$.
	\item[\namedlabel{thrm:pkfrechet}{$(\mathrm{\mathbf{Fr\acute echet}})$}] When $\alpha>2$,
	\be\label{eq:pkfrechet}
	\underline \ell(k)k^{-\alpha} \leq p(k)\leq \overline\ell(k)k^{-\alpha},
	\ee
	where $\overline \ell,\underline \ell$ are slowly varying at infinity. 
\end{enumerate}  
\end{theorem}

\begin{remark}[Finer asymptotics when $W$ has an atom at $x_0$]
In the case that $W$ has an atom at $x_0$, i.e.\ $q_0=\P{W=x_0}>0$, then the proof immediately gives a more precise
version of~\eqref{eq:pkatom_orig} under additional assumptions.
If there exists an $s\in(0,1)$ such that $\P{W\in(s,1)}=0$,
\be \label{eq:pkatomgap}
p(k)=q_0(1-\theta_m^{-1})\theta_m^{-k}\big(1+\mathcal O\big(\exp(-(1-\theta_m^{-1})(1-s)k)\big)\big).
\ee 
Otherwise, for any positive sequence $(s_k)_{k\in\N}$ such that $s_k\uparrow 1$, we have that
\be \label{eq:pkatom}
p(k)=q_0(1-\theta_m^{-1}) \theta_m^{-k} \big(1+{\cb \mathcal O\big(\max\{\exp(-(1-\theta_m^{-1})(1-s_k)k), \P{W\in(s_k,1)}\}\big)}\big).
\ee 
\end{remark}

\subsection{Maximum degree}
The asymptotic behaviour of the degree distribution $p(k)$ in Theorem~\ref{thrm:pkasymp} allows for a non-rigorous estimation of the size of the maximum degree in $\G_n$. As is the case for the degree distribution, the behaviour of the maximum degree in the WRG model is highly dependent on the underlying distribution of the vertex-weights as well, and on a heuristic level one would expect the size of the maximum degree, say $d_n$, to be such that $\sum_{k \geq d_n} p(k)\approx 1/n$. The following theorem makes this heuristic statement precise and states the first-order growth rate of the maximum degree for three different classes of vertex-weight distributions. In all classes, we find, up to the leading order in the asymptotic expression of $p(k)$, that $\sum_{k \geq d_n} p(k)\approx 1/n$ is satisfied when considering the asymptotic expressions in Theorem~\ref{thrm:pkasymp}. In most cases, we can also identify the order of the label or location of the vertex that attains the maximum degree, which tells us \emph{when} the vertices with largest degree are introduced to the graph. Throughout, we set 
\be\label{eq:in}
I_n:=\inf\{\inn:\zni\geq \Zm_n(j)\text{ for all }j\in[n]\}.
\ee 

\subsubsection{Bounded case}

\begin{theorem}[Maximum degree in WRGs,~\ref{ass:weightbdd} case]\label{thrm:maxbdd}
Consider the WRG model as in Definition~\ref{def:wrg} and assume the vertex-weights satisfy the~\ref{ass:weightbdd} case. Let $\theta_m:=1+\E{\F}/m$. Then,
\be
\frac{\max_{\inn}\zni }{\log n}\toas \frac{1}{\log \theta_m}.
\ee 
\end{theorem}

In~\cite{EslLodOrt21}, Eslava and the authors of this article obtain, under additional assumptions and only in the tree case (i.e.\ $m=1$), higher-order asymptotic behaviour of the maximum degree and of near-maximum degree vertices.

A result we have been unable to prove, is the convergence of the location of the maximum degree in the WRG when the vertex-weights are almost surely bounded, which improves and extends the result proved in~\cite{BanBha20} for the random recursive tree.

\begin{conjecture}[Location of the maximum degree in WRGs with bounded weights]
Consider the WRG model as in Definition~\ref{def:wrg} and assume the vertex-weights satisfy the~\ref{ass:weightbdd} case. Let $\theta_m:=1+\E{\F}/m$ and recall $I_n$ from~\eqref{eq:in}. Then,
\be 
\frac{\log I_n}{\log n}\toas 1-\frac{\theta_m-1}{\theta_m\log(\theta_m)}.
\ee 	
\end{conjecture}

Since the first appearance of this article, this result has been proved by the first author in~\cite[Theorem $2.3$]{Lod21}. With more precise assumptions on the distribution of the vertex-weights, higher-order behaviour of the labels of high-degree vertices is obtained as well. 

\subsubsection{Fr\'echet case}

{\cb Our next result considers the case when 
the weights satisfy the~\ref{ass:weightfrechet} assumption.
In this case, it is classical that  the first order of the maximum of $n$ i.i.d.\ weights 
is random. The next theorem states that this also applies to the maximum degree
in this case. However, the normalizations and limits are different 
due to the non-trivial competition, where older vertices can achieve a higher degree because they have been in the system for longer, while younger vertices have the chance to have a big vertex-weight corresponding to a local maximum.}

\begin{theorem}[Maximum degree in WRGs,~\ref{ass:weightfrechet} case]\label{thrm:maxfrechet}
Consider the WRG model as in Definition~\ref{def:wrg} and assume the vertex-weights satisfy the~\ref{ass:weightfrechet} case. Recall $I_n$ from~\eqref{eq:in} and let $\Pi$ be  a Poisson point process (PPP) on $(0,1)\times (0,\infty)$ with intensity measure $\nu(\d t,\d x):=\d t \times (\alpha-1)x^{-\alpha}\d x$. Then, when $\alpha>2$,
\be\ba\label{eq:wrrtppplimit}
(\max_{i\in[n]}\zni/u_n,I_n/n) &\toindis (m\max_{(t,f)\in\Pi} f\log(1/t),I_\alpha),
\ea \ee
where $m\max_{(t,f)\in\Pi} f\log(1/t)$ and $I_\alpha$ are independent, with $I_\alpha\overset{d}{=}\mathrm{e}^{-W_\alpha}$ and $W_\alpha$ a $\Gamma(\alpha,1)$ random variable, and where $m\max_{(t,f)\in\Pi}f\log(1/t)$ has a Fr\'echet distribution with shape parameter $\alpha-1$ and scale parameter $m\Gamma(\alpha)^{1/(\alpha-1)}$. Finally, when $\alpha\in(1,2)$,
\be\ba \label{eq:wrrtinfmeanppplimit}
(\max_{i\in[n]}\zni/n,I_n/n)&\toindis (Z,I),
\ea \ee 	
for some random variable $I$ with values in $(0,1)$ and where 
\[ Z  = m\max_{(t,f)\in\Pi} f \int_t^1\bigg(\int_{(0,1)\times(0,\infty)} g\ind_{\{u\leq s\}}\,\d \Pi(u,g)\bigg)^{-1}\,\d s . \]
\end{theorem} 

\begin{remark}
$(i)$ The result in~\eqref{eq:wrrtinfmeanppplimit} is equivalent to the behaviour of the preferential attachment models with infinite mean power-law fitness random variables, as presented in~\cite{LodOrt20}. Here, the influence of the fitness (vertex-weights) overpowers the preferential attachment mechanism so that the preferential attachment graph behaves like a weighted recursive graph.

$(ii)$ The result in~\eqref{eq:wrrtppplimit} can actually be extended to hold joint for the $K$ largest degrees and their locations as well, for any $K\in\N$. The limits $(Z^{(K)},I_\alpha^{(K)})$ of the $K^{\mathrm{th}}$ largest degree and its location are independent, $I_\alpha^{(K)}\overset{d}{=}\e^{-W_\alpha^{(K)}}$, where the $(W_\alpha^{(K)})_{K\in\N}$ are i.i.d.\ $\Gamma(\alpha,1)$ random variables, and 
\be
\P{Z^{(K)}\leq x}=\sum_{i=0}^{K-1} \frac{1}{i!}(\Gamma(\alpha)(x/m)^{-(\alpha-1)})^i\exp(-\Gamma(\alpha)(x/m)^{-(\alpha-1)}).
\ee 
\end{remark}

\subsubsection{Gumbel case} 

{\cb In the \ref{ass:weightgumbel} case, we expect by comparing to the maximum
of i.i.d.\ weights that the first order of the maximal degree is deterministic, 
which is confirmed by our next result.}

\begin{theorem}[Maximum degree in WRGs,~\ref{ass:weightgumbel} case]\label{thrm:maxgumb}
Consider the WRG model as in Definition~\ref{def:wrg} and assume the vertex-weights satisfy the~\ref{ass:weightgumbel} case. Recall $I_n$ from~\eqref{eq:in}. For sub-case~\ref{ass:weighttauinf},
\be \label{eq:degtauinf}
\Big(\max_{\inn}\frac{\zni}{mb_n\log n},\frac{\log I_n}{\log n}\Big)\toinp (1,0).
\ee
For sub-case~\ref{ass:weighttaufin}, let $\gamma:=1/(\tau+1)$.
Then,
\be\ba\label{eq:gumbppplim}
\Big(\max_{\inn}\frac{\zni}{m(1-\gamma)b_{n^\gamma}\log n},\frac{\log I_n}{\log n}\Big)&\toas (1,\gamma).
\ea\ee
Finally, for sub-case~\ref{ass:weighttau0}, let $t_n:=\exp(-\tau\log n/\log(b_n))$. Then,
\be \label{eq:degtau0}
\Big(\max_{\inn}\frac{\zni }{mb_{t_nn}\log(1/t_n)},\frac{\log I_n}{\log n}\Big)\toinp (1,1).
\ee
\end{theorem}

\begin{remark}

$(i)$ We conjecture that the convergence in~\eqref{eq:degtauinf} can be strengthened to almost sure convergence. This is readily checked for particular vertex-weight distributions that satisfy the~\ref{ass:weightgumbel}-\ref{ass:weighttauinf} sub-case, e.g. $W:=\log W'$, where $W'$ satisfies the~\ref{ass:weightgumbel}-\ref{ass:weighttaufin} sub-case.

$(ii)$ In~\eqref{eq:gumbppplim} and~\eqref{eq:degtau0} the rescaling of the maximum degree can be interpreted as $m(1-\gamma)c_1\gamma^{1/\tau}(\log n)^{1+1/\tau}$ and $m\e^{-1}c_2^{-1}(\log n)^{1-1/\tau}\exp\big(c_1(\log n)^{1/\tau}\big)$, respectively, since the lower order terms of $b_n$ as in~\eqref{eq:bnrv} and~\eqref{eq:bnrav} can be ignored when considering only the first-order behaviour of the maximum degree. As a result, it should be possible to weaken the assumptions on the tail-distribution in the~\ref{ass:weightgumbel}-\ref{ass:weighttaufin} and~\ref{ass:weightgumbel}-\ref{ass:weighttau0} sub-cases to $\P{W\geq x}=\exp\big(-(x/c_1)^\tau(1+o(1))\big)$ and $\P{W\geq x}=\exp\big(-(\log(x)/c_1)^\tau(1+o(1))\big)$, respectively, such that the results in~\eqref{eq:gumbppplim} and~\eqref{eq:degtau0} still hold.
\end{remark}

{\cb For the maximum of  i.i.d.\ random variables satisfying 
the~\ref{ass:weightgumbel} case, it is also classical that the second order term is random, }as in~\eqref{eq:maxconv}. A natural conjecture would be that the same is true for this model. Our next result shows that if we consider the indices in a compact window around the maximal one (as identified in Theorem~\ref{thrm:maxgumb}), then this is indeed true.

To formulate our results, we introduce the following notation:
For $0<s<t<\infty$, $\gamma\in(0,1)$ and a strictly positive function $f$, define
\be\ba\label{eq:cnin}
C_n(\gamma,s,t,f)&:=\{\inn:sf(n)n^\gamma\leq i\leq tf(n)n^\gamma \},\\
I_n(\gamma,s,t,f)&:=\inf\{ i\in C_n(\gamma,s,t,f): \zni\geq \Zm_n(j)\text{ for all }j\in C_n(\gamma,s,t,f)\}.
\ea \ee 
We abuse notation to also write $C_n(1,s,t,t_n)$ and $I_n(1,s,t,t_n)$ when we deal with vertices $i$ such that $st_nn\leq i\leq tt_nn$ for some sequence $(t_n)_{n\in\N}$. We then present the following theorem:

\begin{theorem}[Random second-order asymptotics in the~\ref{ass:weightgumbel} case]\label{Thrm:maxsecond}
In the same setting as in Theorem~\ref{thrm:maxgumb}, we further assume that the vertex-weights fall into the sub-case~\ref{ass:weighttaufin}.
Let $\gamma:=1/(\tau+1)$ and let $\ell$ be a strictly positive function such that $\lim_{n\to\infty}\log(\ell(n))^2/\log n=\zeta_0$ for some $\zeta_0\in[ 0,\infty)$. Furthermore, let $\Pi$ be a Poisson point process (PPP) on $(0,\infty)\times \R$ with intensity measure $\nu(\d t,\d x):=\d t\times \e^{-x}\d x$. Then, when $\tau\in(0,1)$,
\be\label{eq:gumbppplim2ii} \begin{aligned}
	\Big(\max_{i\in C_n(\gamma,s,t,\ell)}& \frac{\zni-m(1-\gamma)b_{n^\gamma}\log n}{m(1-\gamma)a_{n^\gamma}\log n},\frac{I_n(\gamma,s,t,\ell)}{\ell(n)n^\gamma}\Big)\\
	& \quad\toindis \Big(\max_{\substack{(v,w)\in \Pi\\ v\in[s,t]}} w-\log v-\frac{\zeta_0(\tau+1)^2}{2\tau},\e^U \Big),
\end{aligned}\ee 
where $U\sim \mathrm{Unif}(\log s,\log t)$ and the maximum over the PPP follows a Gumbel distribution with location parameter $\log\log(t/s)-\zeta_0(\tau+1)^2/2\tau$. 

Finally, let us assume that the vertex-weights fall into the sub-class~\ref{ass:weighttau0} and let\\ $t_n:=\exp(-\tau\log n/\log(b_n))$. Then, for any $0<s<t<\infty$ and with $\Pi$ and $U$ as above,
\be\label{eq:rav2ndorder}
\Big(\max_{i\in C_n(1,s,t,t_n)}\frac{\zni-mb_{t_nn}\log(1/t_n)}{ma_{t_nn}\log(1/t_n)},\frac{I_n(1,s,t,t_n)}{t_nn}\Big)\toindis \Big(\max_{\substack{(v,w)\in\Pi\\ v\in[s,t]}}w-\log v,\e^U\Big),
\ee 
where now the maximum follows a Gumbel distribution with location parameter $\log\log(t/s)$.
\end{theorem} 

\begin{remark}[The vertex with largest degree for~\ref{ass:weightgumbel} weights]\label{rem:gumbel_tau_large}
The restriction to $\tau \in (0,1)$ in~\eqref{eq:gumbppplim2ii} comes from the fact our result only looks at the fluctuations coming from the random weights and indeed the same statement is true for all $\tau>0$ when looking at the conditional expected degrees (conditioned on the random weights), see Proposition~\ref{prop:condmeantaufin} later on.
By a central limit theorem-type argument we would expect that the fluctuations  of the degree around its conditional mean are of the order square root of its variance (which is comparable to its mean and so of order $(\log n)^{(1/\tau+1)/2}$), therefore
if $\tau > 1$ this term would be larger than the fluctuations coming from the random weights (which are of order $(\log n)^{1/\tau}$) and so we would expect a different scaling  limit.
\end{remark}

A standard Poisson process calculation (for more details see Section~\ref{sec:overview}) shows that the limit random variables describing the second-order growth of the near-maximal
degree in Theorem~\ref{Thrm:maxsecond} become infinite if we let $s \downarrow 0$ and $t \rightarrow \infty$. 
This phenomenon indicates that we need to consider a much larger window of indices to capture the true second-order
asymptotics of the maximal degree over the full set of indices. This fact also shows that the 
competition between the advantages of older vertices compared to vertices with high weight is very finely balanced.
The following result captures the resulting effect on the second-order asymptotics, {\cb 
which surprisingly are no longer random.}

.

\begin{theorem}[Precise second-order asymptotics in the~\ref{ass:weightgumbel} case]\label{Thrm:maxsecond2}
In the same setting as in Theorem~\ref{thrm:maxgumb}, we first assume that the vertex-weights fall into the sub-case~\ref{ass:weighttaufin} and let $\gamma:=1/(\tau+1)$.
For $\tau\in(0,1]$, 
\be \label{eq:nottight}
\max_{\inn}\frac{\zni-m(1-\gamma)b_{n^\gamma}\log n}{m(1-\gamma)a_{n^\gamma}\log n\log\log n}\toinp \frac 12.
\ee 
Now assume that the vertex-weights fall into the sub-class~\ref{ass:weighttau0} and let\\ $t_n:=\exp(-\tau\log n/\log(b_n))$. If $\tau\in(1,3]$,
\be \label{eq:ravcor2nd}
\max_{\inn}\frac{\zni-mb_{t_nn}\log(1/t_n)}{ma_{t_nn}\log(1/t_n)\log\log n}\toinp \frac{1}{2}\Big(1-\frac{1}{\tau}\Big),
\ee 
whilst for $\tau>3$, 
\be \label{eq:ravcor2nd3}
\max_{\inn}\frac{\zni-mb_{t_nn}\log(1/t_n)}{ma_{t_nn}\log(1/t_n)(\log n)^{1-3/\tau}}\toinp -\frac{\tau(\tau-1)^2}{2c_1^3}.
\ee 
\end{theorem} 

\begin{remark}
Although the result in~\eqref{eq:gumbppplim2ii} only holds for $\tau\in(0,1)$, the result in~\eqref{eq:nottight} turns out to hold for $\tau=1$ as well. This slight deviation is due to the fact that the additional $\log\log n$ term allows us to prove concentration of the maximum degree around the maximum conditional mean degree, which cannot be done with the second-order rescaling in~\eqref{eq:gumbppplim2ii} when $\tau=1$.
\end{remark}

As mentioned above {\cb in Remark~\ref{rem:gumbel_tau_large},} the problem of capturing the second-order fluctuations in the~\ref{ass:weightgumbel}-\ref{ass:weighttaufin} case when $\tau > 1$
and for lighter tailed weights (including bounded weights) requires different techniques
and is currently on-going research, see~\cite{EslLodOrt21} and~\cite{Lod21}.

\begin{remark}[More general model formulation]
As in~\cite{LodOrt20}, it is possible to prove some of the results for a more general class of models. More specifically, the results in Theorems~\ref{Thrm:wrrtdegree} and~\ref{thrm:maxfrechet} hold for a growing network that satisfies the following conditions as well: let $\Delta \zni:=\Zm_{n+1}(i)-\zni$. For all $n\in\N$:
\begin{enumerate}[label=($\mathrm{A}$\arabic*)]
	\item \label{Ass:gwrrtA1} $\Ef{}{\Delta \zni}= \F_i/S_n\ind_{\{\inn\}}.$
	\item \label{Ass:gwrrtA4} For all $k\in\N$, there exists a $C_k>0$ such that \[\Ef{\big}{\prod_{j=0}^{k-1}(\Delta \Zm_n(i)-j)}\leq C_k\Ef{}{\Delta \zni}.\]
	\item \label{Ass:gwrrtA3} $\sup_{i=1,\ldots,n} n\big|\Pf{\Delta \Zm_n(i)=1}-\Ef{}{\Delta \Zm_n(i)}\big| \toas 0.$
	\item \label{Ass:gwrrtA5} Conditionally on $(\F_i)_{i\in\N}$, $\{\Delta \Zm_n(i)\}_{i\in[n]}$ is negatively quadrant dependent in the sense that for any $i\neq j$ and $k,l\in\Z^+$,
	\be \label{eq:gwrrtAssA5}
	\Pf{\Delta \Zm_n(i)\leq k,\Delta \Zm_n(j)\leq l}\leq \Pf{\dzn{i}\leq k}\Pf{\dzn{j}\leq l}.
	\ee
\end{enumerate}
If we further assume that $\Delta \zni\in\{0,1\}$ then all the results presented in this paper hold as well.
\end{remark}

\textbf{Outline of the paper}\\
In Section~\ref{sec:overview} we provide a short overview and explain the intuitive idea of the results. We then prove Theorems~\ref{Thrm:wrrtdegree} and~\ref{thrm:pkasymp} in Section~\ref{sec:degree}. In Section~\ref{sec:meandeg}, we state and prove several propositions regarding the maximum conditional mean degree. We then discuss under which scaling the maximum degree concentrates around the maximum conditional expected degree in Section~\ref{sec:conc}. Finally, we use these results in Section~\ref{sec:mainproof} to prove the main results, Theorems~\ref{thrm:maxbdd},~\ref{thrm:maxfrechet},~\ref{thrm:maxgumb},~\ref{Thrm:maxsecond} and~\ref{Thrm:maxsecond2}.

\section{Overview of the proofs}\label{sec:overview}

First, since the proof of Theorem~\ref{Thrm:wrrtdegree} heavily relies on the proof of Theorem $2.4$ in~\cite{LodOrt20}, we refer to~\cite[Section $3$]{LodOrt20} for an overview of its proof. The same holds for Theorem~\ref{thrm:maxbdd}, which follows the same strategy as the proof of Theorem $2$ in~\cite{DevLu95} but where we need to take extra care because of the random weights. Finally, the proof of Theorem~\ref{thrm:pkasymp} is mainly computational in nature and we therefore do not include an overview in this section.

Here, we provide an intuitive idea of the proof of Theorems~\ref{thrm:maxfrechet} and~\ref{thrm:maxgumb} regarding first-order scaling limits of the maximum degree, as well as Theorems~\ref{Thrm:maxsecond} and~\ref{Thrm:maxsecond2} regarding second-order scaling limits. In the first two theorems, the main idea consists of two ingredients: We first consider the asymptotics of the conditional expected degree  $\Ef{}{\zni}$ of a vertex $\inn$, where we  condition on the weights $(W_i)_{i \in [n]}$. Then in a second step, we show that the degrees concentrate around their conditional expected values.

More precisely for the concentration argument, we show that
\be \label{eq:concentration}
|\max_{\inn}\zni -\max_{\inn}\Ef{}{\zni}|/g_n \toinp 0,
\ee 
for some suitable sequence $(g_n)_{n\in\N}$ such that $g_n$ diverges with $n$. What sequence $g_n$ is sufficient depends on the different vertex-weight distribution cases, as outlined in Assumption~\ref{ass:weights}. For completeness, $g_n=mb_n\log n,g_n=m(1-\gamma)b_{n^\gamma}\log n, g_n=m b_{t_nn}\log(1/t_n), g_n=u_n$ and $g_n=n$ for the~\ref{ass:weightgumbel}-\ref{ass:weighttauinf},~\ref{ass:weightgumbel}-\ref{ass:weighttaufin},~\ref{ass:weightgumbel}-\ref{ass:weighttau0} sub-cases and the~\ref{ass:weightfrechet} case with $\alpha>2$ and $\alpha\in(1,2)$, respectively. \eqref{eq:concentration} follows by applying standard large deviation bounds to $|\zni-\Ef{}{\zni}|$ for all $\inn$, as in Proposition~\ref{lemma:wrrtconcentration}. 
We can also construct
a concentration argument when $g_n$ is equivalent to the second-order growth rate of the maximum as in Theorems~\ref{Thrm:maxsecond} and~\ref{Thrm:maxsecond2}, however
in this case a more careful analysis of the different terms in the large deviation bounds is required.

The bulk of the argument for our results is to show that the conditional expected degree behaves as we claim in the main results. For the first-order asymptotics as in Theorem~\ref{thrm:maxgumb}, we have 
\be 
\max_{\inn}\frac{\Ef{}{\zni}}{g_n}\toinp 1,
\ee 
with $g_n$ as described above for the~\ref{ass:weighttauinf},~\ref{ass:weighttaufin} and~\ref{ass:weighttau0} sub-cases, whereas in~\ref{thrm:maxfrechet} we have
\be \ba
\max_{\inn}\frac{\Ef{}{\zni}}{u_n}&\toindis m\max_{(t,f)\in\Pi}f\log(1/t),\\ \max_{\inn}\frac{\Ef{}{\zni}}{n}&\toindis m\max_{(t,f)\in\Pi} f \int_t^1\bigg(\int_{(0,1)\times(0,\infty)} g\ind_{\{u\leq s\}}\d\, \Pi(u,g)\bigg)^{-1}\,\d s, 
\ea \ee 
depending on whether $\alpha > 2$ or $\alpha \in (0,1)$.
Together with~\eqref{eq:concentration}, the results in Theorems~\ref{thrm:maxfrechet} and~\ref{thrm:maxgumb} follow. For the~\ref{ass:weightgumbel}-\ref{ass:weighttaufin} sub-case (and probably the~\ref{ass:weightgumbel}-\ref{ass:weighttauinf} sub-case, too), the result can be strengthened to almost sure convergence. 

Let us delve a bit more into why the maximum conditional mean in-degree has the limits as claimed above. We stress that, as stated in Remark~\ref{remark:def}, $\E{\F}=1$ for the cases we discuss here. It is clear from the definition of the model that
\be 
\Ef{}{\zni}=m\F_i\sum_{j=i}^{n-1}1/S_j\approx m\F_i\log(n/i),
\ee 
for any $\inn$. Let us start with the~\ref{ass:weightgumbel}-\ref{ass:weighttaufin} sub-case, where we recall that $\gamma:=1/(1+\tau)$. If we set
\be 
\Pi_n:=\sum_{i\geq 1}\delta_{(i/n,(\F_i-b_n)/a_n)},
\ee 
where $\delta$ is a Dirac measure, then classical extreme value theory tells us that (see e.g.~\cite{Res13}),
\be \label{eq:weakconv}
\Pi_n{\cb \toindis} \Pi,
\ee 
where $\Pi$ is a Poisson point process on $(0,\infty)\times \R$ with intensity measure $\nu(\d t,\d x):= \d t \times \e^{-x}\d x$. Then, if we consider $i=t\ell(n) n^\gamma $ and $(\F_i-b_{\ell(n)n^\gamma })/a_{\ell(n)n^\gamma }=f$ where $\ell$ is a strictly positive function such that $\log(\ell(n))^2/\log n$ converges, it follows that 
\be \ba\label{eq:approxmean}
\Ef{}{\zni}&\approx m\F_i\log(n/i)=m(b_{\ell(n)n^\gamma }+a_{\ell(n)n^\gamma }f)\log (1/(tn^{1-\gamma }\ell(n)))\\
&\approx mc_1\gamma ^{1/\tau}(1-\gamma )(\log n)^{1/\tau+1}+mc_2\gamma ^{1/\tau-1}( f(1-\gamma )-\tau\gamma  \log t)(\log n)^{1/\tau}\\
&\approx m (1-\gamma )b_{n^\gamma }\log n+m(1-\gamma )a_{n^\gamma }\log n(f-\log t),
\ea \ee 
when using that $a_n=c_2(\log n)^{1/\tau-1}$ and $b_n\sim c_1(\log n)^{1/\tau}$, using a Taylor approximation and leaving out all lower order terms. This yields the first-order behaviour of the maximum conditional mean in-degree as well as its location $n^{\gamma(1+o(1))}$. This is proved rigorously in Proposition~\ref{prop:condmeantaufin}.

For the~\ref{ass:weighttauinf} sub-case, as in Proposition~\ref{prop:condmeantauinf}, a similar approach as in~\eqref{eq:approxmean} can be applied, though we set $i=tn^\beta$, $(W_i-b_{n^\beta})/a_{n^\beta}=f$ for any $\beta\in(0,1)$. We divide the right-hand side of~\eqref{eq:approxmean} by $b_n\log n$ and observe that $b_{n^\beta}=\ell(\log (n^\beta))=\ell(\beta\log n)$, so that it follows that $b_{n^\beta}/b_n$ converges to $1$ with $n$ for any fixed $\beta\in(0,1)$ since $\ell$ is slowly varying at infinity. Then, the constant in front of the leading term is decreasing in $\beta$, so that taking the limit $\beta\downarrow 0$ yields the required result. 

Then, for the~\ref{ass:weighttau0} sub-case, as in Proposition~\ref{prop:condmeantau0}, we realise that the location of the maximum should grow faster than $n^\gamma$ for any $\gamma\in(0,1)$, as the tails of these distribution are heavier than those of any distribution in the~\ref{ass:weighttaufin} sub-case. By a similar argument as for the~\ref{ass:weighttauinf} sub-case and using that now $b_{n^\beta}/b_n$ converges to $0$ with $n$, one might want to set $\beta=1$, that is, the location of the maximum degree is of order $n$. However, this would imply that the growth rate of the maximum expected degree should be $b_n$. This is not the case, however, since for any $t\in (0,1)$, approximately,
\be 
\max_{\inn}\Ef{}{\zni}/b_n\geq (\max_{i\in[tn]}\F_i/b_{tn})\log(1/t) b_{tn}/b_n.
\ee 
Since $b_{tn}/b_n$ converges to $1$ for $t$ fixed ($b_n$ is slowly varying) and the maximum converges to $1$ in probability, letting $t$ tend to $0$ shows scaling by just $b_n$ is insufficient. 
To find the correct behaviour, we need to let $t$ tend to zero with $n$, i.e.\ $t=t_n$, such that $b_{t_nn}/b_n$ has a non-trivial limit (not $0$ or $1$). The sequence that satisfies this requirement is $t_n=\exp(-\tau\log n/\log(b_n))$. This suggests that the location of the maximum degree is $t_nn$ and that the maximum degree grows as $b_{t_nn}\log(1/t_n)$. 

The second-order growth rate for the~\ref{ass:weighttau0} sub-case, as in Theorem~\ref{Thrm:maxsecond}, is obtained in a similar way as in~\eqref{eq:approxmean}, where we now consider $i=st_nn$ and $(\F_i-b_{t_nn})/a_{t_nn}=f$ for some $(s,f)\in(0,\infty)\times \R$. This yields
\be \ba
\Ef{}{\zni}&\approx m\F_i\log(n/i)=m(b_{t_nn}+a_{t_nn}f)\log(1/st_n)\\
&= mb_{t_nn}\log(1/t_n)+(ma_{t_nn}\log(1/t_n)f-mb_{t_nn}\log s)+ma_{t_nn}f\log(1/s).
\ea \ee 
Here, the first order again appears in the first term on the right-hand side, and the second order can be obtained by realising that $b_{t_nn}/(a_{t_nn}\log(1/t_n))\to 1$ as $n$ tends to infinity. A similar approach using the weak convergence of $\Pi_n$ to $\Pi$ as in~\eqref{eq:weakconv} allows us to obtain the required limits.  For the results of the~\ref{ass:weighttaufin} sub-case in Theorem~\ref{Thrm:maxsecond}, we take a more in-depth look at the approximation in~\eqref{eq:approxmean}. First, when subtracting the first term on the right-hand side and dividing by $m (1-\gamma )a_{n^\gamma }\log n$, we are left with exactly the functional which is used in the maximum over the Poisson point process in~\eqref{eq:gumbppplim2ii}. When combining this with the weak convergence of $\Pi_n$ to $\Pi$ the desired result follows. To understand how the additional term in~\eqref{eq:gumbppplim2ii} arises, we include an extra lower order term in the approximation in~\eqref{eq:approxmean}. That is,
\be \ba\label{eq:approx2}
\Ef{}{\zni}\approx{}& mc_1\gamma^{1/\tau}(1-\gamma)(\log n)^{1/\tau+1}+mc_2\gamma^{1/\tau-1}( f(1-\gamma)-\tau\gamma \log t)(\log n)^{1/\tau}\\
&-mc_1\frac{\gamma^{1/\tau}}{1-\gamma}\frac{1+\tau}{2}(\log n)^{1/\tau-1}\log(\ell(n))^2.
\ea \ee 
Hence, the requirement that $\lim_{n\to\infty}\log(\ell(n))^2/\log n=\zeta_0$ ensures that the last term on the right-hand side does not grow faster than $(\log n)^{1/\tau}$. Also, when divided by the second-order growth rate $m(1-\gamma)a_{n^\gamma}\log n$, the last term converges to $-(1+\tau)^2\zeta_0/(2\tau)$, exactly the additional term found in~\eqref{eq:gumbppplim2ii}. 

Somewhat surprising is that for both the~\ref{ass:weighttaufin} and~\ref{ass:weighttau0} sub-cases, when we consider not a compact window around the optimal index but instead all $\inn$, we find that the second-order correction as suggested above is insufficient, as can be observed in~\eqref{eq:nottight} and~\eqref{eq:ravcor2nd} in Theorem~\ref{Thrm:maxsecond2}. The reason this behaviour is observed, informally, is that we can move away even further from what one would expect to be the optimal window, i.e.\ $\ell(n)n^\gamma$ for $\ell$ that do not grow or decay `too quickly' in the~\ref{ass:weighttaufin} sub-case and $t_nn$ in the~\ref{ass:weighttau0} sub-case, and find higher degrees. As we have observed above, when setting $s=0,t=\infty$ in the limit in~\eqref{eq:gumbppplim2ii}, which would mimic considering all $\inn$ rather than the indices in a compact window around $n^\gamma$, we obtain $\sup_{(v,w)\in\Pi}w-\log v$ (when we would set $\zeta_0=0$). It is readily checked that integrating the intensity measure $\nu$ over $\R_+\times \R$ yields an infinitely large rate. Hence, $\sup_{(v,w)\in\Pi}w-\log v=\infty$ almost surely, so that the second-order scaling $(1-\gamma)a_{n^{\gamma}}\log n$ is insufficient when considering all $\inn$.

A more insightful argument is the following: consider the PPP limit as in~\eqref{eq:gumbppplim2ii} (with $\zeta_0=0$) and~\eqref{eq:rav2ndorder}. Its distribution depends only on the ratio $t/s$. Thus, for any integer $j\in\N$ such that $j=o(\sqrt{\log n})$,
\be 
\max_{\e^{j-1}n^\gamma\leq i\leq \e^jn^\gamma}\frac{\zni-m(1-\gamma)b_{n^\gamma}\log n}{m(1-\gamma)a_{n^\gamma}\log n}\toindis \Lambda_j,
\ee  
where the $(\Lambda_j)_{j\in\N}$ are i.i.d.\ standard Gumbel random variables (where the location parameter equals $0$). Their independence follows from the independence property of the PPP. Now, as the Gumbel distribution satisfies the~\ref{ass:weightgumbel}-\ref{ass:weighttaufin} sub-case with $\tau=1,c_1=1,b=0$, we can argue that, for any $\eta>0$ and $x\in\R$,
\be \ba
\mathbb P\bigg({}&\max_{n^\gamma\leq i\leq \e^{(\log n)^{1/2-\eta}}n^\gamma}\frac{\zni-m(1-\gamma)b_{n^\gamma}\log n}{m(1-\gamma)a_{n^\gamma}\log n}\leq x(1/2-\eta)\log\log n\bigg)\\
&=\P{\max_{1\leq j\leq(\log n)^{1/2-\eta}} \max_{\e^{j-1}n^\gamma\leq i\leq \e^jn^\gamma}\frac{\zni-m(1-\gamma)b_{n^\gamma}\log n}{m(1-\gamma)a_{n^\gamma}\log n}\leq x(1/2-\eta)\log\log n}\\
&\approx \P{\max_{1\leq j\leq (\log n)^{1/2-\eta} }\Lambda_j\leq x(1/2-\eta)\log\log n}\\
&=\P{\Lambda_1\leq x(1/2-\eta)\log\log n}^{(\log n)^{1/2-\eta}}\\
&=\exp(-(\log n)^{1/2-\eta}\exp(-x(1/2-\eta)\log\log n)),
\ea \ee   
which has limit $1$ (resp.\ $0$) if $x>1$ (resp.\ $x<1$). Then, as we can choose $\eta$ arbitrarily close to $0$, the result in~\eqref{eq:nottight} follows. Here it is essential that $j=o(\sqrt{\log n})$ to obtain the correct limit. Note that making the approximation $\approx$ in the above argument rigorous is the highly non-trivial part of the argument. The reason is that we cannot rely on the elegant theory of convergence to a Poisson point processes, but 
have to explicitly control the errors made in this approximation.

A similar reasoning can be applied for the~\ref{ass:weightgumbel}-\ref{ass:weighttau0} when $\tau\in(1,3]$, where now $j=o((\log n)^{(1-1/\tau)/2})$ needs to be satisfied. When $\tau>3$, more care needs to be taken of lower-order terms that appear in the double exponent, yielding a different scaling.

For Theorem~\ref{thrm:maxfrechet}, a similar approach can be used. We now set 
\be 
\Pi_n:=\sum_{i=1}^n \delta_{(i/n,\F_i/u_n)},
\ee 
where again $\Pi_n\Rightarrow \Pi$, with $\Pi$ a Poisson point process on $(0,1)\times (0,\infty)$ with intensity measure $\nu(\d t, \d x):=\d t\times (\alpha-1)x^{-\alpha}\d x$. Now considering $i=tn, \F_i=fu_n$, yields
\be 
\Ef{}{\zni/u_n}\approx m\frac{\F_i}{u_n}\log(n/i)=mf\log(1/t), 
\ee 
which is the functional in the maximum in~\eqref{eq:wrrtppplimit}. Again, combining this with the weak convergence of $\Pi_n$ yields the result. Finally, the heuristic idea for~\eqref{eq:wrrtinfmeanppplimit} is contained in~\cite[Section $3$]{LodOrt20} as well, since the rescaled maximum degree in the preferential attachment model with additive fitness studied there has the same distributional limit when $\alpha\in(1,2)$.

\section{The limiting degree sequence of weighted recursive graphs}\label{sec:degree}

In this section we prove Theorems~\ref{Thrm:wrrtdegree} and~\ref{thrm:pkasymp}. {\cb For more in-depth details regarding this entire proof, we refer to the similar proof in~\cite[Theorem $2.4$]{LodOrt20}, on which this proof is based.} We first state a result which is crucial for the proof of the former:

{\cb 
\begin{lemma}[Lemma $3.1$, \cite{DerOrt14}]\label{lemma:stochapprox}
	Let $(X_n)_{n\geq 0}$ be a non-negative stochastic process. Suppose that 
	\be 
	X_{n+1}-X_n\leq \frac{1}{n+1}(A_n-B_nX_n)+(R_{n+1}-R_n)
	\ee 
	holds almost surely, where 
	\begin{itemize}
		\item $(A_n)_{n\geq 0}$ and $(B_n)_{n\geq 0}$ are almost surely convergence stochastic processes with deterministic limits $A,B>0$, respectively. 
		\item $(R_n)_{n\geq 0}$ is an almost surely convergent stochastic process.
	\end{itemize}
	Then, almost surely,
	\be 
	\limsup_{n\to\infty} X_n \leq \frac AB. 
	\ee 
	If instead, a lower bound applies to $X_{n+1}-X_n$, we can bound $\liminf_{n\to\infty}X_n$ from below by $A/B$ almost surely.
	\end{lemma} }
	
	\begin{proof}[Proof of Theorem~\ref{Thrm:wrrtdegree}]
{\cb We first prove~\eqref{eq:gammaconv}-~\eqref{eq:pk}, for which we can, without loss of generality, assume that $\E W=1$ as $\E W<\infty$ is satisfied. It then follows that $S_n/n\toas 1$. We let $0\leq f<f'<\infty$, $\mathbb F=[0,\infty)$, and set $\I_n:=\{\inn: W_i\in(f,f']\}$ and $X_n:=\Gamma_n((f,f'])=\tfrac1n \sum_{i\in\I_n}\zni$. By expressing $\E{X_{n+1}\,|\, \G_n}$ in terms of $X_n$ and the vertex-weights, we can obtain the bounds
	\be \ba
	X_{n+1}-X_n&\geq \frac{1}{n+1}\Big(-X_n+\frac{\I_n}{n}\frac{mf}{S_n/n}\Big)+(R_{n+1}-R_n),\\
	X_{n+1}-X_n&\leq \frac{1}{n+1}\Big(-X_n+\frac{\I_n}{n}\frac{mf' }{S_n/n}\Big)+(R_{n+1}-R_n),
	\ea\ee 
	where $R_{n+1}-R_n:=X_{n+1}-\E{X_{n+1}\,|\, \G_n}$. Using the law of large numbers and Lemma~\ref{lemma:stochapprox}, this results in the upper and lower bound,
	\be \ba
	\liminf_{n\rightarrow \infty}X_n\geq mf\mu((f,f']),\qquad
	\limsup_{n\rightarrow \infty}X_n\leq mf'\mu((f,f']),
	\ea\ee 
	almost surely, if we assume that $R_n$ converges almost surely. We show this assumption holds at the end of the proof. We now take a countable subset $\mathbb F\subset [0,\infty)$, such that $\mathbb F$ is dense and for each $f\in \mathbb F, \mu(\{f\})=0$. It follows that there exists an almost sure event $\Omega_0$ on which both bounds hold for any pair $f,f'\in \mathbb F$ such that $f<f'$. Then, take an arbitrary set open set $U$ and approximate $U$ from below by sequence of sets $(U_j)_{j\in\N}$, where each $U_j$ is a finite union of intervals $(f,f']$, with $f,f'\in\mathbb F$. Then, for any $j\in\N$, a Riemann approximation yields 
	\be 
	\liminf_{n\to\infty} \Gamma_n(U)\geq \liminf_{n\to\infty}\Gamma_n(U_m)\geq \Gamma(U_m), \qquad \text{on }\Omega_0. 
	\ee 
	The monotone convergence theorem thus implies that $\liminf_{n\to\infty} \Gamma_n(U)\geq \Gamma(U)$ almost surely. A similar argument for any closed set $C$ yields $\limsup_{n\to\infty}\Gamma_n(C)\leq \Gamma(C)$. The Portmenteau lemma finally provides us with the almost surely convergence in the weak$^*$ topology. 
	
}

In the remainder of the proof, we let $X_n:=\Gamma_n^{(k)}((f,f'])=\tfrac 1n\sum_{i\in\I_n}\ind_{\{\zni=k\}}$ and use a proof by induction. So, we assume that $\Gamma_n^{(j)}$ converges almost surely to $\Gamma^{(j)}$ in the weak$^*$ topology. By expressing $\E{X_{n+1}\,|\, \G_n}$ in terms of $X_n$ and the vertex-weights, we can obtain a lower bound
\be 
X_{n+1}-X_n\geq \frac{1}{n+1}(A_n-B_n'X_n)+(R_{n+1}'-R_n'),
\ee
where $A_n,B_n'$ almost surely converge to 
\be 
A:=m\int_{(f,f']}x\; \Gamma^{(k-1)}(\d x),\qquad B':=\frac{1/m+f'}{1/m},
\ee  
respectively, and where $R_{n+1}'-R_n':=X_{n+1}-\E{X_{n+1}\,|\, \G_n}$ (we again show the convergence of $R_n'$ for each $k\in\N_0$ at the end of the proof). With a similar approach, an upper bound on the recursion $X_{n+1}-X_n$ can be obtained with sequences $A_n,B_n$ that converge to $A$ and $B$, respectively, with $B=1+mf$. Now, applying Lemma~\ref{lemma:stochapprox} yields
\be \ba
\liminf_{n\rightarrow \infty}X_n&\geq \frac{A}{B'}=\frac{1}{1/m+f'} \int_{(f,f']}x\;\Gamma^{(k-1)}(\d x),\\
\limsup_{n\rightarrow \infty}X_n&\leq \frac{A}{B}=\frac{1}{1/m+f}\int_{(f,f']}x\;\Gamma^{(k-1)}(\d x).
\ea\ee 
The same argument as used for $\Gamma_n$ then yields
\be 
\Gamma^{(k)}(\d x)=\Big(\frac{x}{x+1/m}\Big)^{k} \Gamma^{(0)}(\d x). 
\ee 
With minor adjustments, it also follows that
\be 
\Gamma^{(0)}(\d x)=\frac{1/m}{x+1/m}\mu(\d x),
\ee 
from which~\eqref{eq:gammaconv},~\eqref{eq:gammalimit} and~\eqref{eq:pk} follow. 

{\cb It thus remains to prove the convergence of $R_n$ and $R_n'$. First, 
	\be\ba 
	R_{n+1}-R_n & =\frac{1}{n+1}\sum_{i\in\I_n}(\Zm_{n+1}(i)-\E{\Zm_{n+1}(i)\,|\, \G_n})\\ & =\frac{1}{n+1}\sum_{i\in\I_n}(\Delta \zni-\E{\Delta \zni\,|\, \G_n}), 
	\ea \ee 
	where $\Delta \zni:=\Zm_{n+1}(i)-\zni$. Then, we can bound
	\be \ba 
	\E{(R_{n+1}-R_n)^2\, |\, \G_n}&=\frac{1}{(n+1)^2}\mathbb E \bigg[\Big(\sum_{i\in\I_n}(\Delta \zni-\E{\Delta \zni\,|\, \G_n})\Big)^2\,\bigg|\, \G_n\bigg]\\
	&\leq \frac{1}{(n+1)^2}\sum_{i\in\I_n}\Var(\Delta\zni\,|\, \G_n).
	\ea \ee 
	Here, the second step follows from the negative quadrant dependence of the random variables $(\Delta \zni)_{\inn}$, in the sense that for any distinct $i,j\in[n]$ and any $k,\ell\in\N_0$, 
	\be 
	\P{\Delta\zni\geq k, \Delta \Zm_n(j)\geq \ell\,|\, \G_n}\leq \P{\Delta\zni\geq k\,|\, \G_n}\P{\Delta\Zm_n(j)\geq \ell\,|\, \G_n}. 
	\ee 
	As the $\Delta \zni$ are the outcome of a multinomial experiment (conditionally on the vertex-weights, in particular conditionally on $\G_n$) this property directly follows from~\cite[Section $3.1$]{DevPros83}. When we consider the model with \emph{random} out-degree, we note that all the $\Delta \zni$ are independent indicator random variables, so that the negative quadrant dependence also applies.
	
	It thus remains to bound the conditional variance of $\Delta \zni$ from above. As $\Delta \zni$ is a sum of $m$ i.i.d.\ indicator random variables, its variance is at most $m$ times its mean, for each $\inn$. Since $\sum_{i=1}^n \E{\Delta \zni}=m$, we obtain the upper bound $(m/(n+1))^2$, which is summable almost surely. As this implies that the quadratic variation of the martingale $R_{n+1}-R_n$ {\cb is bounded}, it follows that $R_n$ converges almost surely.  
	
	We now prove the almost sure convergence of $R_n'$. For $k\geq 1$,
	\be
	R_{n+1}'- R_n'=\frac{1}{n+1}\sum_{i\in\I_n}\left(\ind_{\{\Zm_{n+1}(i)=k\}}-\P{\Zm_{n+1}(i)=k\;|\;\G_n}\right)=\Delta M^{(1)}_n-\Delta M_n^{(2)},
	\ee 
	where $\Delta M_n^{(i)}$ is a martingale difference, i.e. $\Delta M_n^{(i)}=M_{n+1}^{(i)}-M_n^{(i)},\ i\in\{1,2\}$, and 
	\be \ba \label{eq:mns}
	\Delta M_{n}^{(1)}&=\frac{1}{n+1}\bigg(\sum_{i\in\I_n}\ind_{\{\Zm_n(i)<k,\Zm_{n+1}(i)\geq k\}}-\mathbb{E}\bigg[\sum_{i\in\I_n}\ind_{\{\Zm_n(i)<k,\Zm_{n+1}(i)\geq k\}}\,\bigg|\,\G_n\bigg]\bigg),\\
	\Delta M_{n}^{(2)}&=\frac{1}{n+1}\bigg(\sum_{i\in\I_n}\ind_{\{\Zm_n(i)\leq k,\Zm_{n+1}(i)> k\}}-\mathbb{E}\bigg[\sum_{i\in\I_n}\ind_{\{\Zm_n(i)\leq k,\Zm_{n+1}(i)> k\}}\,\bigg|\,\G_n\bigg]\bigg).
	\ea \ee 
	Here, we use that 
	\be \ba
	\ind_{\{\Zm_{n+1}(i)=k\}}&=\ind_{\{\Zm_{n+1}(i)=k,\Zm_n(i)\leq k\}}=\ind_{\{\Zm_{n+1}(i)\geq k,\Zm_n(i)\leq k\}}-\ind_{\{\Zm_{n+1}(i)>k,\Zm_n(i)\leq k\}}\\
	&=\ind_{\{\Zm_n(i)=k\}}+\ind_{\{\Zm_{n+1}(i)\geq k,\Zm_n(i)< k\}}-\ind_{\{\Zm_{n+1}(i)>k,\Zm_n(i)\leq k\}}.
	\ea \ee 
	We note that, as the indicators in $M_n^{(1)},M_n^{(2)}$ only differ by one index $k$, it is sufficient to prove the summability of the conditional second moment of $\Delta M_n^{(2)}$ for all fixed $k\geq 1$. So, we write
	\be \ba\label{eq:mn2bound}
	\mathbb{E}\Big[{}&(\Delta M_n^{(2)})^2\,\Big|\,\G_n\Big]\\
	&=\frac{1}{(n+1)^2}\mathbb{E}\bigg[\Big(\sum_{i\in\I_n}\big(\ind_{\{\Zm_n(i)\leq k,\Zm_{n+1}(i)> k\}}-\P{\Zm_n(i)\leq k,\Zm_{n+1}(i)> k|\G_n}\big)\Big)^2\bigg|\G_n\bigg].
	\ea\ee
	Using the negative quadrant dependence of the random variables $(\Delta \zni)_{\inn}$, we can bound this from above by,
	\be\ba\label{eq:Rnconvbound}
	\frac{1}{(n+1)^2}&\sum_{i\in\I_n}\mathbb{E}\Big[\Big(\ind_{\{\Zm_n(i)\leq k,\Zm_{n+1}(i)> k\}}-\P{\Zm_n(i)\leq k,\Zm_{n+1}(i)> k\;|\;\G_n}\Big)^2\,\Big|\,\G_n\Big]\\
	&\leq \frac{1}{(n+1)^2}\sum_{i\in\I_n}\ind_{\{\Zm_n(i)\leq k\}}\P{\Delta \Zm_n(i)\geq 1\;|\;\G_n}\\
	&\leq  \frac{1}{(n+1)^2}\sum_{i=1}^n \E{\Delta \Zm_n(i)\;|\;\G_n}= \frac{m}{(n+1)^2},
	\ea\ee 
	where we use Markov's inequality in the penultimate step and use that the (expected) increments of all in-degrees is exactly $m$. Hence, the final statement is summable almost surely, which proves the almost sure convergence of $ R_n'$. For $k=0$, we can write $R_{n+1}'- R_n'$ as
	\be 
	\Delta M_{n}^{(1)}+\Delta M_{n}^{(2)} +(\ind_{\{\F_{n+1}\in(f,f']\}}-\mu((f,f']))/(n+1),
	\ee 
	where $\Delta M_n^{(1)}=0$ and $\Delta M_n^{(2)}$ is as in \eqref{eq:mns} with $k=0$. We already proved the summability of the second conditional moment of $M_n^{(2)}$ which follows for $k=0$ as well, and the last term has a second conditional moment bounded by $\mu((f,f'])/(n+1)^2$, which is summable too. This proves the almost sure convergence of $R_n'$ for all $k\in\N_0$.
}

We conclude by proving~\eqref{eq:gwrrtEn} for $m=1$ (the proof for $m>1$ follows analogously).  Let $\beta\in(0,(2-\alpha)/(\alpha-1)\wedge 1)$. We then obtain from a union bound that
\be \ba\label{eq:probbound}
\P{E_n^c\cap\{ \F_{U_n}\leq n^\beta\}}&\leq \frac{1}{n}\sum_{j=1}^{n-1}\sum_{k=1}^j n^\beta \E{(1/S_j)\ind_{\{\F_k\leq n^\beta\}}}\leq Cn^{\beta-1}\sum_{j=1}^{n-1} j\E{1/M_j},
\ea\ee 
where we bound $S_j$ from below by the maximum vertex-weight $M_j:=\max_{i\in[j]}\F_i$ and $C>0$ is a constant. We can then bound the expected value of $1/M_j$ by
\be 
\E{1/M_j}\leq \P{M_j\leq j^{1/(\alpha-1)-\eps}}/x_l +j^{-1/(\alpha-1)+\eps}.
\ee 
For $j$ large, say $j>j_0\in\N$, we can bound the probability on the right-hand side from above by 
\be
\P{M_j\leq j^{1/(\alpha-1)-\eps}}\leq \exp(-j^{(\alpha-1)\eps/2}),
\ee 
which leads to the bound
\be 
\E{1/M_j}\leq \ind_{\{j\leq j_0\}}/x_l +\ind_{\{j>j_0\}}(1+1/x_l)j^{-1/(\alpha-1)+\eps}.
\ee 
We then use this in~\eqref{eq:probbound} to obtain
\be 
\P{E_n^c\cap\{ \F_{U_n}\leq n^\beta\}}\leq \wt C n^{\beta-((2-\alpha)/(\alpha-1)\wedge 1)+\eps},
\ee 
for a constant $\wt C>0$. Combining this with $\P{\F_{U_n}\geq n^\beta}=\ell(n^\beta)n^{-\beta/(\alpha-1)}\leq n^{-\beta/(\alpha-1)+\eps}$ for $n$ sufficiently large, by~\cite[Proposition $1.3.6$ $(v)$]{BinGolTeu87}, yields
\be \ba
\P{E_n} & \geq \P{W_{U_n}\leq n^\beta}-\P{E_n^c\cap \{W_{U_n}\leq n^\beta\}}\\ &\geq  1-n^{-\beta(\alpha-1)+\eps}-\wt C n^{\beta-((2-\alpha)/(\alpha-1)\wedge 1)+\eps}.
\ea \ee 
Taking $C=1+\wt C$ and choosing the optimal value of $\beta$, namely $\beta=((2-\alpha)/(\alpha(\alpha-1)))\wedge (1/\alpha)$, yields the desired result and concludes the proof. 
\end{proof}

Prior to proving Theorem~\ref{thrm:pkasymp}, we prove Lemma~\ref{lemma:gumbelgumbel} as stated in Section~\ref{sec:results}, and introduce some general bounds on $p(k)$. 

\begin{proof}[Proof of Lemma~\ref{lemma:gumbelgumbel}]
Without loss of generality, we can set $x_0=1$. The claim relating the Weibull and Fr\'echet maximum domains of attractions follows directly from~\cite[Propositions $1.11$ and $1.13$]{Res13} and the fact that $\P{W\geq 1-1/x}=\P{(1-W)^{-1}\geq x}$.

By~\cite[Corollary $1.7$]{Res13}, the random variable $W$ belongs to the Gumbel MDA if and only if there exist a $z_0<1$ and measurable functions $c,g,f$ 
satisfying $\lim_{x\to 1}c(x)=\widehat c>0$ and $\lim_{x\to 1}g(x)=1$
such that 
\be 
\P{W\geq x}=c(x)\exp\Big(-\int_{z_0}^x \frac{g(t)}{f(t)}\,\d t\Big),\quad z_0<x<1,
\ee  
with $f$, known as the auxiliary function, being absolutely continuous, $f>0$ on $(z_0,1)$ and \\$\lim_{x\to 1}f'(x)=0$. We are required to find a $\wt z_0<\wt x_0$ and measurable functions $\wt c,\wt g,\wt f$ with the same properties for the random variable $(1-W)^{-1}$ to prove one direction.

First, we readily have that $\wt x_0=\sup\{x>0:\P{(1-W)^{-1}\leq x}<1\}=\infty$, so that $(1-W)^{-1}$ has unbounded support. Then, take $\wt z_0=(1-z_0)^{-1}$ and note that $\wt z_0<\infty$ as $z_0<1$. For any $x>\wt z_0$,
\be
\P{(1-W)^{-1}\geq x}=\P{W\geq 1-1/x}=c(1-1/x)\exp\Big(-\int_{\wt z_0}^x\frac{g(1-1/u)}{f(1-1/u)}u^{-2}\,\d u\Big),
\ee 
via a variable transformation $u=(1-t)^{-1}$. We thus set $\wt c(x):=c(1-1/x)$, $\wt g(x):=g(1-1/x)$, and $\wt f(x):=f(1-1/x)x^2$. It follows that $\wt f$ is absolutely continuous and strictly positive on $(\wt z_0,\infty)$. Moreover, 
\be 
\lim_{x\to \infty}\wt c(x)=\lim_{x\to \infty}c(1-1/x)=\lim_{u\uparrow 1}c(u)=\widehat c,\quad 
\lim_{x\to\infty}\wt g(x)=\lim_{u\uparrow 1}g(u)=1,
\ee 
and 
\be
\wt f'(x)=-f'(1-1/x)+2xf(1-1/x)=-f'(1-1/x)+2(1-(1-1/x))^{-1}f(1-1/x).
\ee
Hence, using the result from \cite[Lemma $1.2$]{Res13} that $\lim_{u\uparrow 1}(1-u)^{-1}f(u)=0$, we find that
\be 
\lim_{x\to\infty}\wt f'(x)=-\lim_{u\uparrow 1}f'(u)+2\lim_{u\uparrow 1}(1-u)^{-1}f(u)=0,
\ee 
so that all the conditions of~\cite[Corollary $1.7$]{Res13} are satisfied for the tail distribution of $(1-W)^{-1}$.

For the other direction, we use the other result of~\cite[Lemma $1.2$]{Res13}, which states that \\$\lim_{u\to\infty} u^{-1}f(u)=0$ when $f$ is an auxiliary function for an unbounded distribution belonging to the Gumbel MDA. With similar steps as above, the required result then follows as well.
\end{proof}

\begin{lemma}
For any weight distribution that belongs to the~\ref{ass:weightbdd} case and for any non-negative sequence $(s_k)_{k\in\N}$ such that $s_k\uparrow 1$, when $k$ is sufficiently large,
\be \ba \label{eq:pkboundsbdd}
p(k)&\leq \exp(-(1-\theta_m^{-1})(1-s_k)k)\theta_m^{-k}+(1-\theta_m^{-1})\theta_m^{-k}\P{W\geq s_k},\\
p(k)&\geq \exp(-(1-\theta_m^{-1})k/t_k)\theta_m^{-k}\P{W\geq 1-1/t_k}(1+o(1)), 
\ea \ee 
where $t_k=(1-s_k)^{-1}$, and where the lower bound only holds if $\sqrt k=o(t_k)$ (otherwise the $1+o(1)$ term is to be included in the exponent).

For any weight distribution that has unbounded support and for any non-negative sequence $(s_k)_{k\in\N}$ such that $s_k$ diverges with $k$ such that $s_k\leq k/m$ for all $k\in\N$ sufficiently large, 
\be\ba \label{eq:pkboundsunbdd}
p(k)&\leq \frac{1}{s_k}\Big(\exp\Big(-\frac{k}{ms_k}+\frac{k}{(ms_k)^2}\Big)+\P{W\geq s_k}\Big),\\
p(k)&\geq \Big(1-\frac{1}{ms_k}\Big)^k\int_{s_k}^\infty \frac{1}{1+mx}\mu(\d x) .
\ea\ee  
\end{lemma} 

\begin{proof}
We first prove~\eqref{eq:pkboundsbdd}. We know from Theorem~\ref{Thrm:wrrtdegree}, with $\mu$ the distribution
of $W$, that
\be\label{eq:pkint}
p(k)
=\!\int_0^\infty\!\!\! \frac{\E W/m}{\E W/m+x}\Big(\frac{x}{\E W/m+x}\Big)^k\mu(\d x),
\ee 
for any choice of vertex-weights $W$ such that $\E{W}<\infty$. We recall that $\E{W}/m=\theta_m-1$. It is straightforward to check that the integrand in~\eqref{eq:pkint} is unimodal and maximised at $x=\E{W}k/m$. Thus, for $k>m/\E{W}$, the integrand is increasing and maximal at $x=1$. This directly yields, for some non-negative sequence $(s_k)_{k\in\N}$ such that $s_k\uparrow 1$,
\be \ba\label{eq:pkub}
p(k)&\leq \int_0^{s_k}\Big(\frac{s_k}{\theta_m-1+s_k}\Big)^k\mu(\d x)+\int_{s_k}^1\frac{\theta_m-1}{\theta_m}\theta_m^{-k}\mu(\d x)\\
&\leq \Big(\frac{\theta_m s_k}{\theta_m-1+s_k}\Big)^k\theta_m^{-k}+\frac{\theta_m-1}{\theta_m}\theta_m^{-k}\P{W\geq s_k}\\
&\leq \exp(-(1-\theta_m^{-1})(1-s_k)k)\theta_m^{-k}+(1-\theta_m^{-1})\theta_m^{-k}\P{W\geq s_k}.
\ea \ee 
For a lower bound, we again split the integral at $s_k$, but only keep the second integral. This yields
\be
p(k)\geq (1-\theta_m^{-1})\Big(1-\frac{(\theta_m-1)(1-s_k)}{\theta_m-1+s_k}\Big)^k\theta_m^{-k}\P{W\geq s_k}.
\ee 
We now bound the second term from below by setting $s_k=1-1/t_k$ and by writing
\be\label{eq:pklbterm}
\Big(1-\frac{(\theta_m-1)/t_k}{\theta_m-1/t_k}\Big)^k=\exp\Big(k\log\Big(1-\frac{(\theta_m-1)/t_k}{\theta_m-1/t_k}\Big)\Big)=\exp(-(1-\theta_m^{-1})k/t_k)(1+o(1)),
\ee  
provided that $\sqrt k=o(t_k)$ (otherwise include the $(1+o(1))$ in the exponent) so that we arrive at 
\be \label{eq:pklb}
p(k)\geq(1-\theta_m^{-1})\exp(-(1-\theta_m^{-1})k/t_k)\theta_m^{-k}\P{W\geq 1-1/t_k}(1+o(1)).
\ee
For the proof of~\eqref{eq:pkboundsunbdd}, we have that $x_0=\infty$ and we can now assume, without loss of generality, that $\E W=1$. Hence, the expression in~\eqref{eq:pkint} simplifies to 
\be\ba 
p(k)&=\int_0^\infty \frac{1}{1+mx}\Big(\frac{mx}{1+mx}\Big)^k\mu(\d x)\\
&=\int_0^{s_k}\frac{1}{1+mx}\Big(\frac{mx}{1+mx}\Big)^k\,\mu(\d x)+\int_{s_k}^\infty\frac{1}{1+mx}\Big(\frac{mx}{1+mx}\Big)^k\,\mu(\d x).
\ea \ee 
For the first integral, we notice that the integrand is unimodal and maximised in $x=k/m$. Since $s_k\leq k/m$, we can bound the integral from above by substituting $x=k/m$ in the integrand. By observing that $(1+mx)^{-1}\leq  1/x\leq 1/s_k$ for $x\in(s_k,\infty)$ and that $x\mapsto mx/(1+mx)$ is increasing in $x$ and bounded from above by one, we thus obtain
\be 
p(k)\leq \frac{1}{s_k}\Big(\Big(\frac{ms_k}{1+ms_k}\Big)^k+\P{W\geq s_k}\Big)\leq \frac{1}{s_k}\Big(\exp\Big(-\frac{k}{ms_k}+\frac{k}{(ms_k)^2}\Big)+\P{W\geq s_k}\Big).
\ee 
For a lower bound, we omit the integral over $(0,s_k)$ an instead note that $mx/(1+mx)\geq ms_k/(1+ms_k)$ for $x\geq s_k$. This yields 
\be 
p(k)\geq \Big(\frac{ms_k}{1+ms_k}\Big)^k \int_{s_k}^\infty \frac{1}{1+mx}\mu(\d x)\geq \Big(1-\frac{1}{ms_k}\Big)^k\int_{s_k}^\infty \frac{1}{1+mx}\mu(\d x),
\ee 
which completes the proof.		
\end{proof}

We now prove Theorem~\ref{thrm:pkasymp}. To aid the reader, we split the proof into several parts, based on the different cases presented in Theorem~\ref{thrm:pkasymp}.

\begin{proof}[Proof of Theorem~\ref{thrm:pkasymp}, \ref{thrm:pkbdd} case, Weibull MDA]
We prove~\eqref{eq:pkbddweibull}. Since $W$ belongs to the Weibull MDA (with $x_0=1$), $(1-W)^{-1}$ belongs to the Fr\'echet MDA by Lemma~\ref{lemma:gumbelgumbel}, so that 
\be \label{eq:weibull}
\P{W\geq 1-1/x}=\P{(1-W)^{-1}\geq x}=\ell(x)x^{-(\alpha-1)},
\ee 
for some slowly-varying function $\ell$ and $\alpha>1$. Thus, with the upper bound in~\eqref{eq:pkboundsbdd} we obtain
\be 
p(k)\leq \Big[\exp\big(-(1-\theta_m^{-1})k/t_k\big)+(1-\theta_m^{-1})\ell(t_k)\exp(-(\alpha-1)\log t_k)\Big]\theta_m^{-k},
\ee 
where we recall that $t_k=(1-s_k)^{-1}$ and diverges with $k$. We balance the two terms in the square brackets by setting $t_k=(1-\theta_m^{-1})k/((\alpha-1)\log k)$. This yields
\be \ba 
p(k)&\leq \big( k^{-(\alpha-1)}+(1-\theta_m^{-1})^{2-\alpha}(\alpha-1)^{\alpha-1}\log(k)^{\alpha-1}\ell(t_k) k^{-(\alpha-1)}\big) \theta_m^{-k}\\ 
&=\overline L(k)k^{-(\alpha-1)}\theta_m^{-k},
\ea \ee 
where $\overline L(k):=1+(1-\theta_m^{-1})^{2-\alpha}(\alpha-1)^{\alpha-1}\log(k)^{\alpha-1}\ell(t_k)$. As $t_k$ is regularly varying and $\ell$ is slowly varying, $\ell(t_k)$ is slowly varying by~\cite[Proposition $1.5.7.(ii)$]{BinGolTeu87}.

For a lower bound we use the lower bound in~\eqref{eq:pkboundsbdd} combined with~\eqref{eq:weibull} and set $t_k=k$ to obtain
\be
p(k)\geq \underline L(k)k^{-(\alpha-1)} \theta_m^{-k},
\ee 
where $\underline L(k):=(1-\theta_m^{-1})\e^{-(1-\theta_m^{-1})} \ell(k)$, which concludes the proof.
\end{proof}

\begin{proof}[Proof of Theorem~\ref{thrm:pkasymp}, \ref{thrm:pkbdd}-\ref{ass:weighttaufin} sub-case, Gumbel MDA]	
We use an improved version of the upper bound in~\eqref{eq:pkboundsbdd} to prove~\eqref{eq:pkbddgumbelrv}. 
Recall that $\gamma = 1/(\tau+1)$. Then, we define sequences $t_{k,j}:=d_j c_1^{1-\gamma}((1-\theta_m^{-1})k)^\gamma,j\in[J]$ for constants $d_1<d_2<\ldots <d_J$ and some $J \in \N$ that we will choose at the end. We also write $f(x):=((\theta_m-1)/(\theta_m-1+x))(x/(\theta_m-1+x))^k$ for simplicity. Then, we bound 
\be\label{eq:improvbound}
p(k)\leq \int_0^{1-1/t_{k,1}}f(x)\mu(\d x)+\sum_{j=1}^{J-1}\int_{1-1/t_{k,j}}^{1-1/t_{k,j+1}}f(x)\mu(\d x)+\int_{1-1/t_{k,J}}^1 f(x)\mu(\d x).
\ee 
As $(1-W)^{-1}$ satisfies the~\ref{ass:weighttaufin} sub-case, it follows that
\be \ba
\P{W\geq 1-1/t_{k,j}}&=\P{(1-W)^{-1}\geq t_{k,j}}=(1+o(1))a t_{k,j}^b\e^{-(t_{k,j}/c_1)^{\tau}}\\
&=(1+o(1))a t_{k,j}^b\exp\Big(-d_j^\tau \Big(\frac{(1-\theta_m^{-1})k}{c_1}\Big)^{1-\gamma}\Big ).
\ea \ee 
Also using that $f(x)$ is increasing on $[0,1]$ when $k>1/(\theta_m-1)$ allows us to bound $p(k)$ from above even further by 
{\allowdisplaybreaks 
	\begin{align*}
		f(1{}&-1/t_{k,1})+\sum_{j=1}^{J-1} f(1-1/t_{k,j+1})\P{W\geq 1-1/t_{k,j}}+f(1)\P{W\geq 1-1/t_{k,J}}\\
		\leq{}& (\theta_m-1)\theta_m^{-(k+1)}\bigg[\Big(1-\frac{1}{t_{k,1}}\Big)^k \Big(1-\frac{1}{\theta_m t_{k,1}}\Big)^{-(k+1)}\\
		& +\sum_{j=1}^{J-1}\Big(1-\frac{1}{t_{k,j+1}}\Big)^k \Big(1-\frac{1}{\theta_m t_{ k,j+1}}\Big)^{-(k+1)}
		a t_{k,j}^b\exp\Big(\! -d_j^\tau \Big(\frac{(1-\theta_m^{-1})k}{c_1}\Big)^{1-\gamma}\Big)\\ & +at_{k,J}^b \exp\Big(\!-d_J^\tau \Big(\frac{(1-\theta_m^{-1})k}{c_1}\Big)^{1-\gamma}\Big) \bigg](1+o(1)) \displaybreak \\ 
		\leq{}& (1-\theta_m^{-1})\theta_m^{-k}\bigg[\exp\Big(\!-\frac{1}{d_1}\Big(\frac{(1-\theta_m^{-1})k}{c_1}\Big)^{1-\gamma}\Big)\\
		&+\sum_{j=1}^{J-1}at_{k,j}^b \exp\Big(\!-\big(d_{j+1}^{-1}+d_j^\tau\big)\Big(\frac{(1-\theta_m^{-1})k}{c_1}\Big)^{1-\gamma}\Big)\\
		&+at_{k,J}^b \exp\Big(-d_J^\tau \Big(\frac{(1-\theta_m^{-1})k}{c_1}\Big)^{1-\gamma}\Big)\bigg](1+o(1)).
\end{align*}	}
Note that the function $g(t) = t^{-1} + t^\tau$ has a minimum at $d^* = \tau^{-\gamma}$. Now, we choose 
$d_1 < d^*$ such that $1/d_1 > g(d^*) = \tau^\gamma / (1-\gamma)$ and similarly $d_\infty$ such that
$d_\infty^\tau > g(d^*)$. Given any $\eps > 0$, we can now choose $J$ sufficiently large such that 
for all $d, d' \in [d_1,d^*]$ with $|d-d'| \leq (d_\infty-d_1)/J$, we have that
$|d^{\tau} - (d')^\tau| < \eps$.
Finally, we define $d_j = d_1 + \frac{j}{J}(d_\infty - d_1)$ for $j = 2,\ldots, J$. In particular, it follows that for any $j =1,\ldots, J-1$,
\[ d_{j+1}^{-1}+d_j^\tau > d_{j+1}^{-1} + d_{j+1}^\tau - \eps \geq g(d^*) - \eps = \tau^\gamma / (1-\gamma) - \eps. \]
Substituting this into the bound for $p(k)$ and using that $t_{k,j}^b=d_j^bc_1^{(1-\gamma)b} (1-\theta_m)^{\gamma b} k^{\gamma b}\leq C k^{\gamma b}$ uniformly in $j\in[J]$ for some constant $C>0$, we arrive at
\be \ba\label{eq:rvpkub}
p(k)&\leq (1-\theta_m^{-1})\theta_m^{-k}\exp\Big(-\Big(\frac{\tau^\gamma}{1-\gamma} - \eps\Big)\Big(\frac{(1-\theta_m^{-1})k}{c_1}\Big)^{1-\gamma}\Big)\Big(1+\sum_{j=1}^JaC k^{\gamma b}\Big)\\
&\leq \exp\Big(-\Big(\frac{\tau^\gamma}{1-\gamma}- 2\eps\Big)\Big(\frac{(1-\theta_m^{-1})k}{c_1}\Big)^{1-\gamma}\Big)\theta_m^{-k},
\ea \ee 
where the last inequality holds for $k$ large enough.

For a lower bound, we use the lower bound in~\eqref{eq:pkboundsbdd} with $t_k=c_1^{1-\gamma}((1-\theta_m^{-1})k/\tau)^\gamma$. We thus obtain
\be
p(k)\geq (1-\theta_m^{b\gamma-1})\tau^{-\gamma b} c_1^{b(1-\gamma)}k^{\gamma b}\exp\Big(-(\tau^{\gamma}+\tau^{\gamma-1})\Big(\frac{(1-\theta_m^{-1})k}{c_1}\Big)^{1-\gamma}(1+o(1))\Big)\theta_m^{-k}.
\ee 
As $\sqrt k=o(t_k)$ is not guaranteed for all values of $\tau>0$, we include the $1+o(1)$ in the exponent, so that
\be 
p(k)\geq \exp\Big(-\frac{\tau^\gamma}{1-\gamma}\Big(\frac{(1-\theta_m^{-1})k}{c_1}\Big)^{1-\gamma}(1+o(1))\Big)\theta_m^{-k},
\ee 
which together with~\eqref{eq:rvpkub} yields~\eqref{eq:pkbddgumbelrv} and concludes the proof.
\end{proof}

\begin{proof}[Proof of Theorem~\ref{thrm:pkasymp}, \ref{thrm:pkbdd}-\ref{ass:weighttau0} sub-case, Gumbel MDA]
We apply a similar approach as in the proof of~\eqref{eq:pkbddgumbelrv} to prove~\eqref{eq:pkbddgumbelrav}. Again, we choose  a sequence $t_{k,1} < \ldots < t_{k,J}$ with $J = J(k)$ to be determined later. 
We set 
$t_{k,j} = (d_{J-j+1})^{-1}  (1-\theta_m^{-1}) k (\log k)^{-(\tau-1)}$ for a sequence $d_1 < d_2 < \ldots < d_J$
to be fixed later on, but such that
$d_1$ is bounded in $k$ and $\log (d_J) = o(\log k)$. Then, by the same steps as in between~\eqref{eq:improvbound} and~\eqref{eq:rvpkub}, but now using  that $(1-W)^{-1}$ satisfies the~\ref{ass:weighttau0} sub-case, 
we obtain 
\be \ba		\label{eq:pkrav1-w1}
p(k) \leq{}& (1-\theta_m^{-1})\theta_m^{-k}\bigg[\exp\Big(-d_J (\log k)^{\tau -1} \Big)+a(\log t_{k,J})^b \exp(- (\log (t_{k,J}) / c_1)^{\tau} )\\
&+\sum_{j=1}^{J-1}a (\log t_{k,j})^b \exp\Big(- d_{J-j} (\log k)^{\tau -1} - (\log (t_{k,j}) / c_1)^{\tau}\Big) \bigg](1+o(1)).
\ea \ee 				
Now, we use that by a Taylor expansion
\be
(\log t_{k,j}/c_1)^\tau \geq \Big( \frac{\log k}{c_1} \Big)^{\tau} \Big( 1 - \tau(\tau-1)\frac{\log \log k}{\log k} \Big) 
- \frac{\tau}{c_1^\tau}\log \Big(\frac{d_{J-j+1}}{1-\theta_m^{-1}}\Big )  (\log k)^{\tau -1} .
\ee
Hence, we obtain that there exists a constant $C >0$ such that
\be \label{eq:2204-11} \ba p(k) 
\leq{} & C \theta_m^{-k} (\log k)^{b \vee 0} \bigg[\exp\Big(-d_J (\log k)^{\tau -1} \Big)\\
&+
\exp\Big(- \Big( \frac{\log k}{c_1} \Big)^{\tau} \Big( 1 - \tau(\tau-1)\frac{\log \log k}{\log k} \Big)  \Big)\\
&\quad \times  \bigg( 
\sum_{j=1}^{J-1} \exp\Big(- \Big( d_{J-j} - \frac{\tau}{c_1^\tau} \log \Big(\frac{d_{J-j+1}}{1-\theta_m^{-1}}\Big) \Big) (\log k)^{\tau -1} \Big) 
\\
&\quad+ \exp\Big(\frac{\tau}{c_1^\tau} \log \Big(\frac{d_1}{1-\theta_m^{-1}}\Big)  (\log k)^{\tau -1} \Big)\bigg) \bigg].
\ea \ee
We will eventually choose $d_J$ such that $d_J \geq c_1^{-\tau}\log k$, so that we can neglect the first term. 
Secondly, we notice that the function $f(x) = x - \tau c_1^{-\tau} \log(x/(1-\theta_m^{-1}))$ is minimised at $x^* = \tau c_1^{-\tau}$, 
so we choose $d_1$ small enough such that $\tau c_1^{-\tau} (\log (d_1)) < f(x^*)  = \tau c_1^{-\tau}\log ( \e c_1^\tau(1-\theta_m^{-1})/\tau)$. 
Therefore, we can neglect the first term and the term outside the sum and can concentrate on the sum itself and so need to estimate
\be \label{eq:2204-10} 
\sum_{j=1}^{J-1} \exp\Big(- \Big( d_{j} - \frac{\tau}{c_1^\tau} \log \Big(\frac{d_{j+1}}{1-\theta_m^{-1}}\Big) \Big) (\log k)^{\tau -1} \Big) .
\ee
Let $d_\infty$ be big enough such that $\tau c_1^{-\tau}\log( d/(1-\theta_m^{-1})) \leq d/2$ for all $d \geq d_\infty$
and also big enough such that $d_\infty \geq 2( f(x^*) + 1)$. 
Given $\eps > 0$, 
let $J'$ be such that $J'\geq \eps^{-1} (d_\infty - d_1)$ (note that $J'$ does not depend on $k$). 
Then define $d_j = d_1 +  (j/J') (d_\infty - d_1)$ for $j = 1, \ldots, J'$.
Moreover, choose $d_j = d_\infty + (j - J')$ for $j \geq J'+1$. Finally, choose $J$ such that $d_{J-1}\leq (\log k) / (c_1)^\tau \leq d_{J}$. We split the sum in~\eqref{eq:2204-10} into summands smaller and bigger than $J'$ and first consider
\[ \ba  \sum_{j=1}^{J'-1}{}& \exp\Big(- \Big( d_{j}  - \frac{\tau}{c_1^\tau} \log \Big(\frac{d_{j+1}}{1-\theta_m^{-1}}\Big) \Big) (\log k)^{\tau -1} \Big)\\
& \leq \sum_{j=1}^{J'-1} \exp\Big(- \Big( d_{j+1} - \frac{\tau}{c_1^\tau} \log\Big (\frac{d_{j+1}}{1-\theta_m^{-1}}\Big) - \eps \Big) (\log k)^{\tau -1} \Big) 
\\
& \leq J' \exp\Big(- (f(x^*)  - \eps \big) (\log k)^{\tau -1} \Big) . 
\ea \]
Now, for the second sum we obtain by the assumptions on $d_j$,
\[ \ba \sum_{j = J'}^{J-1} \exp\Big(- \Big( d_{j} & - \frac{\tau}{c_1^\tau} \log \Big(\frac{d_{j+1}}{1-\theta_m^{-1}}\Big) \Big) (\log k)^{\tau -1} \Big)\\
& \leq \sum_{j= J'}^{J-1} \exp\Big(- ( d_{j}  - d_{j+1}/2) (\log k)^{\tau -1} \Big)\\
& \leq \sum_{j= J'}^{J-1} \exp\Big(- \Big( \frac 12 (d_\infty + (j - J'))  - \frac 12 \Big) (\log k)^{\tau -1} \Big)\\
& \leq \exp\Big(- \frac{1}{2} (d_\infty -1)  (\log k)^{\tau -1} \Big) \sum_{j= 0}^{\infty} \exp\Big(  - j/2 (\log k)^{\tau -1} \Big)
\ea \]
By assumption, we have that $(d_\infty - 1)/2 \geq f(x^*)$, so that combining the two last estimates 
with~\eqref{eq:2204-11},
we obtain that there exists a constant $C_1 > 0$ such that 
\be\label{eq:pkrav1-w2}
p(k) \leq C_1 \theta_m^{-k} (\log k)^{b \vee 0} 
\exp\Big(- \Big( \frac{\log k}{c_1} \Big)^{\tau} \Big( 1 - \tau(\tau-1)\frac{\log\log k}{\log k} \Big) 
- (\log k)^{\tau -1} (f (x^*) - \eps) \Big) ,
\ee
which gives the required bound as we recall that $f(x^*) = \tau c_1^{-\tau}
\log ( \e c_1^\tau(1-\theta_m^{-1})/\tau)$.

For a lower bound, we set $ t_k=(1-\theta_m^{-1})k/(x^*(\log k)^{\tau-1})$, where $x^* = \tau c_1^{-\tau}$ as before. Then, we  use the lower bound in~\eqref{eq:pkboundsbdd} to find 
\be\ba \label{eq:ravlb}
p(k)&\geq (1-\theta_m^{-1})a\log( t_k )^b \exp(-(1-\theta_m^{-1})k/t_k -  (\log(t_k)/c_1)^\tau)\theta_m^{-k}(1+o(1))\\
&\geq C_2\log(k)^b\exp\Big(-\Big(\frac{\log k}{c_1}\Big)^\tau\Big(1-\tau(\tau-1)\frac{\log\log k }{\log k} \Big)\\
& \qquad\qquad		-  (  x^*  - \tau c_1^{-\tau} \log (x^*/(1-\theta_m^{-1})) )(\log k)^{\tau-1} 
(1+o(1))\Big) \theta_m^{-k},
\ea\ee 
for some constant $C_2>0$, which proves the lower bound in~\eqref{eq:pkbddgumbelrav} since we recall that
$f(x^*) = \tau c_1^{-\tau}\log (\e c_1^\tau(1-\theta_m^{-1})/\tau)$ and which concludes the proof.
\end{proof}

\begin{proof}[Proof of Theorem~\ref{thrm:pkasymp}, \ref{thrm:pkbdd} case, Atom at one]
We aim to prove~\eqref{eq:pkatomgap} and~\eqref{eq:pkatom} and assume without loss of generality that $x_0=1$. As $q_0=\P{W=1}>0$, we immediately obtain the lower bound
\be 
p(k)\geq q_0(1-\theta_m^{-1})\theta_m^{-k},
\ee 
for both cases. For an upper bound, let us first assume there exists an $s\in(0,1)$ such that $\P{W\in(s,1)}=0$. Then, the upper bound in~\eqref{eq:pkboundsbdd} yields
\be \ba
p(k) &\leq \exp( -(1-\theta_m^{-1}) ( 1- s) k) \theta_m^{-k} + (1-\theta_m^{-1}) \theta_m^{-k} \mu((s,1])\\
&=q_0(1-\theta_m^{-1})\theta_m^{-k}\big(1+\mathcal O\big(\exp(-(1-\theta_m^{-1})(1-s)k)\big)\big), 
\ea \ee 
which proves~\eqref{eq:pkatomgap}. To conclude the proof of~\eqref{eq:pkatom}, let $(s_k)_{k\in\N}$ be any positive sequence such that $s_k\uparrow 1$. Then, $\mu((s_k,1))$ vanishes as $k\to\infty$ by the continuity of probability measures. As a result, again using the upper bound in~\eqref{eq:pkboundsbdd},
\be \ba
p(k) &\leq \exp( -(1-\theta_m^{-1}) ( 1- s_k) k) \theta_m^{-k} + (1-\theta_m^{-1}) \theta_m^{-k} \mu((s_k,1])
\\
&= q_0(1-\theta_m^{-1}) \theta_m^{-k} \big(1+\mathcal O\big(\exp(-(1-\theta_m^{-1})(1-s_k)k)\vee \mu((s_k,1))\big)\big),
\ea \ee 
as required, which concludes the proof.
\end{proof}

\begin{proof}[Proof of Theorem~\ref{thrm:pkasymp}, \ref{thrm:pkgumb}-\ref{ass:weighttaufin} sub-case]		
We prove~\eqref{eq:pkgumbrv}, and use an improved upper bound compared to~\eqref{eq:pkboundsunbdd}, similar to the proof of the~\ref{ass:weightgumbel}-\ref{ass:weighttaufin} sub-case in the~\ref{thrm:pkbdd} setting. As in~\eqref{eq:improvbound}, writing $f(x)=(mx)^k(1+mx)^{-(k+1)}$ and taking sequences $s_{k,j},k,j\in\N$ such that $s_{k,j}\leq s_{k,j+1}$ and $s_{k,j}\leq k/m$ (so that $f(s_{k,j})$ is increasing in $j$) for all $j\in\N$ and $k$ large,
\be \ba \label{eq:rvfbound}
p(k)\leq{}& \int_0^{s_{k,1}}f(x)\mu(\d x)+\sum_{j=1}^{J-1}\int_{s_{k,j}}^{s_{k,j+1}}f(x)\mu(\d x)+\int_{s_{k,J}}^\infty f(x)\mu(\d x)\\
\leq{}&  f(s_{k,1})+\sum_{j=1}^{J-1}f(s_{k,j+1})\P{W\geq s_{k,j}}+\P{W\geq s_{k,J}}\\
\leq{}& \frac{1}{s_{k,1}}\exp\Big(-\frac{k}{ms_{k,1}}+\frac{k}{(ms_{k,1})^2}\Big)\\
&+\sum_{j=1}^{J-1} \frac{1}{s_{k,j+1}}\exp\Big(-\frac{k}{ms_{k,j+1}}+\frac{k}{(ms_{k,j+1})^2}\Big)a s_{k,j}^b \exp\Big(-(s_{k,j}/c_1)^\tau\Big)(1+o(1))\\
&+as_{k,J}^{b-1} \exp\Big(-(s_{k,J}/c_1)^\tau\Big)(1+o(1)),
\ea\ee 
where $J\geq 2$ is some large integer. We then set $s_{k,j}=d_jc_1^{1-\gamma}(k/m)^\gamma$ for some constants $d_1<d_2<\ldots<d_J$ (so that $s_{k,j}\leq s_{k,j+1}$ and $s_{k,j}\leq k/m$ holds for all $j\in\N$ and all $k$ large) and note that this bound is similar to the one developed in the proof of the upper bound in~\eqref{eq:pkbddgumbelrv}, but with $1-\theta_m^{-1}$ replaced by $1/m$, some additional lower order terms in the exponents and different constants. We can thus use the same approach to conclude that for any fixed $\eps>0$, we can take $J$ large enough such that we obtain the upper bound
\be\ba \label{eq:pkgumbub}
p(k)\leq{}& \exp\Big(-(1-\eps)\frac{\tau^\gamma}{1-\gamma}\Big(\frac{k}{c_1m}\Big)^{1-\gamma}\Big)\Big(\frac{m^\gamma}{d_1c_1^{1-\gamma}k^\gamma}\\
&+ac_1^{(b-1)(1-\gamma)}m^{-(b-1)\gamma}k^{(b-1)\gamma}\sum_{j=1}^{J-1}d_j^b d_{j+1}^{-1}+ad_J^{b-1} c_1^{(1-\gamma)(b-1)}m^{-\gamma (b-1)}k^{\gamma (b-1)}\Big)\\
\leq{}&\exp\Big(-(1-2\eps)\frac{\tau^\gamma}{1-\gamma}\Big(\frac{k}{c_1m}\Big)^{1-\gamma}\Big).
\ea \ee 
For a lower bound we set $s_k= c_1^{1-\gamma}(k/(\tau m))^\gamma$ and use the lower bound in~\eqref{eq:pkboundsunbdd} to obtain for any $\eps\in(0,\tau)$, 
\be\ba \label{eq:rvpkbound}
p(k)&\geq \Big(1-(c_1m)^{-(1-\gamma)}(k/\tau)^{-\gamma}\Big)^k\int_{s_k}^\infty \frac{1}{1+mx}\mu(\d x)\\
&\geq \exp( -  (c_1m)^{-(1-\gamma)}\tau^{\gamma} k^{1-\gamma} (1+o(1)) )\frac{1}{m+1}\int_{s_k}^{s_k^{1-\tau+\eps}}x^{-1}\,\mu(\d x),
\ea \ee
We use that for any $\eps\in(0,\tau)$ and all $k$ sufficiently large,
\be 
(s_k+s_k^{1-\tau+\eps})^b=s_k^b(1+o(1)),\qquad \Big(\frac{s_k+s_k^{1-\tau+\eps}}{c_1}\Big)^\tau \leq (s_k/c_1)^\tau+2\tau s_k^\eps/c_1^\tau, 
\ee 
so that   $\P{W\geq s_k+s_k^{1-\tau+\eps}}=o(\P{W\geq s_k})$. As a result we can bound the integral in~\eqref{eq:rvpkbound} from below by 
\be\ba 
\int_{s_k}^{s_k+s_k^{1-\tau+\eps}}\!\!\!\!\!\!\!\! x^{-1}\mu(\d x) & \geq \frac{1+o(1)}{s_k}\P{W\geq s_k}=(a+o(1)) s_k^{b-1}\e^{-(s_k/c_1)^\tau}\\ & =\e^{-\tau^{\gamma-1} (k/(c_1m))^{1-\gamma}(1+o(1))}.
\ea\ee
Using this in~\eqref{eq:rvpkbound}, we arrive at 
\be \ba 
p(k)&\geq
\frac{1}{m+1}\exp\big( -  (c_1m)^{-(1-\gamma)}\tau^{\gamma} k^{1-\gamma} (1+o(1))-\tau^{\gamma-1} (k/(c_1m))^{1-\gamma}(1+o(1))\big )\\
&=\exp\Big(-\frac{\tau^\gamma}{1-\gamma}(k/(c_1m))^{1-\gamma}(1+o(1))\Big).
\ea \ee 
Combined with~\eqref{eq:pkgumbub} this proves~\eqref{eq:pkgumbrv} and concludes the proof.
\end{proof}

\begin{proof}[Proof of Theorem~\ref{thrm:pkasymp}, \ref{thrm:pkgumb}-\ref{ass:weighttau0} sub-case]
For the~\ref{ass:weighttau0} sub-case, we use a similar approach as for the proof of the~\ref{ass:weighttaufin} sub-case in the~\ref{ass:weightbdd} setting. For $j = 1, \ldots, J$, we set $s_{k,j}=d_{J-j+1}^{-1}(k/m)(\log( k/m))^{-(\tau-1)}$ for a sequence $d_1<d_2<\ldots <d_J$ and $J$ to be determined later on, such that $d_1$ is bounded in $k$ and $\log d_J=o(\log k)$. Additionally, we use that the weights satisfy the~\ref{ass:weighttau0} sub-case and use~\eqref{eq:rvfbound} to obtain 
\be \ba 
p(k)\leq{}& \Big[\frac{1}{s_{k,1}}\exp\Big(-\frac{k}{ms_{k,1}}+\frac{k}{(ms_{k,1})^2}\Big)\\
&+\sum_{j=1}^{J-1} \frac{1}{s_{k,j+1}}\exp\Big(\!-\frac{k}{ms_{k,j+1}}+\frac{k}{(ms_{k,j+1})^2}\Big)a (\log s_{k,j})^b \exp\Big(\!-(\log (s_{k,j})/c_1)^\tau\Big)\\
&+as_{k,J}^{-1}\log s_{k,J}^b \exp\Big(-(\log(s_{k,J})/c_1)^\tau\Big)\Big](1+o(1))\\
={}&\Big[d_J\frac{m\log(k)^{\tau-1}}{k}\exp\Big(-d_J(\log(k/m))^{\tau-1}+o(1)\Big)\\
&+\sum_{j=1}^{J-1}ad_{J-j}\frac{m\log (k)^{b+\tau-1}}{k}\exp\Big(\!-d_{J-j}(\log(k/m))^{\tau-1}-(\log(s_{k,j})/c_1)^\tau+o(1)\Big)\\
&+ad_1\frac{m\log(k)^{b+\tau-1}}{k}\exp\Big(-(\log(s_{k,J}/c_1))^\tau\Big)\Big](1+o(1)).
\ea \ee 
We find that determining the optimal value of the $d_1,\ldots,d_J$ follows a similar approach as in the case when $(1-W)^{-1}$ satisfies the~\ref{ass:weighttau0} sub-case in~\eqref{eq:pkrav1-w1}-~\eqref{eq:pkrav1-w2} (but with $k$ replaced with $k/m$ in the exponent and $1-\theta_m^{-1}$ omitted). As a result, we obtain for any $\eps>0$,
\be \label{eq:ravub}
p(k)\leq k^{-1}\exp\Big(-\Big(\frac{\log (k/m)}{c_1}\Big)^\tau\Big(1-\tau(\tau-1)\frac{\log\log( k/m)}{\log (k/m)}+\frac{\tau\log(\e c_1^\tau/\tau)-\eps}{\log (k/m)}\Big)\Big).
\ee 
For a lower bound on $p(k)$ we set $ s_k=c_1^\tau \tau^{-1} (k/m)(\log(k/m))^{-(\tau-1)}$. As $s_k/\sqrt k$ diverges, it follows that we can improve the lower bound in~\eqref{eq:pkboundsunbdd} to find for some small constant $C>0$,
\be \ba \label{eq:improvedlb}
p(k)&\geq C\exp(-k/(ms_k)) \int_{s_k}^\infty \frac{1}{1+mx}\mu(\d x)\geq C_m\exp(-k/(ms_k))\int_{s_k}^{2s_k} x^{-1}\mu(\d x),
\ea \ee 
for some constant $C_m> 0$. Now, since $\tau>1$, when $k$ is large,
\be 
(\log(2s_k)/c_1)^\tau\leq (\log(s_k)/c_1)^\tau+2\tau c_1^{-\tau}\log 2(\log s_k)^{\tau-1}, 
\ee 
so that $\P{W\geq 2s_k}=o(\P{W\geq s_k})$. We can thus bound~\eqref{eq:improvedlb} from below by 
\be \ba 
C_m{}&\exp(-k/(ms_k))(2s_k)^{-1}\P{W\geq s_k}(1+o(1))\\ 
&\geq C_2s_k^{-1}(\log s_k)^b\exp(-(\log(s_k)/c_1)^\tau-k/(ms_k)),
\ea \ee  
for some constant $C_2>0$. Using the precise value of $s_k$ and a Taylor expansion of $(\log s_k)^\tau$ yields
\be
p(k)\geq C_2k^{-1}\exp\Big(-\Big(\frac{\log k}{c_1}\Big)^\tau\Big(1-\tau(\tau-1)\frac{\log\log k}{\log k}+\frac{\tau\log(\e c_1^\tau/\tau)}{\log k}(1+o(1))\Big)\Big).
\ee  
Combined with~\eqref{eq:ravub} this proves~\eqref{eq:pkgumbrav} and concludes the proof.
\end{proof}

\begin{proof}[Proof of Theorem~\ref{thrm:pkasymp}, \ref{thrm:pkfrechet} case]
We prove~\eqref{eq:pkfrechet}. First, let us set $s_k=k/(m(\alpha-1+\eps)\log k)$ for some $\eps>0$ (note $s_k\leq k/m$). Then, using the upper bound in~\eqref{eq:pkboundsunbdd} we bound $p(k)$ from above by
\be\ba
p(k)&\leq m(\alpha-1+\eps)\frac{\log k}{k} \big(k^{-(\alpha-1+\eps)}(1+o(1))+\ell(s_k)(m(\alpha-1+\eps)\log k)^{\alpha-1}k^{-(\alpha-1)}\big)\\
&=o(k^{-\alpha})+L(k)k^{-\alpha}, 
\ea \ee 
where $L(k):=(m(\alpha-1+\eps)\log k)^\alpha \ell(k/(m(\alpha-1+\eps)\log k))$ is slowly varying by~\cite[Proposition 1.5.7 (ii)]{BinGolTeu87}. The required upper bound is obtained by taking $\overline  \ell(k):=(1+\eps)L(k)$.

To conclude the proof, we construct a lower bound for $p(k)$. We set $s_k = k/m$	and use the improved lower bound for $p(k)$ as in the first line of~\eqref{eq:improvedlb} to obtain 
\be\ba
p(k)\geq \frac{C}{\e}\int_{k/m}^\infty\frac{1}{1+mx}\mu(\d x)&\geq \frac{Cm}{\e(1+2k)}(\P{W\geq k/m}-\P{W\geq 2k/m})\\
&=\frac{Cm}{3\e }k^{-1}\ell(k/m)(k/m)^{-(\alpha-1)}\Big(1-\frac{\ell(2k/m)}{\ell(k/m)}2^{-(\alpha-1)}\Big). 
\ea \ee
As $\ell$ is slowly-varying at infinity, is follows that the last term can be bounded from below by a constant, as the fraction converges to one, and that $\ell(k/m)\geq \ell(k)/2$, when $k$ is large. As a result,
\be 
p(k)\geq C_2\ell(k)k^{-\alpha}=:\underline \ell(k)k^{-\alpha},
\ee 
where $C_2>0$ is a suitable constant.	
\end{proof}

\section{The maximum conditional mean degree  in WRGs}\label{sec:meandeg}

It turns out that the analysis of the maximum degree of WRGs can be carried out via the maximum of the conditional mean degrees under certain assumptions on the vertex-weight distribution. To this end, we formulate several propositions to describe the behaviour of the maximum conditional mean degree when the vertex-weights satisfy the different conditions in Assumption~\ref{ass:weights}. Let us first introduce an important quantity, namely the location of the maximum conditional mean degree,
\be 
\wt I_n:=\inf\{\inn: \Ef{}{\zni}\geq \Ef{}{\Zm_n(j)}\text{ for all }j\in[n]\}.
\ee 
Furthermore, it is important to note that, as $\zni$ is a sum of indicator random variables for any $\inn$, its conditional mean equals
\be 
\Ef{}{\zni}=m\F_i\sum_{j=i}^{n-1}\frac{1}{S_j}, 
\ee 
where we recall that $S_j = \sum_{\ell =1}^j \F_\ell$.
This is also true when we work with the model with a \emph{random out-degree}, as discussed in Remark~\ref{remark:def}$(ii)$, so that all the results in the upcoming propositions also hold for this model by setting $m=1$. 

Another important result that we use throughout the proofs of the propositions is the following lemma. We note that the conditions in the lemma are satisfied for all cases in Assumption~\ref{ass:weights} such that $\E{W}<\infty$. A similar result (under a different condition) can be found in~\cite[Theorem 1]{AthKar67}.

\begin{lemma} Let $W_1,W_2, \ldots $ be i.i.d.\ non-negative random variables such that $W_i > 0$ a.s.\ and $\E{W^{1+\eps}} < \infty$ for some $\eps > 0$. Moreover, we assume that $\E{W_i} = 1$ and write  $S_n = \sum_{i=1}^n W_i$. Then,
there exists  an almost surely finite random variable $Y$ such that	 
\be\label{eq:sumfitnessconv}
\sum_{j=1}^{n-1}\frac{1}{S_j}-\log n \toas Y.
\ee  
\end{lemma}

\begin{proof} We first write
\be\label{eq:split15}
\sum_{j=1}^{n-1}\frac{1}{S_j}-\log n = \sum_{j=1}^{n-1}\frac{j-S_j}{jS_j}+\sum_{j=1}^{n-1}\frac{1}{j}-\log n=:\sum_{j=1}^{n-1}\frac{j-S_j}{jS_j}+E_n,
\ee 
where $E_n$ is deterministic and converges to the Euler-Mascheroni constant. 
Therefore, it suffices to show that the first sum on the right-hand side is almost surely absolutely convergent.	

By the strong law of large numbers and since $\E{W_i} = 1$, there exists a (random) $J$ such that $S_j>\frac 12j$ for all $j\geq J$ almost surely. So, we can bound almost surely,
\be 
\sum_{j=1}^{n-1}\frac{|j-S_j|}{jS_j}\leq \sum_{j=1}^{J-1}\frac{|j-S_j|}{jS_j}+ 2 \sum_{j=J}^{n-1}\frac{|j-S_j|}{j^2}.
\ee
The first term is finite almost surely since each $W_i > 0$ a.s. We now claim that the second term has a finite mean. Namely,
\be 
\E{\sum_{j=J}^{n-1}\frac{|j-S_j|}{j^2}}\leq \sum_{j=1}^{\infty} \frac{1}{j^2}\E{|j-S_j|^{1+\eps}}^{1/(1+\eps)}\leq \sum_{j=1}^\infty \frac{c_\eps}{j^2}j^{1/(1+\eps)},
\ee 
which is summable, where $c_\eps>0$ is a constant and where we use a Zygmund-Marcinkiewicz bound, see~\cite[Corollary 8.2]{Gut13} in the last step. Therefore, the sum on the right-hand side in~\eqref{eq:split15} is 
almost surely (absolutely) convergent, which completes the proof.
\end{proof}

{\cb In the upcoming subsections, we state and prove several propositions related to the maximum conditional mean degree in WRGs, based on the different conditions in Assumption~\ref{ass:weights}. We note that it suffices to state the proofs of the results below for $m=1$ only, as the expected degrees scale linearly with $m$.}

\subsection{Maximum conditional mean degree, \ref{ass:weightgumbel}-\ref{ass:weighttauinf} sub-case}

\begin{proposition}[\cb Max expected degree, \ref{ass:weightgumbel}-\ref{ass:weighttauinf}]\label{prop:condmeantauinf} 
Consider the WRG model as in Definition~\ref{def:wrg} and suppose the vertex-weights satisfy the~\ref{ass:weightgumbel}-\ref{ass:weighttauinf} sub-case in Assumption~\ref{ass:weights}. Then,
\be \label{eq:tauinfty}
\Big(\max_{\inn}\frac{\Ef{}{\zni}}{mb_n\log n},\frac{\log \wt I_n}{\log n}\Big)\toinp (1,0).
\ee 
\end{proposition}

\begin{proof}
Let $\beta\in(0,1)$. It follows that 
\be 
\max_{\inn}\frac{\F_i\sum_{j=i}^{n-1}1/S_j}{b_n \log n}\geq \max_{i\in[n^{1-\beta}]}\frac{\F_i\sum_{j=\lceil n^{1-\beta}\rceil}^{n-1}1/S_j}{b_n \log n}=\max_{i\in[n^{1-\beta}]}\frac{\F_i}{b_{n^{1-\beta}}}\frac{\sum_{j=\lceil n^{1-\beta}\rceil}^{n-1}1/S_j}{\log n}\frac{b_{n^{1-\beta}}}{b_n}.
\ee 
We then note that $b_{n^{1-\beta}}/b_n = \ell ((1-\beta)\log n)/\ell (\log n)\to 1$ as $n$ tends to infinity, since $\ell$ is slowly varying at infinity. Furthermore, the maximum on the right-hand side tends to $1$ in probability and the fraction in the middle converges to $\beta$ almost surely by~\eqref{eq:sumfitnessconv}. Thus, with high probability,
\be \label{eq:liminftauinfty}
\max_{\inn}\frac{\F_i\sum_{j=i}^{n-1}1/S_j}{b_n \log n}\geq \beta,
\ee 
where we note that we can choose $\beta$ arbitrarily close to $1$. Similarly, we obtain an upper bound by setting the range of the sum from $1$ to $n-1$. Since $\max_{\inn}W_i/b_n$ converges to one in probability, combining this lower bound with~\eqref{eq:sumfitnessconv} yields that, wit high probability
\be 
\max_{\inn}\frac{\F_i\sum_{j=i}^{n-1}1/S_j}{b_n \log n}\leq 1+\eta,
\ee 
for any $\eta>0$. Together with~\eqref{eq:liminftauinfty} this yields the first part of~\eqref{eq:tauinfty}. Now, for the second part, let $\eps>0$, and let us write, for $\eta<\eps$, 
\be 
E_n:=\bigg\{\max_{\inn}\frac{\F_i\sum_{j=i}^{n-1}1/S_j}{b_n \log n}\geq 1-\eta\bigg\},
\ee 
which holds with high probability by the above. Then,
\be \ba
\P{\frac{\log \wt I_n}{\log n}>\eps}& =\P{\Big\{\frac{\log \wt I_n}{\log n}>\eps\Big\}\cap E_n}+\P{E_n^c}\\ & \leq \P{\max_{i>n^\eps}\frac{\F_i\sum_{j=i}^{n-1}1/S_j}{b_n \log n}\geq 1-\eta}+\P{E_n^c}.
\ea \ee 
Clearly, the second probability on the right-hand side tends to zero with $n$. What remains to show is that the same holds for the first probability. Via a simple upper bound, where we substitute $j=\lfloor n^\eps\rfloor$ for $j=i$ in the summation, we immediately obtain
\be 
\P{\max_{i>n^\eps}\frac{\F_i\sum_{j=\lfloor n^\eps\rfloor}^{n-1}1/S_j}{b_n \log n}\geq 1-\eta}\to 0,
\ee 
as the maximum over the fitness values scaled by $b_n$ tends to one in probability, and the sum scaled by $\log n$ converges to $1-\eps$ almost surely, so that the product of the two converges to $1-\eps<1-\eta$ in probability, and so the result follows.
\end{proof}

\subsection{Maximum conditional mean degree, \ref{ass:weightgumbel}-\ref{ass:weighttaufin} sub-case}\label{sec:gumbrv}

Before we turn our attention to the maximum conditional mean in-degree in the WRG model for the~\ref{ass:weightgumbel}-\ref{ass:weighttaufin} sub-case, we first inspect the behaviour of maxima of i.i.d.\ vertex-weights in this class in the following lemma.

\begin{lemma}[Almost sure convergence of rescaled maximum vertex-weight]\label{lemma:asconvmax}
Let $(\F_i)_{i\in\N}$ be i.i.d. random variables that satisfy the~\ref{ass:weightgumbel}-\ref{ass:weighttaufin} sub-case in Assumption~\ref{ass:weights}. Then, 
\be 
\max_{\inn}\frac{\F_i}{b_n}\toas 1.
\ee 
\end{lemma}

\begin{proof}
A particular case, when $c_1=a=1,b=0,\tau\in(0,1]$ and the asymptotic equivalence is replaced with an equality, follows directly from~\cite{HofMorSid08}[Lemma $4.1$] when we set $d=1$ in this lemma. The lemma provides an almost sure lower and upper bound for the maximum of $n$ i.i.d.\ random variables with a distribution as described. The leading order term in these bounds is asymptotically equal to $b_n$, from which the statement of the lemma follows.

We observe that Lemma $4.1$ in~\cite{HofMorSid08} can be extended to hold for any $\tau>1$ as well, in which case only lower order terms may need to be adjusted slightly, so that the leading order terms are still asymptotically equivalent to $b_n$. Thus, it remains to show that for any $\tau>0$, we can extend the case $c_1=a=1,b=0$ to any $c_1,a>0,b\in\R$. 

To that end, let $(\F_i)_{i\in\N}$ be i.i.d.\ copies of a random variable $\F$ with a tail distribution as in the~\ref{ass:weightgumbel}-\ref{ass:weighttaufin} sub-case. This implies that there exists a function $\ell$ such that $\ell(x)\to 1$ as $x\to \infty$, and 
\be 
\P{W\geq x}=\ell(x)ax^b \e^{-(x/c_1)^\tau}.
\ee
Then, let $(X_i)_{i\in\N}$ be i.i.d.\ copies of a random variable $X$ with a tail distribution as in the~\ref{ass:weightgumbel}-\ref{ass:weighttaufin} sub-case, with $\ell\equiv  1, a=c_1=1,b=0$, which are also independent of the $\F_i$. As follows from the above, 
\be \label{eq:asmaxconv}
\max_{\inn}\frac{X_i}{b_n}\toas 1.
\ee
Let us write $b_n(X),b_n(\F)$ to distinguish between the respective first-order growth-rate sequences of $X$ and $\F$, respectively. Define the functions $f,h:\R\to \R$ as $f(x):=x(c_1^{-\tau}-(b\log x+\log a)/x^\tau)^{1/\tau}$ and $h(x):=(f(x)^\tau-\log(\ell(x)))^{1/\tau}$. Then, we can write
\be 
\P{\F\geq x}=\ell(x)ax^b\exp\big(-\big(\tfrac{x}{c_1}\big)^{\tau}\big)=\ell(x)\exp(-f(x)^\tau)=\exp(-h(x)^\tau)=\P{X\geq h(x)}.
\ee 
Hence, $W\overset d = h^{\leftarrow}(X)$, where $h^{\leftarrow}$ is the generalised inverse of $h$, defined as $h^\leftarrow(x):=\inf\{y\in \R: h(y)\geq x\}, x\in\R$. We can write $h$ as 
\be \ba 
h(x)&=f(x)\Big(1-\frac{\log(\ell(x))}{f(x)^\tau}\Big)^{1/\tau}\!\!\!\!\\ & =x\Big(c_1^{-\tau}-\frac{b\log x+\log a}{x^\tau}\Big)^{1/\tau}\!\!\Big(1-\frac{\log(\ell(x))}{(x/c_1)^\tau-(b\log x+\log a)}\Big)^{1/\tau}\\
&=:xL(x).
\ea \ee 
Note that $L(x)\to 1/c_1$ as $x$ tends to infinity, so that $h$ is regularly varying at infinity with exponent $1$. \cite[Theorem $1.5.12$]{BinGolTeu87} then provides a slowly-varying function $\wt L$ such that 
\be \label{eq:asympinv}
\lim_{x\to\infty}\wt L(x)L(x\wt L(x))=1,
\ee 
which implies that $h^\leftarrow(x)\sim \wt L(x)x$ and that $\wt L(x)$ converges to $c_1$. Since $h^\leftarrow $ is increasing, we obtain 
\be 
\max_{\inn}\frac{\F_i}{b_n(\F)}=\frac{h^\leftarrow(\max_{\inn}X_i)}{\wt L(\max_{\inn}X_i)\max_{\inn}X_i}\frac{\max_{\inn}X_i}{b_n(X)}\frac{b_n(X)}{b_n(\F)}\wt L(\max_{\inn}X_i)\toas 1,
\ee 
since the maximum over $X_i$ tends to infinity with $n$ almost surely, $b_n(X)/b_n(\F)\sim 1/c_1$, by~\eqref{eq:asmaxconv} and~\eqref{eq:asympinv} and the continuous mapping theorem.
\end{proof}

With this lemma at hand, we now investigate the maximum conditional mean in-degree of the WRG when the vertex-weights satisfy the~\ref{ass:weightgumbel}-\ref{ass:weighttaufin} sub-case. {\cb We state and prove three propositions which consider the maximum expected degree's first order, the second order in a small window around the expected index that attains the maximum expected degree, and the entire second order, respectively.  

\begin{proposition}[Max expected degree, first order, \ref{ass:weightgumbel}-\ref{ass:weighttaufin}]\label{prop:condmeantaufin}
	Consider the WRG model as in Definition~\ref{def:wrg} and suppose the vertex-weights satisfy the~\ref{ass:weightgumbel}-\ref{ass:weighttaufin} sub-case in Assumption~\ref{ass:weights}. Let $\gamma:=1/(\tau+1)$. Then,
	\be\label{eq:taufin1st}
	\Big(\max_{\inn}\frac{\Ef{}{\zni}}{m(1-\gamma)b_{n^\gamma}\log n},\frac{\log \wt I_n}{\log n}\Big)\toas (1,\gamma).
	\ee 
\end{proposition} 
}

\begin{proof}
We start by proving the first-order growth rate of the maximum. We can immediately construct the lower bound
\be \label{eq:1stordlow}
\frac{\max_{\inn}\F_i\sum_{j=i}^{n-1}1/S_j}{(1-\gamma)b_{n^\gamma}\log n}\geq \frac{\max_{i\in[n^\gamma]}\F_i\sum_{j=n^\gamma}^{n-1}1/S_j}{(1-\gamma)b_{n^\gamma}\log n},
\ee 
and the right-hand side converges almost surely to $1$ by~\eqref{eq:sumfitnessconv} and Lemma~\ref{lemma:asconvmax}. For an upper bound, we first define the sequence $(\wt \eps_k)_{k\in\Z_+}$ as 
\be \label{eq:wtepsk}
\wt \eps_k=\frac{\gamma}{2}\Big(1-\Big(\frac{1-\gamma}{1-(\gamma-\wt \eps_{k-1})}\Big)^\tau \Big)+\frac{1}{2}\wt\eps_{k-1},\quad k\geq 1,\qquad \wt \eps_0=\gamma.
\ee 
This sequence is defined in such a way that it is decreasing and tends to zero with $k$, and that the maximum over indices $i$ such that $n^{\gamma-\wt \eps_{k-1}}\leq i\leq n^{\gamma-\wt \eps_k}$ is almost surely bounded away from $1$: For any $k\geq 1$, we obtain the upper bound
\be \ba \label{eq:epskub}
\max_{i\in [n^{\gamma-\wt\eps_k}]}\frac{\F_i\sum_{j=i}^{n-1}1/S_j}{(1-\gamma)b_{n^\gamma}\log n}&=\max_{1\leq j\leq k}\max_{n^{\gamma-\wt \eps_{j-1}}\leq i\leq n^{\gamma-\wt\eps_j}}\frac{\F_i\sum_{j=i}^{n-1}1/S_j}{(1-\gamma)b_{n^\gamma}\log n}\\
&\leq \max_{1\leq j\leq k}\max_{ i\in[ n^{\gamma-\wt\eps_j}]}\frac{\F_i}{b_{n^{\gamma-\wt\eps_j}}}\frac{\sum_{j=n^{\gamma-\wt\eps_{j-1}}}^{n-1}1/S_j}{(1-\gamma)\log n}\frac{b_{n^{\gamma-\wt\eps_j}}}{b_{n^\gamma}},
\ea\ee 
which, using the asymptotics of $b_n$,~\eqref{eq:sumfitnessconv} and Lemma~\ref{lemma:asconvmax} converges almost surely to 
\be \label{eq:eps_limit}
c_k:=\max_{1\leq j\leq k}\frac{1-(\gamma-\wt\eps_{j-1})}{1-\gamma}\Big(\frac{\gamma-\wt\eps_j}{\gamma}\Big)^{1/\tau},
\ee
which is strictly smaller than one by the choice of the sequence $(\wt \eps_k)_{k\geq 0}$. Now, by writing, for some $\eta>0$ to be specified later,
\be 
E_n:=\bigg\{\max_{\inn}\frac{\F_i\sum_{j=i}^{n-1}1/S_j}{(1-\gamma)b_{n^\gamma}\log n}\geq 1-\eta\bigg\},
\ee 
which holds almost surely for all $n$ large by~\eqref{eq:1stordlow}, we obtain, for any $\eps>0$,
\be \ba 
\Big\{\frac{\log \wt I_n}{\log n}<\gamma-\eps\Big\}&\subseteq\Big\{\Big\{\frac{\log \wt I_n}{\log n}<\gamma-\eps\Big\}\cap E_n\Big\}\cup E_n^c\\ & \subseteq \Big\{\max_{i<n^{\gamma-\eps}}\frac{\F_i\sum_{j=i}^{n-1}1/S_j}{(1-\gamma)b_{n^\gamma}\log n}\geq 1-\eta\Big\}\cup E_n^c.
\ea \ee 
The second event in the union on the right-hand side holds for finitely many $n$ only, almost surely. For the first event in the union, we use~\eqref{eq:epskub} for a fixed $k$ large enough such that $\wt\eps_k<\eps$ to obtain
\be \label{eq:maxsubset} 
\Big\{\max_{i<n^{\gamma-\eps}}\frac{\F_i\sum_{j=i}^{n-1}1/S_j}{(1-\gamma)b_{n^\gamma}\log n}\geq 1-\eta\Big\}  \subseteq \Big\{\max_{i<n^{\gamma-\wt \eps_k}}\frac{\F_i\sum_{j=i}^{n-1}1/S_j}{(1-\gamma)b_{n^\gamma}\log n}\geq 1-\eta\Big\}
\ee
If we then choose $\eta$ small enough such that 
\be \ba
c_k&=\max_{1\leq j\leq k}2^{-1/\tau}\Big(\frac{(1-(\gamma-\wt\eps_{j-1}))^\tau(\gamma-\wt\eps_{j-1})}{(1-\gamma)^\tau \gamma}+1\Big)^{1/\tau}\\
&=2^{-1/\tau}\Big(\frac{(1-(\gamma-\wt\eps_{k-1}))^\tau(\gamma-\wt\eps_{k-1})}{(1-\gamma)^\tau \gamma}+1\Big)^{1/\tau}<1-\eta,
\ea \ee 
which is possible due to the fact that the expression on the left of the second line is increasing to $1$ in $k$, we find that the event on the right-hand side of~\eqref{eq:maxsubset} holds for finitely many $n$ only. Thus, almost surely, the event $\{\log( \wt I_n)/\log n<\gamma-\eps\}$ holds for finitely many $n$ only, irrespective of the value of $\eps$. With a similar argument, and using a sequence $(\eps_k)_{k\in\Z_+}$, defined as 
\be \label{eq:epsk}
\eps_k= \frac{1-\gamma}{2}\Big(1-\Big(\frac{\gamma+\eps_{k-1}}{\gamma}\Big)^{-1/\tau}\Big)+\frac{1}{2}\eps_{k-1},\quad k\geq 1,\qquad \eps_0=1-\gamma,
\ee 
we find that the maximum is not obtained at $n^{\gamma+\eps}\leq i\leq n$ for any $\eps>0$ almost surely as well, which proves the second part of~\eqref{eq:taufin1st}. This also allows for a tighter upper bound of the maximum. On the event that the maximum is obtained at an index $i$ such that $n^{\gamma-\eps}\leq i\leq n^{\gamma+\eps}$, 
\be \ba
\frac{\max_{\inn}\F_i\sum_{j=i}^{n-1}1/S_j}{(1-\gamma)b_{n^\gamma}\log n}& =\max_{n^{\gamma-\eps}\leq i\leq n^{\gamma+\eps}}\frac{ \F_i\sum_{j=i}^{n-1}1/S_j}{(1-\gamma)b_{n^\gamma}\log n}\\ & \leq \max_{i\in[ n^{\gamma+\eps}]}\frac{\F_i}{b_{n^{\gamma+\eps}}}\frac{\sum_{j=n^{\gamma-\eps}}^{n-1}1/S_j}{(1-\gamma)\log n}\frac{b_{n^{\gamma+\eps}}}{b_{n^\gamma}},
\ea \ee 
which, again using the asymptotics of $b_n$,~\eqref{eq:sumfitnessconv} and Lemma~\ref{lemma:asconvmax} converges almost surely to $(1+\eps/(1-\gamma))(1+\eps/\gamma)^{1/\tau}$. This upper bound decreases to $1$ as $\eps$ tends to zero, so that the upper bound can be chosen arbitrarily close to $1$ by choosing $\eps$ sufficiently small. Hence, the left-hand side exceeds $1+\delta$, for any $\delta>0$, only finitely many times. As the event on which this upper bound is constructed holds almost surely eventually for all $n$,  for any fixed $\eps>0$, the first part of~\eqref{eq:taufin1st} follows and which concludes the proof.
\end{proof}

{\cb The previous proposition shows that the size of the maximum conditional mean degree is roughly $m(1-\gamma)b_{n^\gamma}\log n$ and attained at a vertex with label of the order $n^{\gamma+o(1)}$. The next proposition studies the second-order growth rate of the maximum degree in a small window around $n^\gamma$. 

\begin{proposition}[Max expected degree, partial second order, \ref{ass:weightgumbel}-\ref{ass:weighttaufin}]\label{prop:condmeantaufin2ndpart}
	Consider the \\ same conditions as Proposition~\ref{prop:condmeantaufin}. Moreover, recall the sets $C_n$ from~\eqref{eq:cnin}, let $\ell$ be a strictly positive function such that $\lim_{n\to\infty}\log(\ell(n))^2/\log n=\zeta_0$ for some $\zeta_0\in[0,\infty)$, and let $\Pi$ be a Poisson point process on $(0,\infty)\times \R$ with intensity measure $\nu(\d t,\d x):=\d t\times \e^{-x}\d x$. Then, for any $0<s<t<\infty$,
	\be \label{eq:taufin2ndpart}
	\max_{i\in C_n(\gamma,s,t,\ell)}\frac{\Ef{}{\zni}-m(1-\gamma)b_{n^\gamma}\log n}{m(1-\gamma)a_{n^\gamma}\log n}\toindis \max_{\substack{(v,w)\in\Pi\\ v\in(s,t)}}w-\log v-\frac{\zeta_0(\tau+1)^2}{2\tau}.
	\ee
\end{proposition}

}

\begin{proof}
For ease of writing, we omit the arguments and write $C_n :=C_n(\beta,s,t,\ell)$. We use results from extreme value theory regarding the convergence of particular point processes to obtain the results. Let us define the point process 
\be \label{eq:pingumbel}
\Pi_n:=\sum_{i=1}^n \delta_{(i/n,(\F_i-b_n)/a_n)}.
\ee 
By~\cite{Res13}, when the $\F_i$ are i.i.d.\ random variables in the Gumbel maximum domain of attraction (which is the case for the~\ref{ass:weightgumbel}-\ref{ass:weighttaufin} sub-case), then  the weak limit of $\Pi_n$ is $\Pi$, a PPP on $(0,\infty)\times (-\infty,\infty]$ with intensity measure $\nu(\d t,\d x)=\d t \times \e^{-x}\d x$. Here, we understand the topology on $(-\infty,\infty]$ such that sets of the form $[a,\infty]$ for $a \in \R$ are compact and we are crucially using that the measure $e^{-x} \d x$ is finite on these compact sets.

Rather than considering the time-scale $n$ and all $\inn$, we consider the time-scale $\ell(n)n^\gamma$ and $i\in C_n$, and show that the rescaled conditional expected in-degrees can be written as a continuous functional of the point process $\Pi_{\ell(n)n^\gamma}$ with vanishing error terms. Thus, we write
\be \ba
\frac{ \F_i\sum_{j=i}^{n-1}1/S_j-b_{\ell(n)n^\gamma}\log (n^{1-\gamma}/\ell(n))}{a_{\ell(n)n^\gamma}\log (n^{1-\gamma}/\ell(n))}   
&  =\frac{\F_i-b_{\ell(n)n^\gamma}}{a_{\ell(n)n^\gamma}}\frac{\sum_{j=i}^{n-1}1/S_j}{\log(n^{1-\gamma}/\ell(n))}-\log\Big(\frac{i}{\ell(n)n^\gamma}\Big)\\
& \ +\frac{b_{\ell(n)n^\gamma}}{a_{\ell(n)n^\gamma}\log( n^{1-\gamma}/\ell(n))}\Big(\sum_{j=i}^{n-1}\frac{1}{S_j}-\log\Big(\frac{n}{i}\Big)\Big)\\
& \ -\Big(\frac{b_{\ell(n)n^\gamma}}{a_{\ell(n)n^\gamma}\log(n^{1-\gamma}/\ell(n))}- 1\Big)\log(i/\ell(n)n^\gamma).
\ea\ee 
We then let, for $0<s<t<\infty, f \in \R$,
\be
\wt C_n(f):=\{i\in C_n: (\F_i-b_{\ell(n)n^\gamma})/a_{\ell(n)n^\gamma} \geq f\}.
\ee 
Then, for $C_n$ (as well as $\wt C_n(f)$,
\be\ba \label{eq:gumbelmaxdiff}
\bigg|\max_{i\in C_n}&\frac{ \F_i\sum_{j=i}^{n-1}1/S_j-b_{\ell(n)n^\gamma}\log( n^{1-\gamma}/\ell(n))}{a_{\ell(n)n^\gamma}\log (n^{1-\gamma}/\ell(n))} \\
& \hspace{2cm} -\max_{i\in C_n}\bigg(\frac{(\F_i-b_{\ell(n)n^\gamma})\sum_{j=i}^{n-1}1/S_j}{a_{\ell(n)n^\gamma}\log(n^{1-\gamma}/\ell(n))}-\log\Big(\frac{i}{\ell(n)n^\gamma}\Big)\bigg)\bigg|\\
& \leq \frac{b_{\ell(n)n^\gamma}}{a_{\ell(n)n^\gamma}\log( n^{1-\gamma}/\ell(n))}\max_{i\in C_n}\Big|\sum_{j=i}^{n-1}1/S_j-\log(n/i)\Big|\\
& \hspace{2cm}+\Big|\frac{b_{\ell(n)n^\gamma}}{a_{\ell(n)n^\gamma}\log(n^{1-\gamma}/\ell(n))}-1 \Big|\max_{i\in C_n}|\log(i/\ell(n)n^\gamma)|.
\ea\ee
Since $\lim_{n\to\infty}\log(\ell(n))/\log n=0$, it immediately follows that $b_{\ell(n)n^\gamma}\sim b_{n^\gamma}$, $a_{\ell(n)n^\gamma}\sim a_{n^\gamma}$, $\log(n^{1-\gamma}/\ell(n))\sim (1-\gamma)\log n$, so that the fraction on the third line and the first term on the fourth line tend to one and zero, respectively. It also follows from~\eqref{eq:sumfitnessconv} that $\sum_{j=i}^{n-1}1/S_j-\log(n/i)$ converges almost surely for any fixed $i\in\N$, so the maximum on the second line tends to zero almost surely, as the sequence in the absolute value is a Cauchy sequence almost surely (and all $i\in C_n$ tend to infinity with $n$). Finally, we can bound the maximum on the last line by $\max\{|\log t|,|\log s|\}$, so that the left-hand side converges to zero almost surely.

From the convergence of $\Pi_n \toindis \Pi$, it follows that
\be 
\Big(i/(\ell(n)n^\gamma),\frac{\F_i-b_{\ell(n)n^\gamma}}{a_{\ell(n)n^\gamma}}\Big)_{i\in\wt C_n}\toindis (v,w)_{\!\!\!\substack{(v,w)\in\Pi\\ v\in[s,t],w \geq f}},
\ee
on the space of point measures equipped with  the vague topology. It is straight-forward to extend this convergence (using that $[s,t]\times[f,\infty]$ is a compact set) to 
show that 
\be 
\Big(i/(\ell(n)n^\gamma),\frac{\F_i-b_{\ell(n)n^\gamma}}{a_{\ell(n)n^\gamma}}, \frac{\sum_{j=i}^{n-1}1/S_j}{\log(n^{1-\gamma}/\ell(n))}\Big)_{i\in\wt C_n}\toindis (v,w,1)_{\!\!\!\substack{(v,w)\in\Pi\\ v\in[s,t],w\geq f}},
\ee
Hence, the continuous mapping theorem   yields
\be \label{eq:wtcnconv}
\max_{i\in \wt C_n}\frac{\F_i-b_{\ell(n)n^\gamma}}{a_{\ell(n)n^\gamma}}\frac{\sum_{j=i}^{n-1}1/S_j}{\log(n^{1-\gamma}/\ell(n))}- \log\Big(\frac{i}{\ell(n)n^\gamma}\Big)\toindis\!\!\! \max_{\substack{(v,w)\in \Pi\\ v\in[s,t],w\in[f,f']}}\!\!\!\Big(w- \log v\Big),
\ee
as element-wise multiplication and
taking the maximum of a finite number of elements is a continuous operation (which uses that by~\cite[Proposition 3.13]{Res13} vague convergence on a compact set is the same as pointwise convergence). 
Now, we intend to show that the same result holds when considering $i\in C_n$, that is, the distributional convergence still holds when omitting the constraint on the size of the $W_i$. Let $\eta>0$ be fixed,
and for any closed $D \subset \R$, let  let $D_\eta:=\{x\in  \R\,|\, \inf_{y\in D}|x-y|\leq\eta\}$ be its $\eta$-enlargement. We define the random variables and events
\be \ba
X_{n,i}&:=\frac{\F_i-b_{\ell(n)n^\gamma}}{a_{\ell(n)n^\gamma}}\frac{\sum_{j=i}^{n-1}1/S_j}{\log( n^{1-\gamma}/\ell(n))}- \log(i/(\ell(n)n^\gamma)),\quad \inn,\\
E_n(\eta)&:=\{|\max_{i\in C_n}X_{n,i}-\max_{i\in \wt C_n (f)}X_{n,i}|<\eta\},\qquad
A_n(\eta):=\{\max_{i\in C_n}X_{n,i}\in D_\eta\},
\ea \ee 
and note that $D_0=D$ by the definition of $D_\eta$. Then,
\be \ba \label{eq:etacond}
\mathbb{P}(& A_n(0))\leq \P{A_n(0)\cap E_n(\eta)}+\P{E_n(\eta)^c}.
\ea \ee 
Note that the first probability can be bounded from above by    
$\P{\max_{i\in \wt C_n(f)}X_{n,i}\in D_\eta}$ and
from~\eqref{eq:wtcnconv} we obtain that
\be \label{eq:0303-1}
\limsup_{n \rightarrow \infty}    \P{\max_{i\in \wt C_n}X_{n,i}\in D_\eta} \leq \mathbb{P}\Big(\max_{\substack{(v,w)\in\Pi\\ v\in[s,t], w \geq f}}w- \log v \in D_\eta\Big).
\ee 
To see that we can remove the restriction that $w \geq f$, we note that
\be \ba    \Big|\max_{\substack{(v,w)\in\Pi\\ v\in[s,t]}}\Big(w- \log v\Big)-\max_{\substack{(v,w)\in\Pi\\ v\in[s,t], w\geq f}}\Big(w- \log v\Big)\Big|&\leq \max\Big\{0,\max_{\substack{(v,w)\in\Pi\\ v\in[s,t], w\leq f}}w- \log v\Big\}\\
&\leq \max\Big\{0,f- \log s\Big\},
\ea \ee 
which tends to zero almost surely when $f\to-\infty$. Hence, using the above we arrive at 
\be \label{eq:0303-2}
\lim_{\eta\downarrow 0}\lim_{f\to -\infty}\limsup_{n\to \infty}\P{\max_{i\in \wt C_n}X_{n,i}\in D_\eta}\leq \mathbb{P}\Big(\max_{\substack{(v,w)\in\Pi\\ v\in[s,t]}}w- \log v \in D\Big).
\ee
The fact that we can take the limit $\eta \downarrow 0$ follows, since the properties of the Poisson point process imply that the maximum does not hit the boundary of $D$ almost surely.

What remains is to show that the second probability on the right-hand side of~\eqref{eq:etacond} tends to zero.
We bound  
\be\ba \label{eq:extendbound}
\Big|\max_{i\in C_n}X_{n,i}-\max_{i\in \wt C_n(f)}\!\!\!X_{n,i}\Big|
\leq     \max\Big\{0,\max_{i\in C_n \setminus \wt C_n(f)}X_{n,i}\Big\}. 
\ea\ee 
As we intend to let $f$ go to $-\infty$, we can assume $f<0$. Then, as all the terms $(\F_i-b_{\ell(n)n^\gamma})/a_{\ell(n)n^\gamma}$ are negative when $i\in C_n\backslash \wt C_n(f)$, we obtain the upper bound
\be 
\max\bigg\{0,f\frac{\sum_{j=t\ell(n)n^\gamma}^{n-1}1/S_j}{\log( n^{1-\gamma}/\ell(n))}- \log s\bigg\}\toas \max\Big\{0,f- \log s\Big\},
\ee
as $n$ tends to infinity. Then, as $f$ tends to $-\infty$, the right-hand side tends to zero. So, this term tends to zero almost surely as $n\to\infty$ and then $f\to-\infty$. This concludes that the left-hand side of~\eqref{eq:extendbound} converges to zero in probability, and therefore the second probability on the right-hand side of~\eqref{eq:etacond} converges to zero as $n\to\infty$, then $f\to-\infty$ for any $\eta>0$. Combining this with~\eqref{eq:0303-2}, we obtain from the Portmanteau lemma 
\be 
\max_{i\in C_n(\gamma,s,t,\ell(n))}\frac{\F_i-b_{\ell(n)n^\gamma}}{a_{\ell(n)n^\gamma}}\frac{\sum_{j=i}^{n-1}1/S_j}{\log( n^{1-\gamma}/\ell(n))}- \log(i/(\ell(n)n^\gamma))\toindis \max_{\substack{(v,w)\in\Pi\\ v\in[s,t]}}w- \log v.
\ee 
Hence, together with~\eqref{eq:gumbelmaxdiff} and Slutsky's theorem, it follows that 
\be 
\max_{i\in C_n(\gamma,s,t,\ell(n))}\frac{ \F_i\sum_{j=i}^{n-1}1/S_j-b_{\ell(n)n^\gamma}\log (n^{1-\gamma}/\ell(n))}{a_{\ell(n)n^\gamma}\log (n^{1-\gamma}/\ell(n))}\toindis \max_{\substack{(v,w)\in\Pi\\ v\in (s,t)}}w- \log v, 
\ee 
so that the same results hold for the re-scaled maximum conditional mean degree. What remains is to show that the same result is obtained when $\ell(n)n^\gamma$ is replaced with $n^\gamma$ in the first and second-order rescaling, from which~\eqref{eq:taufin2ndpart} follows. To obtain this, we show that 
\be \ba\label{eq:gammacond}
\lim_{n\to\infty}\frac{(1-\gamma)a_{n^\gamma}\log n}{a_{\ell(n)n^\gamma}\log(n^{1-\gamma}/\ell(n))}&=1,\\ \lim_{n\to\infty}\frac{b_{\ell(n)n^\gamma}\log(n^{1-\gamma}/\ell(n))-(1-\gamma)b_{n^\gamma}\log n}{(1-\gamma)a_{n^\gamma}\log n}&=-\frac{\zeta_0(\tau+1)^2}{2\tau},
\ea\ee 
after which the convergence to types theorem yields the required result~\cite[Proposition $0.2$]{Res13}.

First, it immediately follows from Remark~\ref{remark:weights} that
\be 
\frac{a_{\ell(n)n^\gamma}\log(n^{1-\gamma}/\ell(n))}{(1-\gamma)a_{n^\gamma}\log n}=\Big(1+\frac{\log (\ell(n))}{\gamma\log n}\Big)^{1/\tau-1}\Big(1-\frac{\log (\ell(n))}{(1-\gamma)\log n}\Big)\to 1,
\ee 
since we assume that $\log(\ell(n))^2/\log n\to \zeta_0$, so that the first condition in~\eqref{eq:gammacond} is satisfied. Then,
\be \ba 
\phantom{0}& b_{\ell(n)n^\gamma}-b_{n^\gamma}\\ 
& = c_1(\gamma \log n)^{1/\tau}\Big[\Big(1+\frac{\log(\ell(n))}{\gamma\log n}\Big)^{1/\tau}-1\Big]\\
&\quad+\frac{c_1}{\tau}(\gamma\log n)^{1/\tau-1}\Big[\Big(1+\frac{\log( \ell(n))}{\gamma\log n}\Big)^{1/\tau-1}-1\Big]\Big(\frac{b}{\tau}\log(\gamma\log n)+b\log c_1+\log a\Big)\\
&\quad+\frac{bc_1}{\tau^2}(\gamma\log n)^{1/\tau-1}\Big(1+\frac{\log(\ell(n))}{\gamma\log n}\Big)^{1/\tau-1}\log\Big(1+\frac{\log(\ell(n))}{\gamma \log n}\Big).
\ea \ee 
Also, 
\be \ba 
b_{\ell(n)n^\gamma}\log(\ell(n))={}&c_1(\gamma\log n)^{1/\tau}\log(\ell(n))\Big(1+\frac{\log(\ell(n))}{\gamma \log n}\Big)^{1/\tau}\\
&+\frac{c_1}{\tau}(\gamma\log n)^{1/\tau-1}\log(\ell(n))\Big(1+\frac{\log (\ell(n))}{\gamma \log n}\Big)^{1/\tau-1}\\
&\qquad \times\Big[\frac{b}{\tau}\log(\gamma\log n)+b\log c_1+\frac{b}{\tau}\log\Big(1+\frac{\log(\ell(n))}{\gamma\log n}\Big)+\log a\Big].
\ea \ee 
Using Taylor expansions for the terms containing $1+\log(\ell(n))/(\gamma\log n)$ in both these expressions and combining them, yields
\be \ba 
b_{\ell(n)n^\gamma}&\log(n^{1-\gamma}/\ell(n))-b_{n^\gamma}(1-\gamma)\log n\\
={}& (b_{\ell (n)n^{\gamma}}-b_{n^\gamma})(1-\gamma)\log n-b_{\ell(n)n^{\gamma}}\log(\ell(n))\\
={}&-\frac{c_1(\tau+1)}{2\tau}(\gamma\log n)^{1/\tau-1}\log(\ell(n))^2+\frac{c_1b}{\tau}(\gamma \log n)^{1/\tau-1}\log(\ell(n))\\
&-c_1\Big(\frac{b}{\tau}\log(\gamma \log n)+b\log c_1+\log a\Big)(\gamma \log n)^{1/\tau-1}\log(\ell(n))+x_n,
\ea\ee
where $x_n$ consists of lower order terms such that $x_n=o((\log n)^{1/\tau-1}\log(\ell(n)))$. Thus, we obtain
\be \ba\label{eq:bndiff}
& \frac{b_{\ell(n)n^\gamma}\log(n^{1-\gamma}/\ell(n))-(1-\gamma)b_{n^\gamma}\log n}{(1-\gamma)a_{n^\gamma}\log n}\\
& \quad\sim -\frac{((\tau+1)\log( \ell(n)))^2}{2\tau\log n}+\frac{b(\tau+1)}{\tau}\frac{\log(\ell(n))}{\log n}\\
&\qquad -\frac{\tau+1}{b}\Big[\frac{1}{\tau}\log(\gamma\log n )+\log c_1+\frac{\log a}{b}\Big]\frac{\log( \ell(n))}{\log n}.
\ea\ee 
Since $(\log \ell(n))^2/\log n$ converges to $\zeta_0\in[0,\infty)$, it follows that the second condition in~\eqref{eq:gammacond} is indeed satisfied, which completes the proof.
\end{proof}

{\cb Finally, the following proposition identifies the `correct' second-order scaling of the maximum conditional mean degree when considering all vertices in the graph. Somewhat surprisingly, this scaling is different and, as opposed to Proposition~\ref{prop:condmeantaufin2ndpart}, the limit is not random but deterministic. 

\begin{proposition}[Max expected degree, second order, \ref{ass:weightgumbel}-\ref{ass:weighttaufin}]\label{prop:condmeantaufin2ndcomp}
	Under the same conditions as in Proposition~\ref{prop:condmeantaufin},
	\be \label{eq:taufin2}
	\max_{\inn}\frac{\Ef{}{\zni}-m(1-\gamma)b_{n^\gamma}\log n}{m(1-\gamma)a_{n^\gamma}\log n\log\log n}\toinp \frac 12.
	\ee
\end{proposition}

}

\begin{proof}
It suffices to study $\max_{\inn}W_i\log(n/i)$ rather than $\max_{\inn}\Ef{}{\zni}$, since
\be\ba\label{eq:Ewznidiff}
\frac{|\max_{\inn}\Ef{}{\zni}-\max_{\inn}W_i\log (n/i)|}{a_n\log n\log\log n}&\leq \frac{\max_{\inn}W_i}{a_n\log n\log\log n}\Big|\sum_{j=i}^{n-1}1/S_j-\log(n/i)\Big|\\
&=\frac{1}{a_n\log n\log\log n}\max_{\inn}W_i|Y_n-Y_i|,
\ea \ee 
where $Y_n:=\sum_{j=1}^{n-1}1/S_j-\log n$. By~\eqref{eq:sumfitnessconv}, $Y_n$ converges almost surely to $Y$, which is almost surely finite as well. Hence, $\sup_{i\in\N}Y_i$ is almost surely finite, too. Since $\max_{\inn} W_i/b_n\toas 1$ by Lemma~\ref{lemma:asconvmax} and $a_n\log n\sim b_n/\tau$,                                                                                                                                                                                                                                                                                                                                                                                                                                                                                                                                                                                                                                                                                                                                                                                                                                                                                                                                                            this yields the upper bound
\be 
\frac{\max_{\inn}W_i}{a_n\log n}\frac{|Y_n|+\sup_{i\in\N}|Y_i|}{\log\log n}\toas 0.
\ee 
It thus suffices to prove that 
\be \label{eq:second_simplified}
\max_{\inn}\frac{W_i\log(n/i)-(1-\gamma)b_{n^\gamma}\log n}{(1-\gamma)a_{n^\gamma}\log n\log\log n}\toinp \frac 12.
\ee 
Therefore, we set	for $i \in [n]$.
\be 
X_{n,i}:=\frac{W_i\log(n/i)-(1-\gamma)b_{n^\gamma}\log n}{(1-\gamma)a_{n^\gamma}\log n\log \log n}.
\ee 
For an upper bound on the maximum of the $X_{n,i}$, we consider different ranges of indices $i$ separately. 
We will concentrate on the range $i \geq n^\gamma$, the case $i \leq n^\gamma$ follows by completely 
analogous arguments. For a lower bound on the maximum, we will choose a convenient range of indices.

Let $\eps\in(0,1-\gamma)$.
First of all, we notice that by the same argument as in the proof of the first line of~\eqref{eq:taufin1st}, there exists a constant $C<1$ (which is similar
to the constant in~\eqref{eq:eps_limit}) such that almost surely
\[ \max_{n^{\gamma+\eps}<i\leq n} W_i \log (n/i) \leq C(1-\gamma)b_{n^\gamma}\log n . \]
It then follows that the rescaled maximum diverges to $-\infty$ almost surely.

As the next step, we consider the range of $n^\gamma \leq i \leq e^{k_n} n^\gamma$, where
$k_n = \sqrt{\log n} \log \log n$. This will turn out to give the main contribution to the maximum of the $X_{n,i}$.
We now fix $x>1/2$ and let $\delta > 0$. Then,
\be \ba \label{eq:maxbound}
\mathbb{P}{}&\bigg(\max_{n^\gamma\leq i\leq \e^{k_n}n^\gamma}\frac{W_i\log(n/i)-(1-\gamma)b_{n^\gamma}\log n}{(1-\gamma)a_{n^\gamma}\log n\log \log n}\leq x\bigg)\\
&=\prod_{i=n^\gamma}^{e^{k_n}n^\gamma}\P{W_i\leq \frac{1-\gamma}{1-\log i/\log n}(b_{n^\gamma}+a_{n^\gamma}x\log \log n)}\\
&\geq \exp\bigg(-(1+\delta)\sum_{i=n^\gamma}^{\e^{k_n}n^\gamma}\P{W\geq \frac{1-\gamma}{1-\log i/\log n}(b_{n^\gamma}+a_{n^\gamma}x\log \log n)}\bigg)\\
&\geq \exp\bigg(-(1+\delta)\sum_{j=1}^{k_n}\sum_{i=\e^{j-1}n^\gamma}^{\e^j n^\gamma}\P{W\geq \frac{1-\gamma}{1-\gamma-(j-1)/\log n}(b_{n^\gamma}+a_{n^\gamma}x\log \log n)}\bigg),
\ea \ee 
where we use that $1-y\geq \e^{-(1+\delta)y}$ for all $y$ sufficiently small and that the tail probability is decreasing to zero, uniformly in $i$, in the last two steps. Since the probability is no longer dependent on $i$, we can also omit the inner sum and replace it by $\lfloor \e^j n^\gamma\rfloor -\lceil \e^{j-1}n^\gamma\rceil\leq (\e-1)\e^{j-1}n^\gamma$. Also using that $\P{W\geq y}\leq (1+\delta)ay^b\exp(-(y/c_1)^\tau)$ for all $y$ sufficiently large, it follows that for any $x\in\R$ and $n$ sufficiently large we obtain the lower bound
\be \ba \label{eq:doubleexp}
\exp\bigg(-{}&(1+\delta)^2a(\e-1)\sum_{j=1}^{k_n} \e^{j-1}n^\gamma\Big(\frac{1-\gamma}{1-\gamma -(j-1)/\log n}(b_{n^\gamma}+a_{n^\gamma}x\log \log n)\Big)^b\\
&\hspace{2.4cm}\times \exp\Big(-\Big(\frac{1}{c_1}\frac{1-\gamma}{1-\gamma -(j-1)/\log n}(b_{n^\gamma}+a_{n^\gamma}x\log \log n)\Big)^\tau\Big)\bigg).
\ea \ee 
We first bound the fraction $(1-\gamma)/(1-\gamma)-(j-1)/\log n)$ from above by $1+\delta$ if $b\geq 0$ and from below by $1$ if $b<0$, which holds uniformly in $j$ for $n$ large, in the outer exponent. Then, when we combine all other terms that contain $j$, we find
\be \label{eq:jterms}
\exp\Big((j-1)-\Big(\frac{(1-\gamma)\log n}{(1-\gamma)\log n-(j-1)}\Big)^\tau\Big(\frac{b_{n^\gamma}+a_{n^\gamma}x\log \log n}{c_1}\Big)^\tau\Big).
\ee 
Define $\e_n:=(1-\gamma)\log n$. By a Taylor expansion we have that there exists a constant $C_\tau > 0$ such
that uniformly for $|y| \leq  \e_n/2$,
\begin{equation}\label{eq:2nd_taylor} 1  + \tau \frac{y}{\e_n} + \frac{\tau(1+\tau)}{2}  \Big( \frac{y}{\e_n}\Big)^2  \leq \Big(\frac{\e_n}{\e_n-y}\Big)^{\tau} 
	\leq 1 + \tau \frac{y}{\e_n}  + C_\tau \Big( \frac{y}{\e_n}\Big)^2 . \end{equation}
We will also need that again by a Taylor expansion and the explicit form of $a_n, b_n$ as stated in Remark~\ref{remark:weights}, we have that
\be\ba 
\Big(\frac{b_{n^\gamma}+a_{n^\gamma}x\log \log n}{c_1}\Big)^\tau
&= \Big(\frac{b_{n^\gamma}}{c_1}\Big)^\tau\Big(1  + \tau  \frac{a_{n^\gamma}}{b_{n^\gamma}}  x\log \log n (1+o(1))\Big)\\
&= \Big(\frac{b_{n^\gamma}}{c_1}\Big)^\tau + x\log \log n (1+o(1)) ,
\ea\ee
and similarly since $\gamma = 1/(1+\tau)$, 
\[ \frac{(b_{n^{\gamma}}/c_1)^\tau}{\e_n} =\frac{1}{(1-\gamma)} \frac{\log n^\gamma}{\log n} + \mathcal O\Big( \frac{\log \log n}{\log n}\Big) 
= \frac{1}{\tau} + \mathcal O\Big( \frac{\log \log n}{\log n}\Big) . \]
Combining all these estimates, we obtain the following upper bound on~\eqref{eq:jterms}:
\begin{equation}\label{eq:interm_step}     \begin{aligned} 
		\exp \Big((j{}&-1)  -\Big(\frac{(1-\gamma)\log n}{(1-\gamma)\log n-(j-1)}\Big)^\tau\Big(\frac{b_{n^\gamma}+a_{n^\gamma}x\log \log n}{c_1}\Big)^\tau\Big)  \\
		\leq \exp\Big({}& (j-1) \\ 
		{}&- \Big(1+ \tau \frac{j-1}{\e_n} + \frac{\tau(1+\tau)}{2} \Big( \frac{j-1}{\e_n}\Big)^2 \Big) \big( (b_{n^\gamma}/c_1)^\tau + x \log \log n (1+o(1))\big) \Big)\\
		= \exp \Big({}& - (b_{n^\gamma}/c_1)^\tau - x (\log \log n)(1+o(1)) \\
		{}& - (j-1)\, \mathcal O\Big(\frac{\log \log n}{\log n}\Big)  
		- \frac{(1+\tau)(j-1)^2  }{2(1-\gamma)\log n}(1+o(1)) \Big) \\
		\leq \exp \big({}& - (b_{n^\gamma}/c_1)^\tau - x (\log \log n)(1+o(1)) + o(1)  \big) ,
\end{aligned}  \end{equation}
where we used in the last step that $j \leq k_n = o(\log/ \log \log n)$ and the last term in the exponent is negative.
Hence, combining this with~\eqref{eq:maxbound} and~\eqref{eq:doubleexp}, we obtain
\[ \begin{aligned} 
	\mathbb{P}{}&\bigg(\max_{n^\gamma\leq i\leq \e^{k_n}n^\gamma}\frac{W_i\log(n/i)-(1-\gamma)b_{n^\gamma}\log n}{(1-\gamma)a_{n^\gamma}\log n\log \log n}\leq x\bigg)\\
	& \geq \exp \big( - (1+\delta)^{2+b\vee 0} a(\e -1) k_n n^\gamma (b_{n^\gamma})^b \exp \{ - (b_{n^\gamma}/c_1)^\tau - x (\log \log n)(1+o(1)) \big) \\
	& =   \exp\big(-(1+\delta)^{2+b\vee 0}(\e-1)k_n (\log n)^{-x (1+o(1))}\big),
\end{aligned} \]
where we used that $n \P{W > b_n} \sim a n b_n^b e^{-(b_n/c_1)^\tau} \sim 1$ (see e.g.\ \cite[Equation (1.1')]{Res13} with $x = 0$). Hence, by our choice of $k_n = \sqrt{\log n} \log \log n$ and $x > 1/2$, the  latter expression converges to $1$
and we have shown that for any $\eta > 0$, with high probability, 
\[ \max_{n^\gamma\leq i\leq \e^{k_n}n^\gamma} X_{n,i} \leq \frac 12 + \eta. \]
Next, we consider an upper bound on the maximum for the range $e^{k_n}n^\gamma \leq i \leq n^{\gamma + \eps}$, where  $\eps \in (0,  1- \gamma)$ and $k_n = \sqrt{\log n}\log \log n$ are as above. Again, we take $x \in \R$ and use the same idea as in the first step. This time,  however, we need to be more careful in the intermediate step~\eqref{eq:interm_step}. 
For $k_n \leq j \leq \eps \log n$, 
we obtain an upper bound on the expression in~\eqref{eq:jterms}
\[ \begin{aligned} &\exp \Big((j-1)  -\Big(\frac{(1-\gamma)\log n}{(1-\gamma)\log n-(j-1)}\Big)^\tau\Big(\frac{b_{n^\gamma}+a_{n^\gamma}x\log \log n}{c_1}\Big)^\tau\Big)  \\
	& \leq \exp \Big( - (b_{n^\gamma}/c_1)^\tau - x (\log \log n)(1+o(1)) \\
	& \hspace{3cm}- (j-1) \mathcal O\Big(\frac{\log \log n}{\log n}\Big)  
	- \frac{(1+\tau)(j-1)^2  }{2(1-\gamma)\log n}(1+o(1)) \Big) \\
	& \leq \exp \Big( - (b_{n^\gamma}/c_1)^\tau - x (\log \log n)(1+o(1))  + C_1 (j-1) \frac{\log \log n}{\log n} 
	- C_2 (j-1)^2 \frac{1}{\log n} \Big) , 
\end{aligned} \]
for suitable  constants $C_1, C_2 > 0$. We now note that the right-hand side is decreasing in $j$ as long as
$j > (C_1)/(2C_2) \log \log n$. However, $k_n \gg \log \log n$, so that we obtain the following upper bound on the previous display that holds
uniformly for $k_n \leq j \leq \eps\log n$,
\[    \exp \Big( - (b_{n^\gamma}/c_1)^\tau - x (\log \log n)(1+o(1))  + C_1 k_n \frac{\log \log n}{\log n} 
- C_2 k_n^2 \frac{1}{\log n} \Big) . \]
Using this bound in the same way as before (following the analogous steps as in~\eqref{eq:maxbound} and~\eqref{eq:doubleexp}) we obtain for any $x \in \R$,
\[ \begin{aligned} 
	\mathbb{P}{}&\bigg(\max_{\e^{k_n} n^\gamma\leq i\leq  n^{\gamma+\eps}}\frac{W_i\log(n/i)-(1-\gamma)b_{n^\gamma}\log n}{(1-\gamma)a_{n^\gamma}\log n\log \log n}\leq x\bigg)\\
	& \geq \exp \Big( - (1+\delta)^{2+b\vee 0}(\e -1) \eps(\log n)^{1- x (1+o(1))}
	\exp \Big( C_1 k_n \frac{\log \log n}{\log n} 
	- C_2 k_n^2 \frac{1}{\log n} \Big) 
	\Big) 
\end{aligned} \]
Finally, since $k_n = \sqrt{\log n} \log \log n$, the right-hand side converges to $1$. Therefore, we have shown that
for any $x \in \R$, with high probability 
\[ \max_{\e^{k_n} n^\gamma\leq i\leq n^{\gamma+\eps}} X_{n,i} \leq  x. \]

In a similar way to the upper bound, we can construct a lower bound on the maximum by restricting 
to the indices $1 \leq i \leq k_n'$, where $k_n' = \sqrt{\log n} / \log \log n $. 
In this case, we consider $x < 1/2$ and use $k_n'$ instead of $k_n$ in the argument above. We omit the $(1+\delta)$ term in~\eqref{eq:maxbound}, bound the probability from below using $(1-\delta)$ rather than $(1+\delta)$, use the upper bound in~\eqref{eq:2nd_taylor} and obtain 
thus
\be \ba 
\mathbb{P}{}&\bigg(\max_{n^\gamma\leq i\leq \e^{k_n'}n^\gamma}\frac{W_i\log(n/i)-(1-\gamma)b_{n^\gamma}\log n}{(1-\gamma)a_{n^\gamma}\log n\log \log n}\leq x\bigg)\\
& \leq     \exp\Big(\!-(1-\delta)^{1+b\wedge 0}(\e-1)k_n' (\log n)^{- x(1+o(1))}\exp\Big(\! -\frac{C_\tau}{2\tau(1-\gamma)}\frac{(k_n')^2}{\log n}(1+o(1))\Big)\Big).
\ea \ee 
The latter term converges to zero as $x < 1/2$ and therefore, we have shown that
for any $\eta > 0$, with high probability
\[ \max_{i \in [n]} X_{n,i} \geq \max_{1 \leq i \leq e^{k_n'} n^\gamma} X_{n,i} \geq \frac 12 - \eta. \]
This completes the argument for all $n^\gamma \leq i \leq n$. The argument for $1 \leq i \leq n^\gamma$ 
works completely analogously, so that we have shown~\eqref{eq:second_simplified}, which completes the proof.
\end{proof}

{\cb 

\subsection{Maximum conditional mean degree, \ref{ass:weightgumbel}-\ref{ass:weighttau0} sub-case}

In this section we state and prove two propositions that study the behaviour of the maximum conditional mean in-degree in WRGs, when the vertex-weights satisfy the \ref{ass:weightgumbel}-\ref{ass:weighttau0} sub-case of Assumption~\ref{ass:weights}. 
The first provides two results: the first-order asymptotics of the maximum conditional mean in-degree and the label of the vertex that attains this maximum. Its second result shows the second-order scaling of the maximum conditional mean in-degree among all vertices in a small window around $t_nn$, with a random limit. As is the case in Section~\ref{sec:gumbrv}, we see that this second-order scaling is in fact incorrect when we consider the full range of all vertices, as shown in the second proposition. Here, we also observe a phase transition at $\tau=3$, where the second-order scaling changes. 

\begin{proposition}[Max expected degree, first and second order, \ref{ass:weightgumbel}-\ref{ass:weighttau0}]\label{prop:condmeantau0}
	Consider the WRG model as in Definition~\ref{def:wrg} and suppose the vertex-weights satisfy the~\ref{ass:weightgumbel}-\ref{ass:weighttau0} sub-case in Assumption~\ref{ass:weights} and let $t_n:=\exp(-\tau \log n/\log(b_n))$. Then,
	\be\label{eq:tau0}
	\Big(\max_{\inn}\frac{\Ef{}{\zni}}{mb_{t_nn}\log(1/t_n)},\frac{\log \wt I_n}{\log n}\Big)\toinp (1,1).
	\ee
	Moreover, recall $C_n$ from~\eqref{eq:cnin} and let $\Pi$ be a Poisson point process on $(0,\infty)\times \R$ with intensity measure $\nu(\d t,\d x):=\d t \times \e^{-x}\d x$. Then, for any $0<s<t<\infty$,
	\be\label{eq:rav2ndmean} 
	\max_{i\in C_n(1,s,t,t_n)}\frac{\Ef{}{\zni}-mb_{t_nn}\log(1/t_n)}{ma_{t_nn}\log(1/t_n)}\toindis \max_{\substack{(v,w)\in\Pi\\ v\in(s,t)}}w-\log v,
	\ee 
\end{proposition}
}
\begin{proof}
First, we show that, similar to~\eqref{eq:Ewznidiff},
\be \label{eq:conc}
\Big|\max_{\inn}\frac{\F_i\sum_{j=i}^{n-1}1/S_j}{b_n}-\max_{\inn}\frac{\F_i}{b_n}\log(n/i)\Big|\toinp 0,
\ee 
so that in what follows we can work with the rightmost expression in the absolute value rather than the leftmost. This directly follows from writing the absolute value as 
\be \ba 
\Big|\max_{\inn}\frac{\F_i\sum_{j=i}^{n-1}1/S_j}{b_n}-\max_{\inn}\frac{\F_i}{b_n}\log(n/i)\Big| & \leq \max_{\inn}\frac{\F_i}{b_n}\Big|\sum_{j=i}^{n-1}1/S_j-\log(n/i)\Big|\\\ & =\max_{\inn}\frac{\F_i}{b_n}|Y_n-Y_i|,
\ea \ee 
where $Y_n:=\sum_{j=1}^{n-1}1/S_j-\log n$, which converges almost surely by~\eqref{eq:sumfitnessconv}. We then split the maximum into two parts to obtain the upper bound, for any $\gamma\in(0,1)$,
\be 
\max_{i\in [n^\gamma]}\frac{\F_i}{b_{n^\gamma}}(|Y_n|+\sup_{j\geq 1}|Y_j|)\frac{b_{n^\gamma}}{b_n}+\max_{n^\gamma\leq i \leq n}\frac{\F_i}{b_n}\max_{n^\gamma \leq i \leq n}|Y_n-Y_i|.
\ee 
The first maximum converges to $1$ in probability, the term in the brackets converges almost surely and the second fraction tends to zero, as we recall from Remark~\ref{remark:weights} that $b_n=g(\log n)$ with $g$ a rapidly-varying function at infinity. This implies, for any $\gamma\in(0,1)$, by the definition of a rapidly-varying function, that $b_{n^\gamma}/b_n=g(\gamma\log n)/g(\log n)$ converges to zero with $n$. Similarly, the second maximum converges to $1$ in probability and the third maximum tends to zero almost surely, as $Y_n$ is a Cauchy sequence almost surely. In total, the entire expression tends to zero in probability.

For the next part, we define   
\be 
\ell(x):=c_1+c_2x^{-1}\Big(\frac{b}{\tau}\log x +b\log c_1+\log a\Big).
\ee 
Then, as we are working in the~\ref{ass:weightgumbel}-\ref{ass:weighttau0} sub-case in Assumption~\ref{ass:weights},
we can write $b_n  =\exp( (\log n)^{1/\tau} \ell(\log n))$.

Using $t_n$ we can show that for any fixed $r\in\R$ or $r=r(n)$ that does not grow `too quickly' with $n$, $b_{t^r_nn}/b_n\sim \e^{-r}$. Namely, uniformly in $r = r(n) \leq C \log \log (b_n)$ (for any constant $C>0$),
\be \ba \label{eq:bntn}
\frac{b_{t^r_nn}}{b_n}=\exp\bigg({}&(\log n)^{1/\tau}\Big(\Big(1+r\frac{\log t_n}{\log n}\Big)^{1/\tau} {\ell\Big(\log n\Big(1+r\frac{\log t_n}{\log n}\Big)\Big)}-\log n \Big)\bigg)\\
\sim\exp\Big({}&(\log n)^{1/\tau}\Big(\ell\Big(\log n\Big(1+r\frac{\log t_n}{\log n}\Big)\Big)-\ell(\log n)\Big)\\
&+(1/\tau) r\log t_n(\log n)^{1/\tau-1}\ell\Big(\log n\Big(1+r\frac{\log t_n}{\log n}\Big)\Big)\Big), 
\ea\ee 
where we applied a Taylor approximation to $(1+r\log t_n/\log n)^{1/\tau}$, which holds uniformly in $r$ as long as $r=o(\log n/\log t_n)=o(\log b_n)$. It is elementary to show that for such $r$, the first term in the exponent on the last line of~\eqref{eq:bntn} tends to zero. Thus, uniformly in $r \leq C \log \log (b_n)$,
\be \label{eq:btrn}
\frac{b_{t^r_nn}}{b_n}\sim \exp\bigg(-r\frac{\ell(\log n(1+r\log t_n/\log n))}{\ell(\log n)}\bigg)\sim \e^{-r},
\ee 
where the last step follows a similar argument to the one used to show that the first term on the right-hand side of~\eqref{eq:bntn} tends to zero.

We thus note that by~\eqref{eq:btrn} and~\eqref{eq:conc} it suffices to show that 
\be \label{eq:tnscale}
\max_{\inn}\frac{\F_i\log(n/i)}{b_n\log(1/t_n)}\toinp 1/\e, 
\ee 
to prove~\eqref{eq:tau0}. We start by providing a lower bound to the left-hand side of~\eqref{eq:tnscale}. For some fixed $r>0$, we write
\be\ba
\max_{\inn}\frac{\F_i\log(n/i)}{b_n\log(1/t_n)}\geq \max_{i \in[t_n^rn]}\frac{\F_i}{b_{t^r_nn}}\frac{\log(n/(t^r_nn))}{\log(1/t_n)}\frac{b_{t^r_nn}}{b_n}=\max_{i \in[t_n^rn]}\frac{\F_i}{b_{t^r_nn}}r\frac{b_{t^r_nn}}{b_n}.
\ea\ee 
By~\eqref{eq:btrn}, it follows that this lower bound converges in probability to $r\e^{-r}$. To maximise this expression, we choose $r=1$ giving the value $1/\e$ as claimed. For an upper bound, we split the maximum into multiple parts which cover different ranges of indices $i$. First, for ease of writing, let us denote
\be 
X_{n,i}:=\frac{\F_i\log(n/i)}{b_n\log (1/t_n)}.
\ee
Fix $\eps > 0$, then  set $N =\lceil 2 \log \log (b_n) / \eps \rceil$, and define
\[ r_0 = \e^{-1}, \mbox{ and } \quad r_i = r_0 + \eps i \quad    \mbox{for } i = 1, \ldots, N. \]
Then,
\be \ba \label{eq:totalmax}
\max_{\inn}X_{n,i}\leq \max\Big\{\max_{ i\in[t_n^{r_N}n]}X_{n,i},
\max_{k = 1, \ldots, N}
\max_{t^{r_{k}}_nn<i\leq t_n^{r_{k-1}}n}X_{n,i},\max_{t^{r_0}_n n< i\leq n}X_{n,i}\Big\}.
\ea\ee 
We now bound each of these three parts separately. We start with the middle term and note that for $k \in \{1, \ldots, N\}$,
\[  \max_{t^{r_{k}}_nn< i\leq t_n^{r_{k-1}}n}X_{n,i}
=  \max_{t^{r_{k}}_nn< i\leq t_n^{r_{k-1}}n}\frac{\F_i\log(n/i)}{b_n\log (1/t_n)}\leq r_k \frac{ b_{ t_n^{r_{k-1}}n }}{b_n} \max_{ t^{r_{k}}_nn< i\leq t_n^{r_{k-1}}n}
\frac{\F_i}{b_{ t_n^{r_{k-1}}n}} . \]
If we now define for $k = 0,\ldots, N$,
\[ A_n(k) :=  \max_{ t^{r_{k+1}}_nn< i\leq t_n^{r_{k}}n}
\frac{\F_i}{b_{ t_n^{r_{k}}n}} , \]
then, by~\eqref{eq:btrn}, we have that
\begin{equation}\label{eq:0807-1} \begin{aligned} \max_{k = 1, \ldots, N} 
		\max_{t^{r_{k}}_nn< i\leq t_n^{r_{k-1}}n}X_{n,i}
		&  \leq  (1+\eps) \max_{k \in[N]} r_k \e^{-r_{k-1}} 
		A_n(k - 1) \\
		&  \leq (1+\eps) \sup_{x \geq 1/\e} x \e^{-x + \eps} \max_{k\in[N]}
		A_n(k-1)  \\
		& = (1+\eps)  \e^{-1 + \eps} \max_{k\in[N]}
		A_n(k-1), \end{aligned}
	. \end{equation}
using as before that $x \mapsto x \e^{-x}$ is maximised at $x = 1$. Similarly, we can bound the the last term in~\eqref{eq:totalmax} as
\begin{equation}\label{eq:0807-2}   \max_{t^{r_0}_n n< i\leq n}X_{n,i}
	\leq r_0 \max_{t^{r_0}_n n< i\leq n} \frac{\F_i}{b_{n}} = \e^{-1} A_n,
\end{equation}
where we recall that $r_0 = 1/\e$  and we set $A_n := \max_{t^{r_0}_n n<i\leq n} \F_i/b_n$. 
Finally, for the first term in~\eqref{eq:totalmax}, we get that
\begin{equation} \label{eq:0807-3}
	\max_{ i\in[t_n^{r_N}n]}X_{n,i}
	\leq \frac{b_{t_n^{r_N} n}}{b_n}  \frac{\log n}{\log(1/t_n)}\max_{i\in[t^{r_N}_nn]}\frac{W_i}{b_{t^{r_N}_nn}}
	\leq \frac{1+\eps}{\tau} e^{-r_N} \log (b_n)  \max_{i\in[t^{r_N}_nn]}\frac{W_i}{b_{t^{r_N}_nn}}= o_\mathbb{P}(1),
\end{equation} 
where we use that $r_N \geq 2 \log \log (b_n)$ by definition.

Combining~\eqref{eq:totalmax} with the estimates in
~\eqref{eq:0807-1}-\eqref{eq:0807-3}, we obtain 
\be\label{eq:finub} 
\max_{i \in [n]} X_{n,i} \leq (1+\eps)\e^{-1 + \eps} \max\Big\{\max_{k\in[N]}A_n(k-1),A_n\Big\}. 
\ee
Since $\eps > 0$ is arbitrary, it suffices to show that the maximum on the right-hand side is bounded by $1+\eps$ with high probability. As the random variables follow a distribution as in the~\ref{ass:weightgumbel}-\ref{ass:weighttau0} case in Assumption~\ref{ass:weights}, we can write using a union bound and $C>0$ large,
\be \ba 
\mathbb{P}\Big(\max_{\inn}&\frac{\F_i}{b_n}\geq 1+\eps\Big)\leq  Cn\log((1+\eps)b_n)^b\exp(-( \log((1+\eps)b_n)/c_1)^\tau)\\
&=Cn\log(b_n)^b \Big(1+\frac{\log(1+\eps)}{\log(b_n)}\Big)^b\exp\Big(-(\log(b_n)/c_1)^\tau\Big(1+\frac{\log(1+\eps)}{\log(b_n)}\Big)^\tau\Big).
\ea\ee 
We now use the expression of $b_n$ as in the~\ref{ass:weightgumbel}-\ref{ass:weighttau0} case in Assumption~\ref{ass:weights} to obtain the upper bound
\be\ba
\wt C \log(b_n)^b \exp\Big(\log n\Big(1-\Big(1+\frac{(b/\tau)\log \log n+b\log c_1+\log \tau}{\tau \log n}\Big)^\tau\Big(1+\frac{\log(1+\eps)}{\log(b_n)}\Big)^\tau\Big)\Big),
\ea \ee 
where $\wt C>0$ is a suitable constant. Using a Taylor approximation on the terms in the exponent and using the asymptotics of $\log (b_n)$, we find an upper bound 
\be \label{eq:probub}
K_1(\log n)^{b/\tau}\exp(-K_2(\log n)^{1-1/\tau}),
\ee 
for some constants $K_1,K_2>0$ and $n$ sufficiently large. Note that this expression tends to zero as $\tau>1$. Now, we aim to apply this bound to the maximum in~\eqref{eq:finub}. First, a union bound yields
\be \ba 
\mathbb{P}\Big(\max\Big\{\max_{k\in[N]}A_n(k-1),A_n\Big\}\geq 1+\eps\Big)& \leq \sum_{k=0}^{N-1}\!\mathbb{P}\Big(\max_{i\in[t^{r_k}_nn]}\frac{\F_i}{b_{t^{r_k}_nn}}\geq 1+\eps\Big)\!\\ & \quad\hspace{2cm} +\mathbb{P}\Big(\!\max_{\inn}\frac{\F_i}{b_n}\geq 1+\eps\Big).
\ea\ee 
The last term tends to zero with $n$. For the sum we use~\eqref{eq:probub} and note that this upper bound tends to zero slowest for $k=N-1$, so that we obtain the upper bound 
\be \ba 
\sum_{k=0}^{N-1}\mathbb{P}\Big(\max_{i\in[t^{r_k}_nn]}\F_i/b_{t^{r_k}_nn}\geq 1+\eps\Big)&\leq NK_1\log(t^{r_{N-1}}_nn)^{b/\tau}\exp(-K_2\log(t^{r_{N-1}}_nn)^{1-1/\tau})  \\
&\leq K_3 \log\log(b_n)(\log n)^{b/\tau}\exp(-K_4 (\log n)^{1-1/\tau}),
\ea\ee 
for some constants $K_3,K_4$, since $r_{N-1}=\mathcal{O}(\log\log (b_n))$, which again tends to zero with $n$ as $\tau>1$.

Finally, we prove the convergence of $\log(\wt I_n)/\log n$. Let $\eta\in(0,1)$. Then, the event
\be 
E_n:=\bigg\{\max_{\inn}\frac{\Ef{}{\zni}}{ b_{t_nn} \log(1/t_n)}\geq \eta\bigg\}
\ee 
holds with high probability by the above. Using this and~\eqref{eq:conc} yields, for $\eps>0$, 
\be \ba
\mathbb P\Big(\frac{\log \wt I_n}{\log n}<1-\eps\Big)&\leq\mathbb P\Big(\Big\{\frac{\log \wt I_n}{\log n}<1-\eps\Big\}\cap E_n\Big)+\P{E_n^c}\\ 
&\leq \mathbb P\Big(\max_{i< n^{1-\eps}}\frac{\F_i\log (n/i)}{b_{t_nn} \log(1/t_n)}\geq \eta\Big)+o(1).
\ea \ee 
The probability on the right-hand side can be bounded from above by 
\be \label{eq:intau0prob}
\P{\max_{i\leq n^{1-\eps}}\frac{\F_i}{b_{n^{1-\eps}}}\frac{b_{n^{1-\eps}}\log(b_n)}{b_{t_nn}}\geq \tau\eta}.
\ee 
Now,
\be\ba 
\frac{b_{n^{1-\eps}}\log(b_n)}{b_{t_nn}}\sim\exp\Big(1+(\log n)^{1/\tau} \ell(\log n)\Big((1-\eps)^{1/\tau} \frac{\ell((1-\eps)\log n)}{\ell(\log n)}-1\Big)\Big)\log(b_n),
\ea\ee 
which, since $\ell$ is a slowly-varying function at infinity and $(1-\eps)^{1/\tau}<1$, tends to zero with $n$. As the maximum in~\eqref{eq:intau0prob} tends to $1$ in probability, we obtain that the probability in~\eqref{eq:intau0prob} tends to zero with $n$.

To prove~\eqref{eq:rav2ndmean} we use a similar argument as in the proof of Proposition~\ref{prop:condmeantaufin}, as distributions satisfying the~\ref{ass:weightgumbel}-\ref{ass:weighttau0} sub-case also fall in the Gumbel MDA. Namely, we can write
\be \ba 
\frac{ \F_i\sum_{j=i}^{n-1}1/S_j-b_{t_nn}\log(1/t_n)}{a_{t_nn}\log(1/t_n)}={}&\frac{\F_i-b_{t_nn}}{a_{t_nn}}\frac{\sum_{j=i}^{n-1}1/S_j}{\log(1/t_n)}-\log\Big(\frac{i}{t_nn}\Big)\\
&+\frac{b_{t_nn}}{a_{t_nn}\log(1/t_n)}\Big(\sum_{j=i}^{n-1}1/S_j-\log(n/i)\Big)\\
&-\Big(\frac{b_{t_nn}}{a_{t_nn}\log(1/t_n)}-1\Big)\log\Big(\frac{i}{t_nn}\Big),
\ea\ee 
so that 
\be \ba \label{eq:ravdiff}
&\bigg|\max_{i\in C_n}{}\frac{ \F_i\sum_{j=i}^{n-1}1/S_j-b_{t_nn}\log(1/t_n)}{a_{t_nn}\log(1/t_n)}-\max_{i\in C_n}\Big(\frac{\F_i-b_{t_nn}}{a_{t_nn}}\frac{\sum_{j=i}^{n-1}1/S_j}{\log(1/t_n)}-\log\Big(\frac{i}{t_nn}\Big)\Big)\bigg|\\
&\quad\leq \frac{b_{t_nn}}{a_{t_nn}\log(1/t_n)}\max_{i\in C_n}\bigg|\sum_{j=i}^{n-1}1/S_j-\log(n/i)\bigg|+\Big|\frac{b_{t_nn}}{a_{t_nn}\log(1/t_n)}-1\Big|\max_{i\in C_n}  \Big| \log\Big(\frac{i}{t_nn}\Big)  \Big|,
\ea \ee 
where we omit the arguments of $C_n(1,s,t,t_n)$ for brevity. We now note that, in the~\ref{ass:weighttau0} sub-case, $a_n:=c_2 (\log n)^{1/\tau-1}b_n$, which yields
\be \ba
\frac{b_{t_nn}}{a_{t_nn}\log(1/t_n)}& =\Big(c_2\log(t_nn)^{1/\tau-1}\log(1/t_n)\Big)^{-1}\\ & \sim \Big(c_2(\log n)^{1/\tau-1}\frac{\tau \log n}{\log(b_n)}\Big)^{-1}\sim \Big(\frac{c_1}{\ell(\log n)}\Big)^{-1},
\ea \ee 
which converges to $1$ as $n$ tends to infinity. Here, we use that $c_1=c_2\tau$, and that $\log(b_n)=(\log n)^{1/\tau}\ell(\log n)$, with $\lim_{x\to\infty}\ell(x)=c_1$. It follows, with a similar argument as in the proof of Proposition~\ref{prop:condmeantaufin}, that the right-hand side of~\eqref{eq:ravdiff} converges to zero almost surely. Now, the rest of the proof of~\eqref{eq:rav2ndmean} follows the exact same approach as the proof of Proposition~\ref{prop:condmeantaufin2ndpart}.
\end{proof}

{\cb 
\begin{proposition}[Max expected degree, second order, \ref{ass:weightgumbel}-\ref{ass:weighttau0}]\label{prop:condmeantau02nd}
	Consider the same conditions as in Proposition~\ref{prop:condmeantau0}. When $\tau\in(1,3]$,
	\be \label{eq:rav2nd}
	\max_{\inn }\frac{\Ef{}{\zni}-mb_{t_nn}\log(1/t_n)}{ma_{t_nn}\log(1/t_n)\log\log n}\toinp \frac 12 \Big(1-\frac 1\tau\Big),
	\ee 
	whilst for $\tau>3$,
	\be \label{eq:rav2nd3}
	\max_{\inn }\frac{\Ef{}{\zni}-mb_{t_nn}\log(1/t_n)}{ma_{t_nn}\log(1/t_n)(\log n)^{1-3/\tau}}\toinp -\frac{\tau(\tau-1)^2}{2c_1^3}.
	\ee 
\end{proposition}

}

\begin{proof}  
With a similar argument as in the proof of Proposition~\ref{prop:condmeantau0},it suffices to prove the result for $\max_{\inn}W_i\log(n/i)$ instead of $\max_{\inn}\Ef{}{\zni}$. The general approach is similar to the proof of~\eqref{eq:taufin2} in Proposition~\ref{prop:condmeantaufin}, though the details differ. We first consider the case $\tau\in(1,3]$ and then tend to the case $\tau>3$. In both cases, we prove a lower and upper bound. Moreover, for the lower bound we need only consider indices $\e^{k_n}t_nn\leq i\leq t_nn$ for a particular choice of $k_n$.

Fix $\tau\in(1,3], x\in\R$ and let $k_n:=\sqrt{(\tau-1)/c_1}\sqrt{(\log n)^{1-1/\tau}\log\log n}$. We bound  
\be \ba \label{eq:maxravbound}
\mathbb{P}{}&\bigg(\max_{t_nn\leq i\leq \e^{k_n}t_nn}\frac{W_i\log(n/i)-b_{t_nn}\log (1/t_n)}{a_{t_nn}\log (1/t_n)\log\log n}\leq x\bigg)\\
&=\prod_{i=t_nn}^{\e^{k_n}t_nn}\bigg(1-\P{W\geq\frac{\log(1/t_n)}{\log(n/i)}(b_{t_nn}+a_{t_nn}x\log \log n)}\bigg)\\
&\leq \exp\bigg(-\sum_{i=t_nn}^{ \e^{k_n}t_nn}\P{W\geq \frac{\log(1/t_n)}{\log(n/i)}(b_{t_nn}+a_{t_nn}\log\log n)}\bigg)\\
&\leq \exp\bigg(-\sum_{j=1}^{ k_n}\sum_{i=\e^{j-1}t_nn}^{\e^j t_nn}\P{W\geq \frac{\log(1/t_n)}{\log(1/t_n)-j}(b_{t_nn}+a_{t_nn}x\log \log n )}\bigg).
\ea \ee
We then obtain an upper bound by bounding the probability from below. So, for $\delta>0$ small and $n$ large, we arrive at the upper bound
\be \ba\label{eq:doubleexprav}
\exp\bigg(-{}&(1-\delta)a(\e-1)\sum_{j=1}^{k_n}\e^{j-1}t_nn\Big( \log\Big(\frac{\log(1/t_n)}{\log(1/t_n)-j}(b_{t_nn}+a_{t_nn}x\log \log n )\Big)\Big)^b\\
&\times \exp\Big(-\Big(\log\Big(\frac{\log(1/t_n)}{\log(1/t_n)-j}(b_{t_nn}+a_{t_nn}x\log\log n)\Big)/c_1\Big)^\tau\Big)\bigg).
\ea \ee 
As $\log(1/t_n)/(\log(1/t_n)-j)=1+o(1)$ uniformly in $j$, we can write the inner exponent as
\begin{align}
	-\Big(\log{}&\Big(\frac{\log(1/t_n)}{\log(1/t_n)-j}(b_{t_nn}+a_{t_nn}x\log\log n )\Big)/c_1\Big)^\tau\\
	={}&- \Big(-\log\Big(1-\frac{j}{\log(1/t_n)}\Big)/c_1+\log(b_{t_nn}+a_{t_nn}x\log\log n )/c_1\Big)^\tau\\
	={}&-\bigg(\sum_{\ell=1}^\infty \frac{1}{\ell c_1}\Big(\frac{j}{\log(1/t_n)}\Big)^\ell +\log(t_nn)^{1/\tau}+\frac1\tau \log(t_nn)^{1/\tau-1}\log\big(a(c_1(\log t_nn)^{1/\tau})^b\big)\\
	&+\sum_{\ell=1}^\infty (-1)^{\ell-1}\frac{1}{\ell c_1}\big(c_2\log(t_nn)^{1/\tau-1}x\log\log n\big)^\ell \bigg)^\tau\\
	={}&-\log(t_nn)\bigg(1+\sum_{\ell=1}^\infty \frac{1}{\ell c_1}\Big(\frac{j}{\log(1/t_n)}\Big)^\ell\log(t_nn)^{-1/\tau} \\ 
	&\hphantom{-\log(t_nn)\bigg(}+\frac1\tau \log(t_nn)^{-1}\log\big(a(c_1(\log t_nn)^{1/\tau})^b\big)\\
	&\hphantom{-\log(t_nn)\bigg(}+\sum_{\ell=1}^\infty (-1)^{\ell-1}\frac{1}{\ell c_1}\big(c_2\log(t_nn)^{1/\tau-1}x\log\log n\big)^\ell\log(t_nn)^{-1/\tau} \bigg)^\tau\\
	={}&-\log(t_nn)-\log\big(a(c_1(\log t_nn)^{1/\tau})^b\big)-x\log\log n-\frac{\log(t_nn)^{1-1/\tau}}{c_2\log(1/t_n)}j\\
	&-\frac{\log(t_nn)^{1-1/\tau}}{2c_2\log(1/t_n)^2}j^2(1+o(1))+o(1). \label{eq:2204-1}
\end{align}
Now, for $\ell\in\{1,2\}$,
\be\ba
\frac{1}{\ell c_2}\Big(\frac{j}{\log(1/t_n)}\Big)^\ell & \log(t_nn)^{1-1/\tau}\\ 
={}&\frac{j^\ell}{\ell c_2}c_2^\ell (\log n)^{-(\ell-1)(1-1/\tau)}\Big(1+\frac{\log t_n}{\log n}\Big)^{1-1/\tau}\Big(\frac{\log(b_n)}{c_1(\log n)^{1/\tau}}\Big)^\ell \\
={}&\frac{j^\ell}{\ell c_2}c_2^\ell (\log n)^{-(\ell-1)(1-1/\tau)}\Big(1-\frac{\tau-1}{\log(b_n)}(1+\mathcal O((\log n)^{-1/\tau}))\Big)\\
&\times\Big(1 + \frac 1\tau (\log n)^{-1}\log\big(a(c_1(\log n)^{1/\tau})^b\big)\Big)^\ell,
\ea \ee      
where the last two terms are both $(1+\mathcal O((\log n)^{-1/\tau}))$. So, the right-hand side of~\eqref{eq:2204-1} equals
\be \ba\label{eq:o1}
-\log(t_nn){}&-\log\big(a(c_1(\log t_nn)^{1/\tau})^b\big)-x\log\log n-\frac{j^2}{2 }c_2(\log n)^{-(1-1/\tau)}\\
&-j\Big(1+\frac 1\tau(\log n)^{-1}\log\big(a(c_1(\log n)^{1/\tau})^b\big) -\frac{\tau-1}{\log(b_n)}\Big)+ o(1),
\ea \ee
uniformly in $j\in[k_n]$, where we note that the final $o(1)$ term vanishes uniformly in $j$ since $\tau\in(1,3]$. Using this in~\eqref{eq:doubleexprav}, we arrive at 
\be \ba \label{eq:rav3ub}
\mathbb P{}&\bigg(\max_{t_nn\leq i\leq \e^{k_n}t_nn}\frac{W_i\log(n/i)-b_{t_nn}\log(1/t_n)}{a_{t_nn}\log(1/t_n)\log\log n}\leq x\bigg)\\
&\leq  \exp\Big(-(1-\delta)(1-1/\e)\sum_{j=1}^{k_n}(\log n)^{-x}(1+o(1))\exp\big(\e_n j-d_n j^2\big)\Big),
\ea \ee 
where 
\be 
\e_n:=\frac{\tau-1}{\log(b_n)}-\frac 1\tau(\log n)^{-1}\log\big(a(c_1(\log n)^{1/\tau})^b\big),\qquad d_n:= \frac{c_2}{2(\log n)^{1-1/\tau}}.
\ee 
The expression $\e_n j-d_nj^2$ is increasing for $j\leq \e_n/(2d_n)=o(k_n)$, so we bound the sum from below by 
\be \ba  
\sum_{j=1}^{k_n}\e^{\e_n j-d_nj^2}&\geq \exp\big(\e^2_n/(4d_n)\big)\int_0^{k_n}\exp\Big(-d_n\Big(y-\frac{\e_n}{2d_n}\Big)^2\Big)\,\d y\\
&=\exp\big(\e^2_n/(4d_n)\big)\sigma_n\sqrt{2\pi}\int_0^{k_n}\frac{1}{\sigma_n\sqrt{2\pi}}\exp\Big(-\frac 12 \Big(\frac{y-\mu_n}{\sigma_n}\Big)^2\Big)\,\d y\\
&=\exp\big(\e^2_n/(4d_n)\big)\sigma_n\sqrt{2\pi}\P{X_n\in(0,k_n)},
\ea \ee 
where $\mu_n:=\e_n/(2d_n), \sigma_n:=1/\sqrt{2d_n}$, and $X_n\sim \mathcal N(\mu_n,\sigma_n^2)$ is a normal random variable. If we let $Z\sim \mathcal N(0,1)$ be a standard normal, we can write the last line as 
\be \label{eq:normalprob}
\exp\big(\e^2_n/(4d_n)\big)\sigma_n\sqrt{2\pi}\P{Z\in\Big(-\frac{\mu_n}{\sigma_n},\frac{k_n-\mu_n}{\sigma_n}\Big)}.
\ee 
It is clear that for $\tau \leq3$,
\be \ba\label{eq:seqexpr}
\frac{\e_n^2}{4d_n}&=\frac{\tau(\tau-1)^2}{2c_1^3}(\log n)^{1-3/\tau}(1+o(1)),\quad \frac{k_n-\mu_n}{\sigma_n}=\sqrt{(1-1/\tau)\log\log n},\\
\frac{\mu_n}{\sigma_n}&=\sqrt{\frac{\tau(\tau-1)^2}{c_1^3}}(\log n)^{(1-3/\tau)/2}(1+o(1)).
\ea \ee     
It thus follows that, when $\tau\in(1,3]$, the probability as well as the exponential term in~\eqref{eq:normalprob} converge to a constant. So, for some $K>0$, we bound the expression in~\eqref{eq:normalprob} from above by $K(\log n)^{(1-1/\tau)/2}$. Using this in~\eqref{eq:rav3ub} finally yields the lower bound
\be 
\exp\Big(-(1-\delta)(1-1/\e)(\log n)^{(1-1/\tau)/2-x}(1+o(1))\Big),
\ee 
which converges to zero for any $x<(1-1/\tau)/2$. We thus arrive at, with high probability,
\be 
\max_{\inn}\max_{t_nn\leq i\leq \e^{k_n}t_nn}\frac{W_i\log(n/i)-b_{t_nn}\log (1/t_n)}{a_{t_nn}\log (1/t_n)\log\log n}\geq \frac12 \Big(1-\frac 1\tau\Big)+\eta,
\ee 
for any $\eta>0$. 

To prove a matching upper bound, we split the set $[n]$ into four parts: the indices $1\leq i\leq \e^{-k_n}t_nn, \e^{-k_n}t_nn\leq i\leq t_nn, t_nn\leq i\leq \e^{k_n}t_nn$ and $\e^{k_n}t_nn\leq i\leq n$. We prove an upper bound for all four ranges of indices, a union bound then concludes the proof. The proof for the first two ranges of indices is analogous to the proof for the latter two, so we focus on the latter two. Let us start with the range $t_nn\leq i\leq \e^{k_n}t_nn$. We can use the same approach as above, though with minor adaptations. First, using that $1-y\geq \exp(-(1+\delta)y)$ for $y$ sufficiently small and $\delta>0$ fixed, we obtain 
\be \ba
\mathbb{P}{}&\bigg(\max_{t_nn\leq i\leq \e^{k_n}t_nn}\frac{W_i\log(n/i)-b_{t_nn}\log (1/t_n)}{a_{t_nn}\log (1/t_n)\log\log n}\leq x\bigg)\\
&\geq \exp\bigg(-(1+\delta)\sum_{j=1}^{ k_n}\sum_{i=\e^{j-1}t_nn}^{\e^j t_nn}\P{W\geq \frac{\log(1/t_n)}{\log(1/t_n)-(j-1)}(b_{t_nn}+a_{t_nn}x\log \log n )}\bigg).
\ea \ee
We then bound the probability from above and use the same Taylor expansions to obtain a lower bound of the same form as~\eqref{eq:doubleexprav}. We bound the sum over $j$ from above by 
\be\ba 
\sum_{j=1}^{k_n}\e^{\e_n(j-1)-d_n(j-1)^2}\leq\exp\big(\e^2_n/(4d_n)\big)\sigma_n\sqrt{2\pi}\P{X_n\in(0,k_n)},
\ea \ee 
where $d_n,\e_n,\mu_n,\sigma_n$ and $X_n$ are as above. With a similar argument, we obtain an upper bound 
\be 
\wt K(\log n)^{(1-1/\tau)/2},
\ee 
for some appropriate constant $\wt K>0$. We thus obtain 
\be 
\mathbb{P}{}\bigg(\max_{t_nn\leq i\leq \e^{k_n}t_nn}\frac{W_i\log(n/i)-b_{t_nn}\log (1/t_n)}{a_{t_nn}\log (1/t_n)\log\log n}\leq x\bigg)\geq \e^{(-(1+\delta)^2(\e-1)(\log n)^{(1-1/\tau)/2-x})},
\ee 
which converges to one for any $x>(1-1/\tau)/2$. 

To prove an upper bound for the range $\e^{k_n}t_nn\leq i\leq n$, a slight adaptation is required. We write,
\be \ba 
\mathbb{P}{}&\bigg(\max_{\e^{k_n}t_nn\leq i\leq n}\frac{W_i\log(n/i)-b_{t_nn}\log (1/t_n)}{a_{t_nn}\log (1/t_n)\log\log n}\leq x\bigg)\\
&\geq \exp\bigg(-(1+\delta)\sum_{j=1}^{\log(1/t_n)-k_n}\!\!\!\sum_{i=\e^{j-1+k_n}t_nn}^{\e^{j+k_n} t_nn}\!\!\!\!\P{W\geq \frac{\log(1/t_n)(b_{t_nn}+a_{t_nn}x\log\log n )}{\log(1/t_n)-(j-1+k_n)}}\bigg),
\ea \ee 
and bound the probability from above by 
\be \ba 
(1+\delta){}&a \Big( \log\Big(\frac{\log(1/t_n)}{\log(1/t_n)-(j-1+k_n)}(b_{t_nn}+a_{t_nn}x\log\log n )\Big)  \Big)^b\\
&\times \exp\Big(-\Big(\log\Big(\frac{\log(1/t_n)}{\log(1/t_n)-(j-1+k_n)}(b_{t_nn}+a_{t_nn}x\log\log n )\Big)/c_1\Big)^\tau\Big).
\ea \ee 
Since the fraction $\log(1/t_n)/(\log(1/t_n)-(j-1+k_n))$ is no longer $1+o(1)$ uniformly in $j$, we treat this term somewhat differently. We write the exponent as
\be \ba 
- (\log ( & b_{t_n n}) / c_1)^\tau   - x \log \log n (1+ o(1)) \\ &+\frac{1}{c_2}\log(t_nn)^{1-1/\tau}\log\Big(1-\frac{j-1+k_n}{\log(1/t_n)}\Big)\Big(1+\mathcal O\Big(\frac{\log\log n}{\log n}\Big)\Big).
\ea \ee
Then using that as before $t_n n a (\log (b_{t_nn}))^b \exp\big(- (b_{t_nn}/c_1)^\tau\big) \sim 1$, this yields a lower bound 
\be \ba \label{eq:lbinexp}
\exp\Big({}&-(1+\delta)^2(\e-1)\sum_{j=1}^{\log(1/t_n)-k_n}(\log n)^{-x(1+o(1))}(1+o(1))\\
\times{}& \exp\Big(j-1+k_n+\frac{1}{c_2}\log(t_nn)^{1-1/\tau}\log\Big(1-\frac{j-1+k_n}{\log(1/t_n)}\Big)\Big(1+\mathcal O\Big(\frac{\log\log n}{\log n}\Big)\Big)\Big)\Big).
\ea\ee 
The mapping $x\mapsto x+f_n\log(1-x/g_n)$, $x < g_n$, for some sequences $f_n,g_n$, is maximised at $x=g_n-f_n$. In this case, with $f_n:=c_2^{-1}\log(t_nn)^{1-1/\tau}$ and $ g_n=\log(1/t_n)$, it follows that $g_n-f_n=\tau(\tau-1)c_1^{-2}(\log n)^{1-2/\tau}(1+o(1))$. Since $(\log n)^{1-2/\tau}=o(k_n)$ when $\tau\in(1,3]$, as $(1-1/\tau)/2\geq 1-2/\tau$ for $\tau\in(1,3]$, it follows that the inner exponent is largest when $j=1$. This yields the lower bound
\be \ba\label{eq:finbound}
\exp\Big({}&-(1-\delta)^2(\e-1)\log(1/t_n)(\log n)^{-x(1+o(1))}(1+o(1))\\
\times{}& \exp\Big(\Big(-K_1\frac{k_n}{\tau\log n}\log\log n+\frac{(\tau-1)k_n}{\log b_n}-\frac{1}{2}\frac{k_n^2}{\log(1/t_n)}\Big)\Big(1+\mathcal O\Big(\frac{\log\log n}{\log n}\Big)\Big)\Big)\Big),
\ea\ee 
for some small constant $K_1>0$. The first two terms in the inner exponent are negligible (compared to the last term) by the choice of $k_n$. The last term equals $-((1-1/\tau)/2)\log\log n(1+o(1))$, so that we finally obtain the lower bound 
\be 
\exp\big(-K_2(\log n)^{(1-1/\tau)/2-x(1+o(1))}\big),
\ee 
for some sufficiently large $K_2>0$. It thus follows that for any $x>(1-1/\tau)/2$ the lower bound converges to one. Together with the result for the range of indices $t_nn\leq i\leq \e^{k_n}t_nn$ (and a similar result for $1\leq i\leq t_nn$, with analogous proofs), the upper bound then follows, and finishes the proof for the case $\tau\in(1,3]$. 

When $\tau>3$, we set $k_n:=(\tau(\tau-1)/c_1^2)(\log n)^{1-2/\tau}$. For a lower bound on the maximum, we again consider the indices $t_nn\leq i\leq \e^{k_n}t_nn$. The steps in~\eqref{eq:maxravbound}-~\eqref{eq:normalprob} are still valid when replacing $\log\log n$ with $(\log n)^{1-3/\tau}$. The only minor differences are that $\e_n/d_n\sim k_n$ rather than $o(k_n)$ and that the $o(1)$ term in~\eqref{eq:o1} needs to be replaced by $\mathcal O\big(\sqrt{\log\log n(\log n)^{1-5/\tau}}\big)$, though this changes nothing for the rest of the argument. The second quantity in~\eqref{eq:seqexpr} does change, however. We now find that $(k_n-\mu_n)/\sigma_n=o(1)$. As a result, the probability in~\eqref{eq:normalprob} still converges to a constant, but now the exponential term diverges. We thus obtain the lower bound 
\be \ba
\mathbb{P}{}&\bigg(\max_{t_nn\leq i\leq \e^{k_n}t_nn}\frac{W_i\log(n/i)-b_{t_nn}\log (1/t_n)}{a_{t_nn}\log (1/t_n)(\log n)^{1 - 3/\tau}}\leq x\bigg)\\
& \leq     
\exp\Big(-(1-\delta)(1-1/\e)(\log n)^{(1-1/\tau)/2}(1+o(1))\\
&\times \exp\Big(\Big(\frac{\tau(\tau-1)^2}{2c_1^3}+x\Big)(\log n)^{1-3/\tau}(1+o(1))+\mathcal O\Big(\sqrt{\log\log n(\log n)^{1-5/\tau}}\Big)\Big)\Big),
\ea \ee 
which converges to zero for any $x>-\tau(\tau-1)^2/(2c_1^3)$. 

To prove an upper bound, we adjust the arguments for the $\tau\in(1,3]$ case. Again, we substitute $(\log n)^{1-3/\tau}$ for $\log\log n$. The lower bound on the probability for indices $t_nn\leq i\leq \e^{k_n}t_nn$ remains valid, which yields lower bound that converges to zero for any $x<-\tau(\tau-1)^2/(2c_1^3)$. 

For the range $\e^{k_n}t_nn\leq i\leq n$, we find that the expression in~\eqref{eq:lbinexp} still holds (again when substituting $(\log n)^{1-3/\tau}$ for $\log\log n$). However, we improve on the accuracy of~\eqref{eq:lbinexp} by including more terms of the Taylor expansion. This yields, for an appropriate constant $K_3>0$, 
\be \ba 
\exp\Big({}&-K_3\sum_{j=1}^{\log(1/t_n)-k_n}\exp\Big(j-1+k_n+\frac{1}{c_2}\log(t_nn)^{1-1/\tau}\log\Big(1-\frac{j-1+k_n}{\log(1/t_n)}\Big)\\
&+\frac{\tau(\tau-1)}{2c_1^2}\log(t_nn)^{1-2/\tau} \Big(\log\Big(1-\frac{j-1+k_n}{\log(1/t_n)}\Big) \Big)^2(1+o(1))+x(\log n)^{1-3/\tau}\Big)\Big).
\ea \ee 
As is the case when $\tau\in(1,3]$, the inner exponent is largest when $j=1$, yielding the lower bound
\be \ba
\exp\Big({}&-K_3\log(1/t_n)\exp\Big(k_n+\frac{1}{c_2}\log(t_nn)^{1-1/\tau}\log\Big(1-\frac{k_n}{\log(1/t_n)}\Big)(1+o(1))\\
&+x(\log n)^{1-3/\tau}+O\big((\log n)^{1-4/\tau}\big)\Big)\Big).
\ea\ee 
Now, applying a Taylor expansion to the logarithmic term, we can write the inner exponent as 
\be \ba 
k_n{}&-\frac{\log(t_nn)^{1-1/\tau}}{c_2\log(1/t_n)}k_n-\frac{\log(t_nn)^{1-1/\tau}}{2c_2\log(1/t_n)^2}k_n^2+x(\log n)^{1-3/\tau}+O\big((\log n)^{1-4/\tau}\big)\\
&=\frac{\tau(\tau-1)^2}{2c_1^3}(\log n)^{1-3/\tau}+x(\log n)^{1-3/\tau}+O\big((\log n)^{1-4/\tau}\big). 
\ea \ee 
Concluding, we obtain the lower bound
\be 
\exp\Big(-K_3\log(1/t_n)\exp\Big(\frac{\tau(\tau-1)^2}{2c_1^3}(\log n)^{1-3/\tau}+x(\log n)^{1-3/\tau}+O\big((\log n)^{1-4/\tau}\big)\Big)\Big),
\ee 
so that we obtain a limit of one when choosing any $x<-\tau(\tau-1)^2/(2c_1^3)$, which concludes the proof.
\end{proof}

\subsection{Maximum conditional mean degree, \ref{ass:weightfrechet} case}

{\cb In this final sub-section we study the maximum conditional mean degree when the vertex-weights satisfy the~\ref{ass:weightfrechet} case of Assumption~\ref{ass:weights}. }

\begin{proposition}\label{prop:condmeanfrechet}
Consider the WRG model as in Definition~\ref{def:wrg} and suppose the vertex-weights satisfy the~\ref{ass:weightfrechet} case in Assumption~\ref{ass:weights}. Let $\Pi$ be a PPP on $(0,1)\times (0,\infty)$ with intensity measure $\nu(\d t,\d x):=\d t \times (\alpha-1)x^{-\alpha}\d x,\ x>0$. When $\alpha>2$,
\be\label{eq:wrrtcondmeanconv}
\max_{i\in[n]}\Ef{}{\zni/u_n}\toindis m\max_{(t,f)\in\Pi}f\log(1/t),
\ee
and when $\alpha\in (1,2)$,
\be\label{eq:wrrtinfcondmeanconv}
\max_{i\in[n]}\Ef{}{\zni/n}\toindis m\max_{(t,f)\in\Pi} f \int_t^1\bigg(\int_{(0,1)\times(0,\infty)} g\ind_{\{u\leq s\}}\,\d \Pi(u,g)\bigg)^{-1}\,\d s.
\ee
\end{proposition}

\begin{proof}
First, let $\alpha>2$. We first claim that
\be \label{eq:wrrtmaxdiff}
\bigg|\max_{i\in[n]} \Ef{}{\zni/u_n} -m\max_{i\in[n]}\frac{\F_i\log(n/i)}{u_n}\bigg|\toinp 0.
\ee
The claim's proof follows a similar structure as that of~\eqref{eq:conc}. Let us define the point process 
\be \label{eq:pin}
\Pi_n:=\sum_{i=1}^n \delta_{(i/n,\F_i/u_n)}.
\ee 
By~\cite{Res13}, when the $\F_i$ are i.i.d.\ random variables in the Fr\'echet maximum domain of attraction with parameter $\alpha-1$, then $\Pi$ is the weak limit of $\Pi_n$. Since $\F_i\log(n/i)/u_n$ is a continuous mapping of $(i/n,\F_i/u_n)$ and since taking the maximum is a continuous mapping too, it follows that
\be
\max_{i\in[n]}\frac{\F_i\log(n/i)}{u_n}\toindis \max_{(t,f)\in\Pi} f\log(1/t),
\ee 
which, together with~\eqref{eq:wrrtmaxdiff}, yields the desired result. We now consider $\alpha\in(1,2)$. Note that 
\be 
\max_{i\in[n]}\Ef{}{\zni/n}=m\max_{i\in[n]}\frac{\F_i}{n}\sum_{j=i}^{n-1}\frac{1}{S_j}.
\ee 
The distributional convergence of the maximum on the right-hand side to the desired limit is proved in~\cite[Proposition 5.1]{LodOrt20}, which concludes the proof.
\end{proof}

\section{Concentration of the maximum degree}\label{sec:conc}

In this section we provide an important step to prove Theorems~\ref{thrm:maxfrechet},~\ref{thrm:maxgumb},~\ref{Thrm:maxsecond} and~\ref{Thrm:maxsecond2}: we discuss the concentration of the maximum degree around the maximum conditional mean degree, the behaviour of which is discussed in the previous section. To this end, we present two propositions, that determine concentration under different scalings.

{\cb  
\begin{proposition}[Concentration of degrees, first-order scaling]\label{lemma:wrrtconcentration}
	Consider the WRG model as in Definition~\ref{def:wrg} and recall the vertex-weight conditions as in Assumption~\ref{ass:weights}. When the vertex-weights satisfy the~\ref{ass:weightgumbel}-\ref{ass:weighttauinf} sub-case, for any $\eta>0$,
	\be \label{eq:concgumbab}
	\lim_{n\to \infty}\P{\max_{\inn}\big|\zni-\Ef{}{\zni}\big|\geq \eta b_n\log n}=0.
	\ee
	When the vertex-weights satisfy the~\ref{ass:weightgumbel}-\ref{ass:weighttaufin} sub-case, 
	\be \label{eq:concrv}
	\max_{\inn}\big|\zni-\Ef{}{\zni}\big|/b_n\log n\overset{\mathbb P-\mathrm{a.s.}}{\longrightarrow} 0.
	\ee 
	Furthermore, when the vertex-weights satisfy the~\ref{ass:weightgumbel}-\ref{ass:weighttau0} sub-case, let \\$t_n:=\exp(-\tau\log n/\log(b_n))$. Then, for any $\eta>0$,
	\be \label{eq:concgumbc}
	\lim_{n\to \infty}\P{\max_{\inn}\big|\zni-\Ef{}{\zni}\big|\geq \eta a_{t_nn}\log(1/t_n)}=0.
	\ee        
	Now suppose the vertex-weights satisfy the~\ref{ass:weightfrechet} case. When $\alpha>2$, for any $\eta>0$,
	\be \label{eq:wrrtconcii}
	\lim_{n\rightarrow \infty}\P{\max_{\inn}\big|\zni-\Ef{}{\zni}\big|>\eta u_n}=0.
	\ee 
	Similarly, when $\alpha\in(1,2)$, for any $\eta>0$,
	\be \label{eq:wrrtinfmeanconcii}
	\lim_{n\rightarrow \infty}\P{\max_{\inn}\big|\zni-\Ef{}{\zni}\big|>\eta n}=0.
	\ee  	
\end{proposition}

\begin{remark}\label{rem:conc}
	The results of Proposition~\ref{lemma:wrrtconcentration} directly imply the concentration of the maximum degree due to the reversed triangle inequality for the supremum norm. That is, for any $I_n\subseteq [n]$, 
	\be 
	\big|\max_{i\in I_n}\zni-\max_{i\in I_n}\Ef{}{\zni}\big|\leq \max_{i\in I_n}\big|\zni-\Ef{}{\zni}\big|.
	\ee 
\end{remark}

}
\begin{proof}
We provide a proof for $m=1$, as the proof for $m>1$ follows analogously. We start by proving the results in which the degrees are scaled by the first-order growth rate and the convergence is in probability, as in ~\eqref{eq:concgumbab},~\eqref{eq:concgumbc},~\eqref{eq:wrrtconcii} and~\eqref{eq:wrrtinfmeanconcii}. By using a large deviation bound for a sum of independent Bernoulli random variables, see e.g.~\cite[Theorem $2.21$]{Hof16}, we obtain
\be \ba\label{eq:largedev}
\Pf{\big|\zni-\Ef{}{\zni}\big|\geq g_n}&\leq 2\exp\Big(-\frac{g_n^2}{2(\Ef{}{\zni}+g_n)}\Big)\\
&\leq 2\exp\Big(-\frac{g_n^2}{2(\max_{\inn}\Ef{}{\zni}+g_n)}\Big),
\ea \ee 
for any non-negative sequence $(g_n)_{n\in\N}$. We start by considering~\eqref{eq:concgumbab}, so that $g_n=\eta b_n\log n$. Hence, the fraction on the right-hand side is $g_nB_n$ for some random variable $B_n$ that converges in probability to some positive constant (see Propositions~\ref{prop:condmeantauinf}). Using a union bound then yields
\be \ba \label{eq:unionconcbound}
\Pf{\max_{\inn}|\zni-\Ef{}{\zni}|\geq g_n}&\leq \sum_{i=1}^n 2\exp(-g_n B_n)\\
&= 2\exp(\log n(1-(g_n/\log n)B_n)),
\ea \ee 
As $g_n/\log n=\eta b_n$ diverges with $n$, it follows that this expression tends to zero in probability. For~\eqref{eq:concgumbc}, we can also use~\eqref{eq:largedev} with $g_n=\eta a_{t_nn}\log(1/t_n)$ and we can write
\be\ba
&\frac{g_n^2}{2(\max_{\inn}\Ef{}{\zni}+g_n)}\\
&\hphantom{ddd}=\frac{(\eta a_{t_nn}\log(1/t_n))^2}{2b_{t_nn}\log(1/t_n)(\max_{\inn}\Ef{}{\zni}/(b_{t_nn}\log(1/t_n))+\eta a_{t_nn}/b_{t_nn})}\\
&\hphantom{ddd}=\frac{\eta^2}{2}\frac{a_{t_nn}^2\log(1/t_n)}{b_{t_nn}}B_n,
\ea \ee
where $B_n$ converges in probability to a positive constant (see the proof of Proposition~\ref{prop:condmeantau0} and use the definition of $a_n$ and $b_n$ in Remark~\ref{remark:weights}). Since $b_n/b_{t_nn}\to \e$ (again see the proof of Proposition~\ref{prop:condmeantau0}) and by the definition of $a_n$ and $b_n$ in the~\ref{ass:weightgumbel}-\ref{ass:weighttau0} sub-case, it follows that the right-hand side is at least $Cb_n(\log n)^{1/\tau-1}$ with high probability, for some small constant $C>0$. Replacing $g_n$ with $Cb_n(\log n)^{1/\tau-1}$, which grows faster than $\log n$, in~\eqref{eq:unionconcbound} then yields the desired result. Finally, for~\eqref{eq:wrrtconcii} and~\eqref{eq:wrrtinfmeanconcii}, the same approach is valid with $g_n=u_n$ and $g_n=n$, respectively, though $B_n$ now converges in distribution to some random variable (see Proposition~\ref{prop:condmeanfrechet}). Still, it follows that $1-\eta^2(g_n/\log n)B_n<0$ with high probability, so that the right-hand side of~\eqref{eq:unionconcbound} still converges to zero in probability. Then, in all the above cases, using the dominated convergence theorem yields~\eqref{eq:concgumbab},~\eqref{eq:concgumbc},~\eqref{eq:wrrtconcii} and~\eqref{eq:wrrtinfmeanconcii}.

We now turn to the almost sure result for the~\ref{ass:weightgumbel}-\ref{ass:weighttaufin} sub-case, as in~\eqref{eq:concrv}. Similar to~\eqref{eq:unionconcbound}, we write for any $\eta>0$,
\be \ba
\Pf{\max_{\inn}|\zni-\Ef{}{\zni}|\geq \eta g_n}\leq 2\exp\Big(\log n-\frac{\eta^2g_n^2}{2(\max_{\inn}\Ef{}{\zni}+\eta g_n)}\Big),
\ea \ee
where $g_n=b_n\log n$. By Proposition~\ref{prop:condmeantaufin}, we can almost surely bound this from above by 
\be 
2\exp(\log n-\eta^2Cb_n\log n)\leq 2\exp\Big(-\frac{1}{2}\eta^2C b_n\log n \Big ),
\ee 
for some sufficiently small constant $C>0$ and when $n$ is at least $N\in\N$, for some random $N$. Thus, we can conclude that this upper bound is almost surely summable in $n$, as $b_n$ tends to infinity with $n$. Since $\eta$ is arbitrary, the $\mathbb P_W$-almost sure convergence to $0$ is established. Then, since
\be \ba
\mathbb P{}&\bigg(\forall \eta>0\ \exists N\in\N\ \forall n\geq N: \max_{\inn}|\zni-\Ef{}{\zni}|<\eta g_n\bigg)\\
&=\E{\Pf{\forall \eta>0\ \exists N\in\N\ \forall n\geq N: \max_{\inn}|\zni-\Ef{}{\zni}|<\eta g_n}}=\E{1}=1,
\ea \ee 
the $\mathbb P$-almost sure convergence follows as well and which concludes the proof.
\end{proof} 

{\cb

Proposition~\ref{lemma:wrrtconcentration} provides concentration of the in-degrees around the conditional mean in-degrees under the first-order scaling (with the exception of~\eqref{eq:concgumbc}) of the maximum conditional mean degree (see the propositions in Section~\ref{sec:meandeg}). In the following proposition, we provide a stronger result for the~\ref{ass:weightgumbel}-\ref{ass:weighttaufin} sub-case. Here, with a more precise and detailed proof, we can also prove concentration under the second-order scaling.

\begin{proposition}[Concentration of degrees, second-order scaling]\label{prop:conc2nd}
	Consider the WRG model as in Definition~\ref{def:wrg} and recall the vertex-weight conditions as in Assumption~\ref{ass:weights}. When the vertex-weights satisfy the~\ref{ass:weightgumbel}-\ref{ass:weighttaufin} sub-case, for any $\tau\in(0,1]$ and $\eta>0$,
	\be\label{eq:gumb2ndconc}
	\lim_{n\to\infty}\P{\big|\max_{\inn} \zni-\max_{\inn}\Ef{}{\zni}\big|\geq \eta a_n\log n\log\log n}=0.
	\ee
	Also, let $\ell$ be a strictly positive function such that $\lim_{n\to\infty}\log(\ell(n))^2/\log n=\zeta_0$ for some $\zeta_0\in[0,\infty)$. Recall $C_n(\gamma,s,t,\ell)$ from~\eqref{eq:cnin}. Then, for any $0<s<t<\infty, \tau\in(0,1)$ and $\eta>0$,
	\be\label{eq:betaconc}
	\lim_{n\to\infty}\P{\big|\max_{i\in C_n(\gamma,s,t,\ell)}\zni-\max_{i\in C_n(\gamma,s,t,\ell)}\Ef{}{\zni}\big|\geq \eta a_n\log n}=0.
	\ee
\end{proposition}
}

\begin{proof}	
We start by proving~\eqref{eq:gumb2ndconc}. We first aim to show that
\be \label{eq:2ndordlower}
\Pf{\max_{\inn}\zni\geq \max_{\inn}\Ef{}{\zni}+\eta a_n\log n\log\log n}\toinp 0.
\ee 
We provide a proof for the case that $\tau\in(0,3/4)$, after which we develop a more involved proof for all $\tau\in(0,1]$ which builds on the initial proof.

Let $\tau\in(0,3/4)$. We use a union bound for~\eqref{eq:2ndordlower} and split the sum into two sets, defined as
\be
C^1_n:=\{\inn: W_i< (1-\sqrt{\eps_n})b_i\},\qquad C^2_n:=\{\inn:W_i\geq (1-\sqrt{\eps_n})b_i\},
\ee
where $\eps_n=(\log n)^{-c}$, for some $c\in(0,2)$ to be determined later on. The size of $C^2_n$ can be controlled well enough, so that a precise union bound can be applied to this part. For the other set, we claim that with high probability,
\be \label{eq:c2inc3}
C^1_n\subseteq \{\inn: \Ef{}{\zni}\leq (1-\sqrt{\eps_n})\max_{\inn}\Ef{}{\zni}\}=:\wt C^1_n.
\ee 
The idea we apply later is that, { \cb on the event $\{C^1_n\subseteq \wt C^1_n\}$ (in the sense that we consider only $\omega\in \{C^1_n\subseteq \wt C^1_n\}$)}, for each $i\in C^1_n$ we are able to manipulate terms in the probability in~\eqref{eq:2ndordlower} to such an extent that we obtain an improved large deviation bound and can then  show this part converges to zero in probability as well. However, we first focus on proving the the inclusion in~\eqref{eq:c2inc3} holds with high probability.

Take $i\in C^1_n$. Then, 
\be
\Ef{}{\zni}=W_i\sum_{j=i}^{n-1}\frac{1}{S_j}\leq (1-\sqrt{\eps_n})b_i \log(n/i)\Big(1+\frac{|Y_n-Y_i|}{\log(n/i)}\Big),
\ee  
where $Y_n:=\sum_{j=1}^{n-1}1/S_j-\log n$. Furthermore, by Proposition~\ref{prop:condmeantaufin2ndcomp},~\eqref{eq:taufin2}, with high probability
\be 
\max_{\inn}\Ef{}{\zni}\geq (1-\gamma)b_{n^\gamma}\log n\Big(1+\frac{1/2-\eta}{1-\gamma}\frac{\log\log n}{\log n}\Big),
\ee 
for any fixed $\eta>0$.	If thus suffices to show that when $i\in C^1_n$, then with high probability
\be \label{eq:concreq}
b_i \log(n/i)\Big(1+\frac{|Y_n-Y_i|}{\log(n/i)}\Big)\leq (1-\gamma)b_{n^\gamma}\log n\Big(1+\frac{1/2-\eta}{1-\gamma}\frac{\log\log n}{\log n}\Big)
\ee 
is satisfied for some $\eta>0$. We show a stronger statement, namely that~\eqref{eq:concreq} holds with high probability for any $\inn$. 
We recall from~\eqref{eq:sumfitnessconv} that $Y_n$ converges almost surely. 
In particular, we have that, with high probability, $\max_{i \in [n]} Y_i \leq (\log \log n)^{1/2}$ and 
$\max_{\log \log n \leq i \leq n} |Y_i - Y_n|$ converges to $0$ in probability.	Note first that for $i \leq (\log \log n)$,
\be \ba
b_i \log(n/i)\Big(1+\frac{|Y_n-Y_i|}{\log(n/i)}\Big)
&\leq (1-\gamma)b_{n^\gamma }\log n\Big(1+\frac{1/2-\eta}{1-\gamma} \frac{\log\log n}{\log n}\Big).
\ea \ee 
Next we consider  $\log \log n \leq i\leq n^{\gamma-\eps}$ or $i\geq n^{\gamma+\eps}$ for some $\eps>0$. We get that with high probability
\be 
b_i \log(n/i)\Big(1+\frac{|Y_n-Y_i|}{\log(n/i)}\Big)\leq C(1-\gamma)b_{n^\gamma}\log n,
\ee 
for some $C\in(0,1)$, as follows from the proof of Proposition~\ref{prop:condmeantaufin} (\eqref{eq:eps_limit} to be more precise),  so that~\eqref{eq:concreq} is satisfied.
It remains to prove that~\eqref{eq:concreq} is satisfied with high probability when $i=n^\gamma k_n$, where $k_n$ is sub-polynomial, in the sense that $|\log k_n| / \log n \rightarrow 0$. First, as before, with high probability
\be 
1+\frac{|Y_n-Y_i|}{\log(n/i)}\leq 1+\eta\frac{\log\log( n^\gamma)}{\log n},
\ee 
for any constant $\eta>0$. Moreover, by Remark~\ref{remark:weights}, for any $\eta>0$,
\be 
(1-\gamma)b_{n^\gamma}\log n\geq c_1(1-\gamma)\log n\log(n^\gamma)^{1/\tau}\Big(1+\frac{(b/\tau-\eta)\log\log(n^\gamma)}{(1-\gamma) \log n}\Big),
\ee 
for all $n$ large. Finally, for $n$ large,
\be \ba
b_i \log(n/i)={}&c_1(1-\gamma)\log n \log(n^\gamma)^{1/\tau} \Big(1+\frac{(b/\tau)\log\log (n^\gamma k_n)+b\log c_1+\log \tau}{(1-\gamma)\log n}\Big)\\
&\times\Big(1-\frac{\log k_n}{(1-\gamma)\log n}\Big)\Big(1+\frac{\log k_n}{\gamma\log n}\Big)^{1/\tau}\\
\leq{}& c_1(1-\gamma)\log n\log(n^\gamma)^{1/\tau}\Big(1+\frac{(b/\tau)\log\log (n^{\gamma\pm \eta})}{(1-\gamma)\log n}\Big)\\
\leq{}& c_1(1-\gamma)\log n\log(n^\gamma)^{1/\tau}\Big(1+\frac{(b/\tau+\eta)\log\log (n^\gamma)}{(1-\gamma)\log n}\Big)
\ea \ee 
as $(1-x/(1-\gamma))(1+x/\gamma)^{1/\tau}\leq 1$ for $x\in[0,1-\gamma]$ and where the $\pm$ sign depends on the sign of $b$. Combining all of the above, we find for $n$ large and with high probability,
\be \ba
b_i{}&\log(n/i)\Big(1+\frac{|Y_n-Y_i|}{\log(n/i)}\Big)\\
&\leq c_1(1-\gamma)\log n\log(n^\gamma)^{1/\tau}\Big(1+\frac{(b/\tau+\eta)\log\log (n^\gamma)}{(1-\gamma)\log n}\Big)\Big(1+\eta\frac{\log\log(n^\gamma)}{\log n}\Big) \\
&\leq c_1(1-\gamma)\log n\log(n^\gamma)^{1/\tau}\Big(1+\Big(\frac{b/\tau+\eta}{1-\gamma}+2\eta\Big)\frac{\log\log (n^\gamma)}{\log n}\Big),
\ea\ee 
and 
\be \ba 
(1-\gamma){}&b_{n^\gamma}\log n\Big(1+\frac{1/2-\eta}{1-\gamma}\frac{\log\log n}{\log n}\Big)\\
\geq{}& c_1(1-\gamma)\log n\log(n^\gamma)^{1/\tau}\Big(1+\Big(\frac{b/\tau-\eta}{1-\gamma}+\frac{1/2-\eta}{1-\gamma} -\eta\Big)\frac{\log\log(n^\gamma)}{\log n}\Big),
\ea \ee 
so that~\eqref{eq:concreq} is established with high probability for all $i\in[n-1]$, in particular for all $i\in C^1_n$, when $\eta$ is sufficiently small.

As a second step, we control the size of $C^2_n$. First, we fix an $\eta>0$ and set $I=\max\{I_1,I_2,I_3\}$, where $I_1,I_2,I_3\in\N$ are such that 
\be \ba
\P{W\geq (1-\sqrt{\eps_n})b_i}&\leq (1+\eta)a((1-\sqrt{\eps_n})b_i)^b \exp(-((1-\sqrt{\eps_n})b_i/c_1)^\tau), \qquad &i\geq I_1,&\\
\P{W\geq b_i}&\geq (1-\eta)ab_i^b\exp(-(b_i/c_1)^\tau),\qquad &i\geq I_2,&\\
\P{W\geq b_i}&\leq (1+\eta)/i,\qquad &i\geq I_3.&
\ea\ee 
We note that $I$ is well-defined, as $b_i$ is (eventually) increasing in $i$ and diverges with $i$ and as $b_i$ is such that $\lim_{i\to\infty}\P{W\geq b_i}i = 1$. We then arrive at 
\be\ba  
\E{|C^2_n|}& =\sum_{i=1}^n \P{W\geq (1-\sqrt{\eps_n})b_i}\\ & \leq I+\sum_{i=I}^n (1+\eta)a((1-\sqrt{\eps_n})b_i)^b \exp(-((1-\sqrt{\eps_n})b_i/c_1)^\tau).
\ea \ee 
By writing $(1-\sqrt{\eps_n})^\tau = 1-\tau\sqrt{\eps_n}(1+o(1))$, and as we can bound $(1-\sqrt{\eps_n})^b$ from above by some sufficiently large constant $C$ which depends only on $b$, we obtain
\be\ba
\E{|C^2_n|}&\leq I+(1+\eta)C\sum_{i=I}^n ab_i^b \exp(-(b_i/c_1)^\tau)\exp(\tau\sqrt{\eps_n}(b_i/c_1)^\tau(1+o(1)))   \\
&\leq I+(1+\eta)^2(1-\eta)^{-1}C\exp(\tau\sqrt{\eps_n}(b_n/c_1)^\tau (1+o(1)))\sum_{i=I}^n i^{-1}\\
&\leq I+(1+\eta)^2(1-\eta)^{-1}C\exp(\tau\sqrt{\eps_n}\log n(1+o(1)))\log n\\
&=I+(1+\eta)^2(1-\eta)^{-1}C\exp(\tau\sqrt{\eps_n}\log n(1+o(1))).
\ea \ee 
Since $\sqrt{\eps_n}\log n=(\log n)^{1-c/2}$ and $c\in(0,2)$, it follows by Markov's inequality that for any $C_3>\tau$,
\be \label{eq:cn1conv}
|C^2_n|\exp(-C_3(\log n)^{1-c/2})\toinp 0.
\ee 
We are now ready to prove~\eqref{eq:2ndordlower}. We use the above and~\eqref{eq:c2inc3} to obtain 
\be \ba
\mathbb P_W{}\Big({}&\max_{\inn}\zni\geq \max_{\inn}\Ef{}{\zni}+\eta a_n\log n\log\log n\Big)\\
\leq{}& \mathbb P_W\Big(\Big\{\max_{\inn}\zni\geq \max_{\inn}\Ef{}{\zni}+\eta a_n\log n\log\log n\Big\}\cap \{ C^1_n\subseteq \wt C^1_n\}\Big)\\ & \qquad +\P{C^1_n \not \subseteq \wt C^1_n}\\
\leq{}& \sum_{i\in C^1_n}\Pf{\{\zni\geq \max_{\inn}\Ef{}{\zni}+\eta a_n\log n\log\log n\}\cap \{ C^1_n\subseteq \wt C^1_n\}}\\
&+\sum_{i\in C^2_n}\Pf{\zni\geq \max_{\inn}\Ef{}{\zni}+\eta a_n\log n\log\log n}+\P{C^1_n\not\subseteq \wt C^1_n}.
\ea \ee  
As established above, the third probability converges to zero with $n$. For the first probability we use that on $\{C^1_n\subseteq \wt C^1_n\}$, $\Ef{}{\zni}\leq  (1-\sqrt{\eps_n})\max_{\inn}\Ef{}{\zni}$, and for the second probability we use that $\max_{\inn}\Ef{}{\zni}\geq \Ef{}{\zni}$ for any $\inn$ to find the upper bound
\be \ba \label{eq:unionbound}
\sum_{i\in C^1_n}{}&\mathbb P_W\Big(\zni-\Ef{}{\zni}\geq \sqrt{\eps_n}\max_{\inn}\Ef{}{\zni}+\eta a_n\log n\log\log n\Big)\\
&+\sum_{i\in C^2_n}\Pf{\zni-\Ef{}{\zni}\geq \eta a_n\log n\log\log n}+o(1).
\ea \ee 
Now, applying a large deviation bound to~\eqref{eq:unionbound} yields
\be \ba\label{eq:largedev2}
\sum_{i\in C^1_n}{}&\exp\Big(-\frac{(\sqrt{\eps_n}\max_{\inn}\Ef{}{\zni}+\eta a_n\log n\log\log n)^2}{2(\Ef{}{\zni}+\sqrt{\eps_n}\max_{\inn}\Ef{}{\zni}+\eta a_n\log n\log\log n)}\Big)\\
&+\sum_{i\in  C^2_n}\exp\Big(-\frac{(\eta a_n\log n\log\log n)^2}{2(\Ef{}{\zni}+\eta a_n\log n\log\log n)}\Big)+o(1).
\ea\ee 
In both exponents we bound the conditional mean in the denominator by the maximum conditional mean. This yields the upper bound
\be \label{eq:finalmeanconcub}
n\exp(-\eps_n b_n \log n A_n)+|C^2_n|\exp\Big(-\frac{\eta^2 a_n^2 \log n (\log\log n)^2}{2b_n}B_n\Big),
\ee 
where both $A_n,B_n$ converge in probability to positive constants. We now set $c<1/\tau$ to ensure that $\eps_n b_n$ diverges, so that the first term converges to zero in probability. Thus, $c<(1/\tau)\wedge 2$ is required. We can write the second term as 
\be 
|C^2_n| \exp(-C_3 (\log n)^{1-c/2})\exp(C_3(\log n)^{1-c/2}-(\log n)^{1/\tau-1}(\log\log n)^2\wt B_n),
\ee 
where $\wt B_n$ converges in probability to a positive constant. Now, by~\eqref{eq:cn1conv}, the product of the first two terms converges to zero in probability and the last term converges to zero in probability when $1-c/2<1/\tau-1$, or $c>4-2/\tau$. We thus find that~\eqref{eq:2ndordlower} is established when we can find a $c\in(0,2)$ such that $4-2/\tau<c<1/\tau$, which holds for all $\tau\in (0,3/4)$.

We now extend this approach so that the desired result in~\eqref{eq:2ndordlower} can be achieved for all $\tau\in(0,1]$. To this end, we define the sequence $(p_k)_{k\in\N}$ as $p_k:=(3/4)p_{k-1}+1/(4c\tau), k\geq 1$, and $p_0=1/2$. We solve the recursion to obtain
\be \label{eq:pkseq}
p_k=\frac{1}{c\tau}-\Big(\frac{1}{c\tau}-\frac{1}{2}\Big)\Big(\frac 3 4\Big)^k,
\ee 
from which it immediately follows that $p_k$ is increasing when $c<(1/\tau)\wedge 2$. Moreover, we can rewrite the recursion as $p_k=p_{k-1}/2+(p_{k-1}/2+1/(2c\tau))/2$, so that 
\be \label{eq:pkbounds2}
p_k\in(p_{k-1},p_{k-1}/2+1/(2c\tau)),\qquad k\geq 1.
\ee 
We also define, for some $K\in\N_0$ to be specified later, the sets
\be \ba 
C^1_n&:=\{\inn: W_i< (1-\eps_n^{p_0})b_i \},\\
C^k_n&:=\{\inn: W_i \in[(1-\eps_n^{p_{k-1}})b_i,(1-\eps_n^{p_k})b_i)\},\quad &k\in\{2,\ldots,K\},&\\
\wt C^{k}_n&:=\{\inn: \Ef{}{\zni}\leq (1-\eps_n^{p_k})\max_{\inn}\Ef{}{\zni}\},\quad &k\in[K],&\\
C^{K+1}_n&:=\{\inn: W_i\geq (1-\eps_n^{p_K})b_i\}.
\ea \ee 
By the same argument as provided above, it follows that with high probability $C^k_n\subseteq \wt C^k_n$ for each $k\in\N$. Similar to the approach for $\tau\in (0,3/4)$ we can then bound
\be\ba \label{eq:maxsumsplit}
\mathbb{P}_\F{}&\Big(\max_{\inn}\zni\geq \max_{\inn}\Ef{}{\zni}+\eta a_n\log n\log\log n\Big)\\
\leq{}& \sum_{i\in C^1_n}\mathbb{P}_\F\Big(\zni-\Ef{}{\zni}\geq \eps_n^{p_0}\max_{\inn}\Ef{}{\zni}+\eta a_n \log n\log\log n\Big)\\
&+\sum_{k=2}^K\sum_{i\in C^k_n}\mathbb{P}_\F\Big(\zni-\Ef{}{\zni}\geq \eps_n^{p_k}\max_{\inn}\Ef{}{\zni}+\eta a_n\log n\log\log n\Big)\\
&+\sum_{i\in C^{K+1}_n}\mathbb{P}_\F\Big(\zni-\Ef{}{\zni}\geq \eta a_n\log n\log\log n\Big)+\P{\bigcup_{k=1}^K\{C^k_n\not\subseteq \wt C^k_n\}}. 
\ea \ee 
We do not include the sum over $i\in C^1_n$ in the double sum, as the upper bound we use is slightly different for these terms. The last term converges to zero by using a union bound, as established above and since $K$ is fixed. As in the simplified proof for $\tau\in(0,3/4)$ where $K=1$, we require $cp_k<1$ for all $k\in\{0,1,\ldots,K\}$, so that $\eta a_n\log n\log\log n$ is negligible compared to $\eps_n^{p_k}\max_{\inn}\Ef{}{\zni}$. Since $p_k$ is increasing, $cp_K<1$ suffices. Using~\eqref{eq:pkseq} yields that $K$ cannot be too large, i.e.\  we need
\be \label{eq:Kconstraint1}
\Big(\frac 3 4\Big)^K>\frac{1}{c}\Big(\frac{1}{\tau}-1\Big)\Big(\frac{1}{c\tau}-\frac{1}{2}\Big)^{-1}.
\ee
We now again apply a large deviation bound as in~\eqref{eq:largedev2}. Furthermore, with an equivalent approach that led to~\eqref{eq:cn1conv}, we find with high probability an upper bound for~\eqref{eq:maxsumsplit} of the form
\be \ba\label{eq:largedevK}
n{}&\exp(-\eps_n^{2p_0}b_n \log n A_{0,n})
+\sum_{k=1}^K C\exp(C_k(\log n)^{1-cp_{k-1}}-\eta^2\eps_n^{2p_k}b_n\log n A_{k,n})\\
&+C\exp\Big(C_{K+1}(\log n)^{1-cp_K}-\eta^2\frac{a_n^2\log n(\log\log n)^2}{b_n} A_{K+1,n}\Big)+o(1),
\ea \ee 
where $C,C_1,\ldots,C_{K+1}>0$ are suitable constants and the $A_{k,n}$ are random variables which converge in probability to some strictly positive constants $A_k,k\in\{0,1,\ldots,K+1\}$. In order for all these terms to converge to zero in probability, the following conditions need to be met:
\be \label{eq:pkcond}
1/\tau-2p_0c>0,\qquad 1-cp_{k-1}<-2cp_k+1/\tau+1,\quad k\in[K],\qquad 1-cp_K\leq 1/\tau-1.
\ee 
Since $p_0=1/2$, it follows that $c<(1/\tau)\wedge 2$ still needs to be satisfied. By the bounds on $p_k$ in~\eqref{eq:pkbounds2} the second condition is satisfied and the final condition holds when $p_K\geq (2-1/\tau)/c$, or
\be
\Big(\frac 3 4\Big)^K\leq 2\frac{1}{c}\Big(\frac 1 \tau-1\Big)\Big(\frac{1}{c\tau}-\frac{1}{2}\Big)^{-1}.
\ee 
Together with~\eqref{eq:Kconstraint1} this yields
\be 
\frac{1}{c}\Big(\frac{1}{\tau}-1\Big)\Big(\frac{1}{c\tau}-\frac{1}{2}\Big)^{-1}<\Big(\frac 3 4\Big)^K\leq \frac{2}{c}\Big(\frac 1 \tau-1\Big)\Big(\frac{1}{c\tau}-\frac{1}{2}\Big)^{-1}.
\ee 
Since the ratio of the lower and upper bound is exactly $2$, such a $K\in\N_0$ can always be found, as long as 
\be 
\frac{1}{c}\Big(\frac 1 \tau-1\Big)\Big(\frac{1}{c\tau}-\frac{1}{2}\Big)^{-1}<1,
\ee 
which is satisfied for any $\tau\in (0,1)$ when $c<2$. It thus follows that~\eqref{eq:2ndordlower} holds for all $\tau\in (0,1)$. When $\tau=1$, the condition $p_K\geq  (2-1/\tau)/c$ simplifies to $p_K\geq 1/c$, which cannot hold together with the condition $p_K< 1/c$ for any $K\in\N$. However, when $\tau=1$, the limit of $p_K$ is $1/c$ (as $K$ tends to infinity), so that we let $K$ tend to infinity with $n$. Therefore, we repeat the same arguments, but now take $K=K(n)=\lceil 1/|\log(3/4)|( \log\log\log n-\log\log \log\log n)\rceil$. We then need to check the following things:
\begin{itemize}
	\item[$(i)$] The conditions on $p_k$ and $c$ are met.
	\item[$(ii)$] The final probability in~\eqref{eq:maxsumsplit} converges to zero in probability with $n$.
	\item[$(iii)$] All terms in~\eqref{eq:largedevK} individually converge to zero with $n$, as well as when summing them all together.
\end{itemize}
The first two conditions in~\eqref{eq:pkcond} still need to be satisfied, and this is the case when $K$ grows with $n$ as well. Furthermore, as $\tau=1$, $p_k<1/c$ is satisfied for all $k\in\N$, establishing $(i)$.

For $(ii)$, we observe that by~\eqref{eq:concreq},
\be \ba
\mathbb P{}&\bigg(\bigcup_{k=1}^K \{C^k_n\not\subseteq \wt C^k_n\}\bigg)\\
\leq{}& \P{\bigcup_{i=1}^{n-1} \bigg\{b_i \log(n/i)\Big(1+\frac{|Y_n-Y_i|}{\log(n/i)}\Big)> (1-\gamma)b_{n^\gamma}\log n\Big(1+\frac{1/2-\eta}{\tau}\frac{\log\log n}{\log n}\Big)\bigg\}},
\ea\ee 
for some small $\eta>0$, and the decay of this probability to zero has been established at the start of this proof.

Finally, for $(iii)$, we check the convergence of the terms in~\eqref{eq:largedevK}. The first term clearly still converges to zero in probability. Then, for each term in the sum we note that the constants $C_k$ can all be chosen such that $C_k<\tau+\delta$ for all $k\in[K+1]$ and any $\delta>0$. Similarly, the random terms $A_{k,n}$, which converge in probability to positive constants $A_k$, can also be shown to be bounded away from zero uniformly in $k$. This yields that we need only consider the rate of divergence of the remaining terms. We write,
\be \ba 
(\log n)^{1-cp_{k-1}}&=(\log n)^{(2-c)(2/3)(3/4)^k}=\exp\Big(\frac{2}{3}(2-c)\exp(k\log(3/4)+\log\log\log n)\Big),\\
\eps_n^{2p_k}b_n\log n&\sim (\log n)^{(2-c)(3/4)^k}=\exp((2-c)\exp(k\log(3/4)+\log\log\log n)), 
\ea \ee 
where we recall that $\tau=1$ and thus the expression of $p_k$ is simplified. We note that both terms diverge with $n$ for each $k\in[K]$ by the choice of $K$ and since $\log(3/4)>-1$. Moreover, the latter term is dominant for every $k\in[K]$, so that each term in the sum in~\eqref{eq:largedevK} tends to zero in probability. An upper bound for the entire sum is established when setting $p_{k-1}=p_{K-1}, p_k=p_K$ and bounding (with high probability) $C_k<\tau+\delta$ and $A_{k,n}>\delta$, for some small $\delta>0$ uniformly in $k$. We then obtain the upper bound 
\be 
C\exp\big(\log K-\eta^2 \delta (\log\log n)^{3(2-c)/4}(1+o(1))\big),
\ee 
which converges to zero as $\log K$ is negligible compared to the double logarithmic term. Now, for the final term in~\eqref{eq:largedevK}, we write as before,
\be \ba
(\log n)^{1-cp_K}&=\exp((1-c/2)\exp(K\log (3/4)+\log\log\log n))\\ 
&\leq \exp((1-c/2)\log \log \log n),\quad \text{and}\\
(\log\log n)^2&=\exp(2\log\log\log n),
\ea\ee 
so that the latter term dominates the former, which yields the desired result.

What remains is to show that 
\be \label{eq:2ndordupper}
\Pf{\max_{\inn}\zni\leq \max_{\inn}\Ef{}{\zni}-\eta a_n\log n\log\log n}\toinp 0
\ee 
holds for any $\tau\in(0,1]$ and $\eta>0$. We note that the event in the brackets occurs when
\be 
\zni\leq \max_{\inn}\Ef{}{\zni}-\eta a_n\log n\log\log n \quad \forall \inn,
\ee 
so that we obtain an upper bound if the inequality holds for $i=\wt I_n$ (recall that $\wt I_n:=\inf\{\inn:\Ef{}{\zni}\geq \Ef{}{\Zm_n(j)}\text{ for all }j\in[n]\}$). Hence,  
\be \ba 
\mathbb{P}_\F{}&\Big(\!\max_{\inn}\zni\leq \max_{\inn}\Ef{}{\zni}-\eta a_n\log n\log\log n\Big)\\
&\leq \mathbb{P}_\F\!\Big(\!\Zm_n(\wt I_n)\leq \max_{\inn}\Ef{}{\Zm_n(i)}-\eta a_n \log n\log\log n\Big).
\ea \ee 
Since $\wt I_n$ is determined by $\F_1,\ldots,\F_n$, it follows that $\max_{\inn}\Ef{}{\zni}=\Ef{}{\Zm_n(\wt I_n)}$. Thus, 
\be \ba\label{eq:concin}
\mathbb{P}_\F{}&\Big(\Zm_n(\wt I_n)\leq \max_{\inn}\Ef{}{\Zm_n(i)}-\eta a_n \log n\log\log n\Big)\\
&\leq \Pf{\big|\Zm_n(\wt I_n)-\Ef{}{\Zm_n(\wt I_n)}\big|\geq \eta a_n\log n\log\log n}\\
&\leq \frac{\Var(\Zm_n(\wt I_n)\,|\, (W_i)_{i\in\N})}{(\eta a_n\log n\log\log n)^2}\\
&\leq \frac{\max_{\inn}\Ef{}{\zni}}{b_n\log n}\frac{b_n}{\eta^2 a_n^2 \log n(\log\log n)^2},
\ea \ee 
which converges to zero almost surely, as the first fraction on the right-hand side converges to a positive constant almost surely and the second fraction converges to zero, since $\tau\in(0,1]$. Therefore,~\eqref{eq:2ndordupper} holds and combining this with~\eqref{eq:2ndordlower} and the dominated convergence theorem concludes the proof of~\eqref{eq:gumb2ndconc}.

Finally,~\eqref{eq:betaconc} can be proved in a similar way as~\eqref{eq:gumb2ndconc}, though the case $\tau=1$ no longer holds due to the absence of the $\log\log n$ term.
\end{proof}

\section{Proof of the main theorems}\label{sec:mainproof}

We now prove the main results, Theorems~\ref{thrm:maxbdd},~\ref{thrm:maxfrechet}, \ref{thrm:maxgumb},~\ref{Thrm:maxsecond} and~\ref{Thrm:maxsecond2}. We split the proof of Theorem~\ref{thrm:maxbdd} into three parts, to aid the reader. In all cases, the proof also holds for the model with a \emph{random out-degree} as discussed in Remark~\ref{remark:def}$(ii)$ by setting $m=1$, as in this model the in-degree can still be written as a sum of indicators. 

Before we prove Theorem~\ref{thrm:maxbdd}, we state an adaptation of~\cite[Lemma $1$]{DevLu95}:

\begin{lemma}\label{lemma:unionprob}
Let $A_{n,i}:=\{\zni\geq a_n\}$ for some sequence $(a_n)_{n\in\N}$. Then,
\be \ba 
\Pf{\bigcup_{i=1}^n A_{n,i}}&\leq \sum_{i=1}^n \Pf{A_{n,i}},\qquad \Pf{\bigcup_{i=1}^n A_{n,i}}\geq \frac{\sum_{i=1}^n \Pf{A_{n,i}}}{1+\sum_{i=1}^n \Pf{A_{n,i}}},
\ea \ee 
and as a result,
\be \ba 
\Pf{\bigcup_{i=1}^n A_{n,i}}\overset{\mathbb{P}/\mathrm{a.s.}}{\longrightarrow}\begin{cases}
	0, \mbox{ if} & \sum_{i=1}^n \Pf{A_{n,i}}\overset{\mathbb{P}/\mathrm{a.s.}}{\longrightarrow} 0,\\
	1, \mbox{ if} & \sum_{i=1}^n \Pf{A_{n,i}}\overset{\mathbb P /\mathrm{a.s.}}{\longrightarrow} \infty.
\end{cases}
\ea \ee 
\end{lemma} 

\begin{proof}
The result directly follows by applying~\cite[Lemma $1$]{DevLu95} to the conditional probability measure $\mathbb{P}_W$.
\end{proof}

{\cb Before we prove Theorem~\ref{thrm:maxbdd}, we obtain a  weaker result in the following proposition which implies convergence in probability of the rescaled maximum degree. The main result follows by {\cb strengthening the estimates from the  proof of the proposition and to combine them with a Borel-Cantelli argument} to obtain almost surely convergence.

\begin{proposition}\label{prop:bddub}
	Consider the WRG model as in Definition~\ref{def:wrg} and assume the vertex-weights satisfy the~\ref{ass:weightbdd} case of Assumption~\ref{ass:weights}. Recall that $\theta_m:=1+\E W/m$. For any $c>1/\log \theta_m$, 
	\be 
	\lim_{n\to\infty}\P{\max_{\inn}\zni \leq c\log n}=1.
	\ee 
	Also, for any $c<1/\log \theta_m$, 
	\be 
	\lim_{n\to\infty}\P{\max_{\inn}\zni \geq c\log n}=1.
	\ee 
\end{proposition}

We start by proving this propositions, as the final proof of Theorem~\ref{thrm:maxbdd} relies on several bounds developed in this proof. Which we split the proof into two parts, where we deal with the upper and lower bound separately. The proofs of the proposition and the final proof all rely on the proof of~\cite[Theorem 2]{DevLu95}, which we adapt to the setting of WRGs.
}

\begin{proof}[Proof of Proposition~\ref{prop:bddub}, upper bound]
We set $a_n:= c\log n $, with $c>1/\log \theta_m$ and let $\eps\in(0,\min\{m/\E{\F}-c+c\log(c\E{\F}/m),c\E{\F}/(m\e^2),1/2\})$. Note that the first argument of the minimum equals zero when $c=m/\E{W}$ and is positive otherwise. As $\theta_m = 1 + \E W/m$ and  $c>1/\log \theta_m>m/\E{W}$, this minimum is strictly positive. We aim to show that
\be \label{eq:condub}
\sum_{i=1}^n \Pf{\zni\geq a_n}\toas 0,
\ee 
which implies via Lemma~\ref{lemma:unionprob} and the dominated convergence theorem that
\be \label{eq:ub}
\mathbb{P}\Big(\max_{\inn}\zni\geq a_n\Big)\to 0.
\ee 
Using a Chernoff bound and the fact that $\zni$ is a sum of independent indicator random variables, we have for any $t >0$,
\be \label{eq:chernoff}
\Pf{\zni\geq a_n}\leq \mathrm{e}^{-ta_n}\prod_{j=i}^{n-1}\Big(\frac{\F_i}{S_j}\e^t+\Big(1-\frac{\F_i}{S_j}\Big)\Big)^m\!\!\leq \e^{-ta_n+(\e^t-1)m\F_i(H_n-H_i)},
\ee 
where $H_n:=\sum_{j=1}^{n-1}1/S_j$. This expression is minimised for $t=\log(a_n)-\log(m\F_i(H_n-H_i))$, which yields the upper bound
\be \label{eq:uibound}
\Pf{\zni\geq a_n}\leq \e^{a_n(1-u_i+\log u_i)},
\ee 
with $u_i=m\F_i(H_n-H_i)/a_n$. We note that the mapping $x\mapsto 1-x+\log x$ is increasing for $x\in(0,1)$. Moreover, by~\eqref{eq:sumfitnessconv}, $mH_n/a_n<1$ holds almost surely for all sufficiently large $n$ by the choice of $c$. Then, as we can bound $W_i$ from above by $1$ almost surely and $(H_n-H_i)$ is decreasing in $i$, we find, almost surely, for $n$ large and uniformly in $i$,
\be \ba
\Pf{\zni\geq a_n}&\leq \exp(a_n(1-mH_n/a_n+\log(mH_n/a_n))\\
&=\exp(c\log n(1-m/(c\E{\F})+\log(m/(c\E{\F})))(1+o(1)))\\
&=\exp(-\log n (m/\E{\F}-c+c\log(c\E{\F}/m))(1+o(1))).
\ea \ee 
Thus,
\be \label{eq:nepsbound}
\sum_{i<n^\eps}\Pf{\zni\geq a_n}\leq \exp(-\log n (m/\E{\F}-c+c\log(c\E{\F}/m)-\eps)(1+o(1))),
\ee 
which tends to zero almost surely as $\eps<m/\E{\F}-c+c\log(c\E{\F}/m)$.  Similarly, again using that $W_i\leq 1,mH_n/a_n<1$ almost surely for $n$ large, 
\be
\sum_{i>n^{1-\eps}}\!\!\!\Pf{\zni\geq a_n}\leq n\exp\Big(a_n\Big (1-\frac{m(H_n-H_{\lceil n^{1-\eps}\rceil})}{a_n}+\log\Big(\frac{m(H_n-H_{\lceil n^{1-\eps}\rceil})}{a_n}\Big)\Big).
\ee 
As $H_n-H_{\lceil n^{1-\eps}\rceil}=\eps\log n(1+o(1))$ almost surely for $n$ large, 
\be \ba \label{eq:n1minepsbound}
\sum_{i\geq n^{1-\eps}}\!\!\!\Pf{\zni\geq a_n}&\leq n\exp\Big(c\log n \Big(1-\frac{\eps m}{c\E{\F}}+\log\Big(\frac{\eps m}{c\E{\F}}\Big)\Big)(1+o(1))\Big)\\
&= n^{-(-c+\eps m/\E{\F}-c\log(\eps m/(c\E{\F}))-1)(1+o(1))},
\ea \ee 
which also tends to zero almost surely since $\eps<c\E{\F}/(m\e^2)$. It thus remains to prove that 
\be \label{eq:middle}
\sum_{n^\eps\leq i<n^{1-\eps}}\Pf{\zni\geq a_n}\toas 0.
\ee 
{\cb We use the same bound as in~\eqref{eq:uibound}, which holds uniformly in $\inn$ and we recall that $u_i=mWi(H_n-H_i)/(c\log n)$. In fact, we bound~\eqref{eq:uibound} from above further by using that $u_i\leq m(H_n-H_i)/(c\log n)=:\wt u_i$ almost surely. Define $u:\R\to \R$ by $u(x):=m(1-\log x/\log n)/(c\E W)$ and $\phi:\R\to\R$ by $\phi(x):=1-x+\log x$. Then, for $n^\eps\leq i< n^{1-\eps}$ such that $i=n^{\beta+o(1)}$ for some $\beta\in[\eps,1-\eps]$ (where the $o(1)$ is independent of $\beta$) and $x\in[i,i+1)$, 
	\be\ba\label{eq:phidiff}
	|\phi(\wt u_i)-& \phi(u(x))| \leq |\wt u_i-u(x)|+|\log(\wt u_i/u(x))|\\
	={}&\bigg|\frac{m}{c\E W}\Big(1-\frac{\log x}{\log n}\Big)-\frac{m}{c\log n}\sum_{j=i}^{n-1}\frac{1}{S_j}\bigg|+\bigg|\log\bigg(\frac{\E W}{\log n-\log x}\sum_{j=i}^{n-1}\frac{1}{S_j}\bigg)\bigg|.
	\ea\ee 
	By~\eqref{eq:sumfitnessconv} and since $i$ diverges with $n$, $\sum_{j=i}^{n-1}1/S_j-\log(n/i)/\E W=o(1)$ almost surely as $n\to\infty$. Applying this to the right-hand side of~\eqref{eq:phidiff} yields
	\be 
	|\phi(\wt u_i)-\phi(u(x))|\leq \frac{m}{c\E W}\Big|\frac{\log x-\log i}{\log n}\Big|+\Big|\log\Big(1+\frac{\log x-\log i+o(1)}{\log n-\log x}\Big)\Big|.
	\ee 
	Since $x\geq i\geq n^\eps$ and $|x-i|\leq 1$, we thus obtain that, uniformly in $n^\eps\leq i< n^{1-\eps}$ and $x\in[i,i+1)$, $|\phi(\wt u_i)-\phi(u(x))|=o(1/(n^\eps\log n))$ almost surely as $n\to\infty$. Applying this to the left-hand side of~\eqref{eq:middle} together with~\eqref{eq:uibound} (using $\wt u_i$ rather than $u_i$), we can bound the sum from above by 
	\be\ba \label{eq:intbound}
	\sum_{n^\eps\leq i< n^{1-\eps}}\Pf{\zni\geq a_n}&\leq \sum_{n^\eps\leq i< n^{1-\eps}}\e^{a_n\phi(\wt u_i)}\\
	&\leq \sum_{n^\eps\leq i< n^{1-\eps}}\int_i^{i+1}\e^{a_n \phi(u(x))+a_n|\phi(\wt u_i)-\phi(u(x))|}\,\d x\\
	&\leq (1+o(1))\int_{n^{\eps} }^{ n^{1-\eps}+1}\e^{a_n\phi(u(x))}\,\d x.
	\ea \ee  
	Recall that $\theta_m=1+\E{\F}/m$ and set $\wt \theta_m:=1+m/\E{\F}$. Using the variable transformation $w=\wt \theta_m(\log n-\log x)$ and Stirling's formula in the last line yields
	\be \ba
	(1+{}&o(1))\frac{n^{1+c-c\log\theta_m} }{\wt \theta_m(c\log n)^{c\log n}}\int_{\eps\wt \theta_m\log n+o(1)}^{(1-\eps)\wt \theta_m\log n}\!\!\!w^{c\log n}\e^{-w}\,\d w\leq \frac{2n^{1+c-c\log\theta_m} }{\wt \theta_m(c\log n)^{c\log n}}\Gamma(1+c\log n)\\
	&\sim \frac{2 n^{1-c\log \theta_m}}{\wt\theta_m}\sqrt{2\pi c \log n},
	\ea \ee
}
which tends to zero by the choice of $c$. Hence, combining the above with~\eqref{eq:nepsbound} and~\eqref{eq:n1minepsbound} yields~\eqref{eq:condub} and hence~\eqref{eq:ub}.
\end{proof} 

\begin{proof}[Proof of Proposition~\ref{prop:bddub}, lower bound]
We set $a_n:=\lceil c\log n\rceil,b_n:=\lceil \delta \log n\rceil$, where we take $c\in(0,1/\log\theta_m)$ and $\delta \in(0,1/\log\theta_m-c)$. For $i\in \N$ fixed, we couple $\zni$ to a sequence of suitable random variables. Let $(P_j)_{j\geq 2}$ be independent Poisson random variables with mean $m\F_i/S_{j-1}, j\geq 2$. Then, we can couple $\zni$ to the $P_j$'s to obtain
\be \label{eq:lowerboundcoupling}
\zni\geq \sum_{j=i+1}^n P_j \ind_{\{P_j\leq 1\}}=\sum_{j=i+1}^n P_j -\sum_{j=i+1}^n P_j\ind_{\{P_j> 1\}}=:W_n(i)-Y_n(i).
\ee 
By Lemma~\ref{lemma:unionprob} and the inequality
\be 
\Pf{\zni\geq a_n}\geq \Pf{W_n(i)\geq a_n+b_n}-\Pf{Y_n(i)\geq b_n},
\ee
it follows that we are required to prove that, for some $\eps,\xi>0$ sufficiently small,
\be \label{eq:conds}
\sum_{\substack{n^\eps\leq i\leq n^{1-\eps}\\ W_i\geq \e^{-\xi}}}\Pf{W_n(i)\geq a_n+b_n}\toinp\infty,\qquad \sum_{n^\eps \leq i \leq n^{1-\eps}}\Pf{Y_n(i)\geq b_n}\toinp 0,
\ee
as $n$ tends to infinity, to obtain 
\be
\mathbb{P}_W\Big(\max_{\inn}\zni\geq a_n\Big)\toinp 1.
\ee 
Using the uniform integrability of the conditional probability measure then yields
\be 
\lim_{n\to\infty}\mathbb{P}\Big(\max_{\inn}\zni \geq a_n\Big)=1,
\ee 
which yields the desired result. It thus remains to prove the two claims in~\eqref{eq:conds}.

We first prove the first claim of~\eqref{eq:conds}. Note that $W_n(i)$ is a Poisson random variable with parameter $m\F_i\sum_{j=i}^{n-1}1/S_j$. We note that by the strong law of large numbers, for some $\eta\in(0,\e^{1/(c+\delta)}-\theta_m)$, 
\be\label{eq:llnbound}
m\F_i\sum_{j=i}^{n-1}1/S_j\geq m\F_i\sum_{j=i}^{n-1}1/(j(\E{\F}+\eta))\geq (m\F_i/(\E{\F}+\eta))\log(n/i)
\ee
for all $ n^\eps\leq i\leq n$ almost surely when $n$ is sufficiently large. We can thus, for $n$ large, conclude that $W_n(i)$ stochastically dominates $X_n(i)$, where $X_n(i)$ is a Poisson random variable with a parameter equal to the right-hand side of~\eqref{eq:llnbound}. Then, also using that $W_i\leq 1$ almost surely, it follows that for $i\geq n^\eps$,
\be \ba 
\Pf{W_n(i)\geq a_n+b_n}&\geq \Pf{X_n(i)\geq a_n+b_n}\\
&\geq \Pf{X_n(i)=a_n+b_n}\\
&\geq \Big(\frac{i}{n}\Big)^{m/(\E{\F}+\eta)}W_i^{a_n+b_n}\frac{((m/(\E{\F}+\eta))\log(n/i))^{a_n+b_n}}{(a_n+b_n)!}.
\ea\ee 
We now sum over all $\inn$ such that $n^\eps \leq i\leq n^{1-\eps}, W_i\geq \e^{-\xi}$ for some sufficiently small $\eps,\xi>0$. By the lower bound on the vertex-weight, we obtain the further lower bound
\be \ba\label{eq:tndef}
\sum_{\substack{n^\eps\leq i\leq n^{1-\eps}\\ W_i\geq \e^{-\xi}}}{}&\!\!\!\Pf{W_n(i)\geq a_n+b_n}\\
\geq{}& \!\!\sum_{n^{\eps}\leq i\leq n^{1-\eps}}\!\!\!\ind_{\{W_i\geq \e^{-\xi}\}}\e^{-\xi(a_n+b_n)}\Big(\frac{i}{n}\Big)^{m/(\E{\F}+\eta)}\frac{((m/(\E{\F}+\eta))\log(n/i))^{a_n+b_n}}{(a_n+b_n)!}\\
=:{}&T_n.
\ea\ee 
We now claim that $T_n\toinp \infty$. This follows from the fact that the mean of $T_n$ diverges, and that $T_n$ concentrates around the mean. We first show the former statement. Let $p=p(\xi)=\P{W\geq \e^{-\xi}}$. Note that, due to the fact that $x_0=\sup\{x\in\R: \P{W\leq x}<1\}=1$, $p>0$ for any $\xi>0$. Hence,
\be\label{eq:meantn}
\E{T_n}=p\e^{-\xi(a_n+b_n)}\Big(\frac{m}{\E{W}+\eta}\Big)^{a_n+b_n}\frac{1}{(a_n+b_n)!}\sum_{n^{\eps}\leq i\leq n^{1-\eps}}\!\!\!\Big(\frac{i}{n}\Big)^{m/(\E{\F}+\eta)}\log(n/i)^{a_n+b_n}.
\ee 
Then, in a similar way as in~\eqref{eq:intbound}, and applying a variable transformation $t=(1+m/(\E{W}+\eta))\log(n/x)$,
\be \ba 
\phantom{0}\!\! & \sum_{n^{\eps}\leq i\leq n^{1-\eps}}\!\!\!\Big(\frac{i}{n}\Big)^{m/(\E{\F}+\eta)}(\log(n/i))^{a_n+b_n}\\
& =(1+o(1)) \int_{n^\eps}^{n^{1-\eps}}\Big(\frac{x}{n}\Big)^{m/(\E{W}+\eta)}\log(n/x)^{a_n+b_n}\,\d x\\
& =(1+o(1))n(a_n+b_n)!\Big(1+\frac{m}{\E{W}+\eta}\Big)^{-(a_n+b_n+1)}\!\!\!\int_{\eps(1+m/(\E{W}+\eta))\log n}^{(1-\eps)(1+m/(\E{W}+\eta))}\frac{\e^{-t}t^{a_n+b_n}}{(a_n+b_n)!}\,\d t.
\ea \ee 
We now identify the integral as the probability of the event $\{Y_n\in (\eps(1+m/(\E W+\eta))\log n,(1-\eps)(1+m/(\E W+\eta))\log n)\}$, where $Y_n$ is a sum of $a_n+b_n+1$ rate one exponential random variables. Since $a_n+b_n=(1+o(1))(c+\delta)\log n$, it follows from the law of large numbers that this probability equals $1-o(1)$ when $c+\delta\in (\eps(1+m/(\E W+\eta)),(1-\eps)(1+m/(\E W+\eta)))$, which is the case for $\eps,\eta$ sufficiently small. Thus, combining the above with~\eqref{eq:meantn}, we arrive at
\be \ba\label{eq:exptn}
\E{T_n}&\sim p\e^{-\xi(a_n+b_n)}\Big(\frac{m}{\E{W}+\eta}\Big)^{a_n+b_n}\Big(1+\frac{m}{\E{W}+\eta}\Big)^{-(a_n+b_n)+1}n\\
&=p\frac{\E{W}+\eta}{m\theta_m+\eta}\exp(\log n(1-(1+o(1))(c+\delta)(\log(\theta_m+\eta/m)+\xi))).
\ea \ee 
By the choice of $c$ and $\delta$, the exponent is positive when $\eta$ and $\xi$ are sufficiently small. What remains is to show that $T_n$ concentrates around $\E{T_n}$. Using a Chebyshev bound yields for any $\zeta>0$ fixed,
\be \label{eq:chebytn}
\P{|T_n/\E{T_n}-1|\geq \zeta}\leq \frac{\Var(T_n)}{(\zeta\E{T_n})^2},
\ee 
so that the result follows if $\Var(T_n)=o(\E{T_n}^2)$. Since $T_n$ is a sum of weighted, independent Bernoulli random variables, we readily have 
\be \ba 
\Var(T_n)= \Big(\frac{m}{\E{W}+\eta}\Big)^{2(a_n+b_n)}\frac{p(1-p)\e^{-2\xi(a_n+b_n)}}{((a_n+b_n)!)^2}\!\!\!\!\!\sum_{n^{\eps}\leq i\leq n^{1-\eps}}\!\!\!\Big(\frac{i}{n}\Big)^{\frac{2m}{(\E{\F}+\eta)}}\log\Big(\frac{n}{i}\Big)^{2(a_n+b_n)}.
\ea \ee 
Again writing the sum as an integral over $x$ instead of $i$, and now using the variable transformation $t=(1+2m/(\E{W}+\eta))\log(n/x)$, we obtain that the sum equals
\be
(n+o(n))(2(a_n+b_n))!\Big(1+\frac{2m}{\E W+\eta}\Big)^{-2(a_n+b_n)-1}\!\!\!\int_{\eps(1+\frac{2m}{\E W+\eta})\log n}^{(1-\eps)(1+\frac{2m}{\E W+\eta})\log n}\!\!\frac{\e^{-t}t^{2(a_n+b_n)}}{(2(a_n+b_n))!}\,\d t.
\ee 
We again interpret the integral as the probability of the event $\{\wt Y_n\in (\eps(1+2m/(\E W+\eta))\log n,(1-\eps)(1+2m/(\E W+\eta))\log n) \}$, where $\wt Y_n$ is a sum of $2(a_n+b_n)$  rate one exponential random variables. Again, for $\eta$ and $\eps$ sufficiently small, this probability is $1-o(1)$ by the law of large numbers. Thus, we obtain,
\be 
\Var(T_n)=(n+o(n))\frac{p(1-p)(\E{W}+\eta)}{m(\theta_m+1)+\eta}\e^{-2\xi(a_n+b_n)} \frac{(2(a_n+b_n))!}{((a_n+b_n)!)^2}\Big(1+\theta_m+\frac{\eta}{m}\Big)^{-2(a_n+b_n)}.
\ee
Using Stirling's approximation for the factorial terms then yields
\be  
\Var(T_n)\sim \frac{p(1-p)(\E{W}+\eta)}{(m(\theta_m+1)+\eta)\sqrt{\pi(a_n+b_n)}}\\e^{\log n+2(a_n+b_n)(\log 2-\log(1+\theta_m+\eta/m)-\xi)}.
\ee  
Combining this with~\eqref{eq:exptn}, we find that
\be\ba \label{eq:vartn}
\frac{\Var(T_n)}{\E{T_n}^2}\leq  \frac{K}{\sqrt{a_n+b_n}}\exp\Big(\log n\Big(-1+2(c+\delta)(1+o(1))\log\Big(2\frac{\theta_m+\eta/m}{1+\theta_m+\eta/m}\Big)\Big)\Big),
\ea \ee 
where $K>0$ is a suitable constant. The exponential terms decays with $n$ when 
\be 
c+\delta < \Big(\log\Big(4\Big(\frac{\theta_m+\eta/m}{1+\theta_m+\eta/m}\Big)^2\Big)\Big)^{-1}, 
\ee 
which is satisfied for any $\theta_m\in(1,2]$ when $\eta$ is sufficiently small, since $c+\delta<1/\log \theta_m$ and 
\be 
\log \theta_m> \log\Big(4\Big(\frac{\theta_m+\eta/m}{1+\theta_m+\eta/m}\Big)^2\Big)
\ee 
holds for any $\theta_m\in(1,2]$ when $\eta$ is sufficiently small. Therefore, $T_n/\E{T_n}\toinp 1$, so that $T_n\toinp \infty$. This then implies the first statement in~\eqref{eq:conds}. 

We now prove the second statement of~\eqref{eq:conds}. For a Poisson random variable $P$ with mean $\lambda$, we find that 
\be \label{eq:lambdabound}
\E{P\ind_{\{P>1\}}}=\E{P}-\P{P=1}=\lambda\big(1-\e^{-\lambda}\big)\leq \lambda^2,
\ee 
and, for any $t\in\R$, it follows from~\cite[Page $8$]{DevLu95} that
\be \label{eq:mgfbound}
\E{\e^{t(P\ind_{\{P>1\}}-\E{P\ind_{\{P>1\}}})}}\leq \e^{\lambda^2\e^{2t}}.
\ee 
Now, since $Y_n(i)=\sum_{j=i+1}^{n}P_j\ind_{\{P_j>1\}}$, where the $P_j$'s are independent Poisson random variables with mean $mW_i/S_{j-1}$, using an upper bound inspired by~\eqref{eq:llnbound} and using~\eqref{eq:lambdabound}, we obtain that almost surely for all $n$ large and $i\geq n^\eps$,
\be
\Ef{}{Y_n(i)}\leq \frac{m^2}{(\E{\F}-\eta)^2}\sum_{j=i}^{n-1}1/j^2\leq \frac{m^2}{(\E{\F}-\eta)^2(i-1)}.
\ee 
Then, for $i\geq n^\eps$ and $n$ large enough so that $b_n(i-1)\geq 4(m/(\E{\F}-\eta))^2$, we write for any $t > 0$ using~\eqref{eq:mgfbound},
\be \ba
\Pf{Y_n(i)\geq b_n}&\leq \Pf{Y_n(i)-\Ef{}{Y_n(i)}\geq b_n/2}\\
&\leq \e^{-tb_n/2}\Ef{}{\e^{t(Y_n(i)-\Ef{}{Y_n(i)})}}\\
&\leq \exp(-tb_n/2 +\e^{2t}m^2/((\E{\F}-\eta)^2(i-1))). 
\ea \ee 
This upper bound is smallest for $t=\log(b_n(i-1)(\E{\F}-\eta)^2/(4m^2))/2$, which yields the upper bound
\be
\exp\Big(\frac{b_n}{4}\Big(1-\log\Big(\frac{b_n(i-1)(\E{\F}-\eta)^2}{4m^2}\Big)\Big)\Big)=\Big(\frac{4\e m^2}{b_n(\E{\F}-\eta)^2(i-1)}\Big)^{b_n/4}\leq n^{-\eps b_n/4},
\ee 
when $n$ is large enough such that $b_n\geq 8\e m^2/(\E{\F}-\eta)^2$ and $n^\eps\leq 2(n^\eps-1)$. It then follows that 
\be \label{eq:ybound}
\sum_{n^\eps\leq i\leq n^{1-\eps}}\Pf{Y_n(i)\geq b_n}\leq n^{1-\eps\delta \log n/4},
\ee 
which tends to zero with $n$ almost surely. This yields the second claim in~\eqref{eq:conds} and concludes the proof.
\end{proof}

As stated at the start of this section, Proposition~\ref{prop:bddub} implies that 
\be 
\max_{\inn}\zni/\log n\toinp 1/\log\theta_m.
\ee 
It thus remains to strengthen this to almost sure convergence to complete the proof of Theorem~\ref{thrm:maxbdd}.

\begin{proof}[Proof of Theorem~\ref{thrm:maxbdd}]
Let $k_n:=\lceil \theta_m^n\rceil$ and set $Z_n:=\max_{\inn}\zni$. Similar to~\cite{DevLu95}, we use the bounds
\be 
\inf_{N\leq n}\frac{Z_{k_n}}{(n+1)\log\theta_m}\leq \inf_{2^N\leq n}\frac{Z_n}{\log n}\leq \sup_{2^N\leq n}\frac{Z_n}{\log n}\leq \sup_{N\leq n}\frac{Z_{k_{n+1}}}{n\log\theta_m}.
\ee 
It thus follows that to prove the almost sure convergence of the rescaled maximum degree $Z_n/\log n$, it suffices to do so for the subsequence $Z_{k_n}/\log k_n$, for which we can obtain stronger bounds due to the fact that $k_n$ grows exponentially in $n$. To prove the almost sure convergence of $Z_{k_n}/\log n$ to $1/\log\theta_m$, it thus suffices to prove 
\be \label{eq:liminfsup}
\liminf_{n\to\infty}\frac{Z_{k_n}}{(n+1)\log\theta_m}\geq \frac{1}{\log\theta_m}, \qquad \limsup_{n\to\infty}\frac{Z_{k_{n+1}}}{n\log\theta_m}\leq \frac{1}{\log\theta_m},
\ee 
almost surely, which can be achieved with the bounds used to prove the convergence in probability in Proposition~\ref{prop:bddub}. Namely, for the upper bound, using~\eqref{eq:nepsbound},~\eqref{eq:n1minepsbound} and~\eqref{eq:intbound}, we obtain for any $c>1/\log\theta_m$, a sufficiently small $\xi>0$ and some large constant $C>0$,
\be
\sum_{i=1}^{k_{n+1}} \Pf{\Zm_{k_{n+1}}(i)/(n\log\theta_m)\geq c}\leq2\e^{-\xi n(1+o(1))}+(1+o(1))C\sqrt n\e^{-\xi n},
\ee 
which is summable, so it follows from the Borel-Cantelli lemma that the upper bound in~\eqref{eq:liminfsup} holds $\mathbb{P}_\F$-almost surely. To extend this to $\mathbb{P}$-almost surely, we write 
\be \label{eq:pas}
\P{\Zm_{k_n}(i)/(n\log \theta_m)\geq c}\leq \E{\Pf{\Zm_{k_n}(i)/(n\log \theta_m)\geq c}\,|\, C_{k_n}(\eta)}+\P{C^c_{k_n}(\eta)},
\ee 
where $C_n(\eta):=\{S_j\in(\E W-\eta,\E W+\eta)\text{ for all }j\geq n^\eps\}$ holds for all but finitely many $n$ almost surely. On the event $C_{k_n}(\eta)$, we can obtain a deterministic and summable bound as demonstrated above. Moreover, since the vertex-weights are bounded, we can obtain exponentially decaying bounds for $\P{C_{k_n}^c(\eta)}$, so that it is summable as well. Hence, the Borel-Cantelli lemma yields the required result.

Similarly, for the lower bound in~\eqref{eq:liminfsup}, we have for $c<1/\log\theta_m$ by Lemma~\ref{lemma:unionprob},
\be \ba \label{eq:asbound}
\mathbb{P}_W(Z_{k_n}/((n+1)\log\theta_m)<c)&\leq 1-\frac{\sum_{i=1}^{k_n}\Pf{\Zm_{k_n}(i)\geq c(n+1)\log\theta_m}}{1+\sum_{i=1}^{k_n}\Pf{\Zm_{k_n}(i)\geq c(n+1)\log\theta_m}}\\
&=\frac{1}{1+\sum_{i=1}^{k_n}\Pf{\Zm_{k_n}(i)\geq c(n+1)\log \theta_m}}.
\ea\ee 
We again bound the sum from below by
\be \ba
\sum_{i=1}^{k_n}\Pf{\Zm_{k_n}(i)\geq c(n+1)\log \theta_m}\geq{}& \sum_{\substack{k_n^\eps\leq i\leq k_n^{1-\eps}\\ W_i\geq \e^{-\xi}}}\Pf{W_{k_n}(i)\geq (c+\delta)(n+1)\log \theta_m}\\
&-\sum_{k_n^\eps\leq i\leq k_n^{1-\eps}}\Pf{Y_{k_n}(i)\geq \delta (n+1)\log \theta_m}.
\ea \ee 
First, by the bound in~\eqref{eq:ybound}, we find that 
\be 
1-\sum_{k_n^\eps\leq i\leq k_n^{1-\eps}}\Pf{Y_{k_n}(i)\geq \delta (n+1)\log \theta_m}\geq 0,
\ee 
almost surely for all $n$ large. Then, we bound the sum of tail probabilities of the $W_{k_n}(i)$ from below by $T_{k_n}$, where we recall the definition of $T_n$ from~\eqref{eq:tndef}. Combining~\eqref{eq:chebytn} and~\eqref{eq:vartn}, we find that $T_{k_n}\geq \E{T_{k_n}}/2$ almost surely for all $n$ large. Together with~\eqref{eq:exptn} and~\eqref{eq:asbound}, we thus obtain  
\be 
\mathbb{P}_W(Z_{k_n}/((n+1)\log\theta_m)<c)\leq C \e^{-n\log \theta_m(1-(1+o(1))(c+\delta)(\log(\theta_m+\eta/m)+\xi))},
\ee  
for some constant $C>0$, which is summable when $\eta$ and $\xi$ are sufficiently small, so that the lower bound holds for all but finitely many $n$ $\mathbb P_W$-almost surely by the Borel-Cantelli lemma. A similar argument as in~\eqref{eq:pas} allows us to extend this to $\mathbb P$-almost surely, which concludes the proof.
\end{proof}

\begin{proof}[Proof of Theorem~\ref{thrm:maxfrechet}]
The proof of the convergence of $\max_{\inn}\zni/u_n$ and \newline $\max_{\inn}\zni/n$ as in~\eqref{eq:wrrtppplimit} and~\eqref{eq:wrrtinfmeanppplimit}, respectively, follows directly from Proposition~\ref{prop:condmeanfrechet} combined with~\eqref{eq:wrrtconcii} and~\eqref{eq:wrrtinfmeanconcii} in Proposition~\ref{lemma:wrrtconcentration}. The distributional convergence of $I_n/n$ to $I_\alpha$ and $I$ as in~\eqref{eq:wrrtppplimit} and~\eqref{eq:wrrtinfmeanppplimit}, respectively, follows from the {\cb following argument. Recall the Poisson point process $\Pi$ in the statement of the theorem with intensity measure $\mu(\d t,\d x):=\d t\times (\alpha-1)x^{-\alpha}\d x$ and define for $0\leq a\leq b\leq 1$, 
	\be \ba
	Q_\ell(a)& :=\!\!\max_{(t,f)\in\Pi: t\in(0,a)}\!\! f\log(1/t), \quad Q(a,b) :=\!\!\max_{(t,f)\in\Pi: t\in(a,b)}\!\!f\log(1/t),\\  Q_r(b)& :=\!\!\max_{(t,f)\in\Pi:t\in(b,1)}\!\! f\log(1/t),
	\ea \ee 
	as well as the events 
	\be \ba
	M_n(a,b)&:=\Big\{\max_{an<i<bn}\zni/u_n>\big(\max_{1\leq i\leq an}\zni/u_n\vee \max_{bn\leq i\leq n}\zni/u_n)\Big\},\\
	M(a,b)&:=\Big\{Q(a,b)>Q_\ell(a)\vee Q_r(b)\Big\}.
	\ea\ee 	
	Then, from the start of the proof, we have that
	\be 
	\P{I_\alpha\in(a,b)}=\lim_{n\to\infty}\P{I_n/n\in(a,b)}=\lim_{n\to\infty}\P{M_n(a,b)}= \P{M(a,b)}.
	\ee 
	A similar approach can be used to show the convergence when $\alpha\in(1,2)$, as in~\eqref{eq:wrrtinfmeanppplimit}. By the independence property of Poisson point process, it follows that the right-hand side equals $g(a,b)/g(0,1)$, where $g(a,b):=\int_a^b \log(1/x)^{\alpha-1}\,\d x$.} Using the variable transformation $w=\log(1/x)$,
\be \ba
g(a,b)& =\int_{\log(1/b)}^{\log(1/a)}w^{\alpha-1}\e^{-w}\,\d w\\ & =\Gamma(\alpha)\P{W_\alpha\in(\log(1/b),\log(1/a))}=\Gamma(\alpha)\P{\e^{-W_\alpha}\in(a,b)}, 
\ea \ee 
where $W_\alpha$ is a $\Gamma(\alpha,1)$ random variable. Thus, 
\be 
\P{I_\alpha\leq t}=\frac{g(0,t)}{g(0,1)}=\P{\e^{-W_\alpha}\leq t},
\ee 
and $m\max_{(t,f)\in\Pi}f\log(1/t)$ follows a Fr\'echet distribution with shape parameter $\alpha-1$ and scale parameter $mg(0,1)^{1/(\alpha-1)}=m\Gamma(\alpha)^{1/(\alpha-1)}$. { \cb  The joint convergence of $I_n/n$ and $\max_{\inn}\zni/u_n$ follows from the fact that
	\be \ba
	\P{I_n/n\in(a,b),\max_{\inn}\zni/u_n\geq x}&=\P{M_n(a,b)\cap \{\max_{an<i<bn}\zni/u_n \geq x\}}\\
	&\to \P{M(a,b)\cap\{Q(a,b)\geq x\}}\\
	&=\P{Q(a,b)>Q_\ell(a)\vee Q_r(b)\vee x}.
	\ea \ee
	Again using the independence property of Poisson point processes, one can show that this equals the probability that $I_\alpha \in(a,b)$, times the probability that $m\max_{(t,f)\in\Pi}f\log(1/t)$ is at least $x$, which concludes the proof.}
	\end{proof}
	
	\begin{proof}[Proof of Theorem~\ref{thrm:maxgumb}]

We only discuss the case $m=1$, as the proof for $m>1$ follows analogously. Most of the proof directly follows by combining Propositions~\ref{prop:condmeantauinf},~\ref{prop:condmeantaufin}, and~\ref{prop:condmeantau0} with Proposition~\ref{lemma:wrrtconcentration},~\eqref{eq:concgumbab},~\eqref{eq:concrv}, and~\eqref{eq:concgumbc}. The one statement that remains to be proved is that $\log(I_n)/\log n$ converges, either in probability or almost surely, depending on the class of distributions the vertex-weight distribution is in, to some constant. Let us start with the~\ref{ass:weightgumbel}-\ref{ass:weighttaufin} sub-case, as in~\eqref{eq:gumbppplim}.

We recall the sequence $(\wt \eps_k)_{k\in\N_0}$ in~\eqref{eq:wtepsk} and note that $\wt \eps_k$ is decreasing and tends to $0$ with $k$. Now, fix $\eps>0$ and let $k$ be large enough such that $\wt \eps_k<\eps$. We are required to show that 
\be 
\sum_{n=1}^\infty \ind_{\{\log I_n/\log n\leq \gamma-\eps\}}<\infty,\ \mathrm{a.s.}
\ee 
First, we note that
\be\ba
\Big\{\frac{\log I_n}{\log n}\leq \gamma-\eps\Big\}& \subseteq \Big\{\frac{\log I_n}{\log n}\leq \gamma -\wt\eps_k\Big\}\\ & \hspace{1cm} \subseteq \Big\{\max_{n^{\gamma-\wt\eps_k}\leq i\leq n}\!\frac{\zni}{(1-\gamma)b_{n^\gamma} \log n}\leq \!\!\max_{i<n^{\gamma-\wt\eps_k}}\!\frac{\zni}{(1-\gamma)b_{n^\gamma}\log n}\Big\}.
\ea \ee 
Let us denote the event on the right-hand side by $A_n$ and, for $\eta>0$, define the event
\be 
C_n:=\Big\{\max_{\inn}|\zni-\Ef{}{\zni}|\leq \frac{\eta}{2}(1-\gamma)b_{n^\gamma}\log n\Big\}.
\ee 
We then have that $A_n\subseteq (A_n\cap C_n)\cup C_n^c$, so that 
\be
\sum_{n=1}^\infty \ind_{\{\log I_n/\log n\leq \gamma-\eps\}}\leq \sum_{n=1}^\infty \ind_{A_n}\leq \sum_{n=1}^\infty \ind_{A_n}\ind_{C_n}+\ind_{C_n^c}.
\ee 
We now use the concentration argument used to prove~\eqref{eq:concrv} in Proposition~\ref{lemma:wrrtconcentration}. It already follows from~\eqref{eq:concrv} that only finitely many of the $\ind_{C_n^c}$ equal $1$, so what remains is to show that, almost surely, only finitely many of the product of indicators equal $1$ as well. On $C_n$, 
\be 
A_n\cap C_n\subseteq \Big\{\max_{n^{\gamma-\wt\eps_k}\leq i\leq n}\frac{\Ef{}{\zni}}{(1-\gamma)b_{n^\gamma}\log n}\leq\max_{i<n^{\gamma-\wt\eps_k}}\frac{\Ef{}{\zni}}{(1-\gamma)b_{n^\gamma}\log n}+\eta\Big\}.
\ee 
The limsup of the second maximum in the event on the right-hand side is almost surely at most $c_k<1$, where $c_k$ is the quantity defined in~\eqref{eq:eps_limit}. Then, we can directly bound the first maximum from below by 
\be 
\max_{n^{\gamma-\wt\eps_k}\leq i\leq n}\frac{\Ef{}{\zni}}{(1-\gamma)b_{n^\gamma}\log n}\geq \max_{n^{\gamma-\wt \eps_k}\leq i\leq n^\gamma}\frac{\Ef{}{\zni}}{(1-\gamma)b_{n^\gamma}\log n}\geq \max_{n^{\gamma-\eps_k}\leq i\leq n^\gamma}\frac{\F_i}{b_{n^\gamma}}\frac{\sum_{i=n^\gamma}^{n-1}1/S_j}{(1-\gamma)\log n},
\ee 
and the lower bound converges almost surely to $1$ by~\eqref{eq:sumfitnessconv} and Lemma~\ref{lemma:asconvmax}. Hence, if we take some $\delta \in(0,1-c_k)$ and set $\eta<\delta/3$, then almost surely there exists an $N\in\N$ such that for all $n\geq N$,
\be 
\max_{i<n^{\gamma-\wt\eps_k}}\frac{\Ef{}{\zni}}{(1-\gamma)b_{n^\gamma}\log n}<c_k+\delta/3,\qquad \max_{n^{\gamma-\wt\eps_k}\leq i \leq n}\frac{\Ef{}{\zni}}{(1-\gamma)b_{n^\gamma}\log n}>1-\delta/3.
\ee 
It follows that that the event 
\be
\Big\{\max_{n^{\gamma-\wt\eps_k}\leq i\leq n}\frac{\Ef{}{\zni}}{(1-\gamma)b_{n^\gamma}\log n}\leq \max_{i<n^{\gamma-\wt\eps_k}}\frac{\Ef{}{\zni}}{(1-\gamma)b_{n^\gamma}\log n}+\eta\Big\}
\ee 
almost surely does not hold for all $n\geq N$. Thus, for any $\eps>0$ and for all $n$ large, $\log I_n/\log n\geq \gamma-\eps$ almost surely. With a similar approach, we can prove that $\log I_n/\log n\leq\gamma+\eps$, so that the almost sure convergence to $\gamma$ is established.

For the~\ref{ass:weightgumbel}-\ref{ass:weighttauinf} and~\ref{ass:weightgumbel}-\ref{ass:weighttau0} sub-cases, we intend to prove the convergence of $\log I_n/\log n$ to $0$ and $1$ in probability, respectively. We provide a proof for the former sub-case, and note that the proof for the latter sub-case follows in a similar way.

Let $\eps>0$. Then,
\be \label{eq:indexbound}
\P{\log I_n/\log n\geq \eps}=\P{I_n\geq n^\eps}=\P{\max_{n^\eps \leq i\leq n}\frac{\zni}{b_n\log n}> \max_{1\leq i<n^\eps}\frac{\zni}{b_n\log n}}.
\ee 
Again, we define the event, for $\eta\in(\eps/3)$ small, 
\be 
C_n:=\Big\{\max_{\inn}|\zni-\Ef{}{\zni}| \leq \eta b_n\log n /2\Big\}, 
\ee 
which holds with high probability due to~\eqref{eq:concgumbab}. We can then further bound the right-hand side in~\eqref{eq:indexbound} from above by
\be \ba\label{eq:svindex}
\mathbb{P}{}&\Big(\Big\{\max_{n^\eps \leq i\leq n}\frac{\zni}{b_n\log n}> \max_{1\leq i<n^\eps}\frac{\zni}{b_n\log n}\Big\}\cap C_n\Big)+\P{C_n^c}\\
&\leq \P{\max_{n^\eps \leq i\leq n}\frac{\Ef{}{\zni}}{b_n\log n}> \max_{1\leq i<n^\eps}\frac{\Ef{}{\zni}}{b_n\log n}-\eta}+\P{C_n^c}.
\ea \ee 
In a similar way as in the proof of Proposition~\ref{prop:condmeantauinf}, as well as in the proof above for the~\ref{ass:weightgumbel}-\ref{ass:weighttaufin} sub-case, we can bound the maximum of the conditional mean in-degrees on the left-hand side from above and the maximum on the right-hand side from below by random quantities that converge in probability to fixed constants. Namely, for $\eps > 0$ and $\beta\in(0,\eps/3)$,
\be \ba 
\max_{n^\eps \leq i\leq n}\frac{\Ef{}{\zni}}{b_n\log n}\leq \max_{n^\eps \leq i\leq n}\frac{\F_i \sum_{j=n^\eps}^{n-1}1/S_j}{b_n\log n}&\toinp 1-\eps,\\
\max_{1\leq i<n^\eps}\frac{\Ef{}{\zni}}{b_n\log n}-\eta\geq \max_{i\in[n^\beta]}\frac{\F_i\sum_{j=n^\beta}^{n-1}1/S_j}{b_n\log n}-\eta&\toinp 1-\beta-\eta,
\ea\ee
so that, with high probability, the first quantity is at most $1-5\eps/6$ and the second quantity is at least $1-(\beta+\eta)-\eps/6>1-5\eps/6$ by the choice of $\beta$ and $\eta$. It thus follows that the first probability on the right-hand side of~\eqref{eq:svindex} tends to zero with $n$, and so does the second probability ($C_n$ holds with high probability), so that the claim follows.
\end{proof}

To conclude, we prove Theorems~\ref{Thrm:maxsecond} and~\ref{Thrm:maxsecond2}:

\begin{proof}[Proof of Theorem~\ref{Thrm:maxsecond}]
We only discuss the case $m=1$, as the proof for $m>1$ follows analogously. Let us first deal with the results for the~\ref{ass:weightgumbel}-\ref{ass:weighttaufin} sub-case. The distributional convergence of the rescaled maximum degree to the correct limit, as in~~\eqref{eq:gumbppplim2ii}, and the convergence result in~\eqref{eq:nottight}, directly follow by combining Proposition~\ref{prop:condmeantaufin} with Proposition~\ref{lemma:wrrtconcentration},~\eqref{eq:gumb2ndconc} and~\eqref{eq:betaconc}. The one thing that remains to be proved is: $I_n(\gamma,s,t,\ell)/(\ell(n)n^\gamma)$ converges in distribution, jointly with the maximum degree of vertices $i\in C_n(\gamma,s,t,\ell)$, as in~\eqref{eq:gumbppplim2ii}.

The distributional convergence of $I_n(\gamma,s,t,\ell)/(\ell(n)n^\gamma)$, as well as its joint convergence with the maximum degree of vertices in $C_n(\gamma,s,t,\ell)$, follows from the same argument as in the proof Theorem~\ref{thrm:maxfrechet}, where now
\be \label{eq:g}
g(a,b):=\log(b/a).
\ee
For the distribution of the distributional limit $I_\gamma$ of $I_n(\gamma,s,t,\ell)/(\ell(n)n^\gamma)$, with $x\in(s,t)$,
\be 
\P{I_\gamma\in(s,x)}=\frac{g(s,x)}{g(s,t)}=\frac{\log(x/s)}{\log(t/s)}=\P{\e^U\in(s,x)},
\ee 
where $U\sim\mathrm{Unif}(\log s,\log t)$. Finally, for any $x\in\R$,
\be \ba
\mathbb{P}\Big(\max_{\substack{(v,w)\in\Pi\\ v\in(s,t)}}w-\log v\leq x\Big)=\exp\Big(-\int_s^t\int_{x+\log v}^\infty \e^{-w}\,\d w \d v\Big)=\exp(-\e^{-(x-\log\log(t/s))}),
\ea \ee
which proves that the distributional limits as described in~\eqref{eq:gumbppplim2ii} and~\eqref{eq:rav2ndorder} have the desired distributions. 
\end{proof}

\begin{proof}[Proof of Theorem~\ref{Thrm:maxsecond2}]
The proof directly follows by combining Proposition~\ref{prop:condmeantaufin2ndcomp} with Proposition~\ref{prop:conc2nd},~\eqref{eq:gumb2ndconc}, and Proposition~\ref{prop:condmeantau02nd} with Proposition~\ref{lemma:wrrtconcentration},~\eqref{eq:concgumbc}.
\end{proof} 

\paragraph{\textbf{Acknowledgements.}} 

 The authors would like to thank the referees for a careful reading of the manuscript. Their suggestions helped to greatly improve the presentation of the article.
 
\bibliographystyle{abbrv}
\bibliography{wrtbib}	
\end{document}